\documentclass[12pt, a4paper]{amsart}
\usepackage{amsmath}
\usepackage{geometry,amsthm,graphics,tabularx,amssymb,shapepar}
\usepackage{verbatim}
\usepackage{enumitem}
\usepackage{amscd}
\usepackage{accents}
\usepackage{bm}
\usepackage[all]{xypic}

\usepackage{hyperref}

\allowdisplaybreaks

\newcommand{\C}{\mathbb{C}}

\newcommand{\N}{\mathbb N}
\newcommand{\Z}{\mathbb Z}

\newcommand{\ba}{\begin {eqnarray}}
\newcommand{\ea}{\end {eqnarray}}
\newcommand{\baa}{\begin {eqnarray*}}
\newcommand{\eaa}{\end {eqnarray*}}
\newcommand{\be}{\begin {equation}}
\newcommand{\ee}{\end {equation}}
\newcommand{\bee}{\begin {equation*}}
\newcommand{\eee}{\end {equation*}}

\newlength{\dhatheight}
\newcommand{\doublehat}[1]{%
    \settoheight{\dhatheight}{\ensuremath{\widehat{#1}}}%
    \addtolength{\dhatheight}{-0.35ex}%
    \widehat{\vphantom{\rule{1pt}{\dhatheight}}%
    \smash{\widehat{#1}}}}
\theoremstyle{Theorem}

\theoremstyle{Theorem}

\newtheorem{thm}{Theorem}[section]

\newtheorem{lemt}[thm]{Lemma}
\newtheorem{prpt}[thm]{Proposition}
\newtheorem{thmt}[thm]{Theorem}

\theoremstyle{Theorem}

\theoremstyle{Theorem}

\theoremstyle{Plain}

\theoremstyle{Definition}

\newtheorem{dfnt}[thm]{Definition}

\def\({\left(}

\def\){\right)}

\def \<{{\langle}}
\def \>{{\rangle}}

\numberwithin{equation}{section}
\title{Quantum vertex algebra associated to quantum toroidal $\mathfrak{gl}_N$}

\author{Fulin Chen$^1$}
\address{1. School of Mathematics and Statistics, Qinghai Nationalities University, Xining, China 810000; 2. School of Mathematical Sciences, Xiamen University,
 Xiamen, China 361005}

\email{chenf@xmu.edu.cn}
\thanks{$^1$Partially supported by China NSF grant (No. 12161141001) and the Fundamental Research Funds for the Central Universities (No. 20720230020).}

\author{Xin Huang}
\address{School of Mathematical Sciences, Xiamen University,
 Xiamen, China 361005}
 \email{xinhuang2@stu.xmu.edu.cn}

\author{Fei Kong$^2$}
\address{Key Laboratory of Computing and Stochastic Mathematics (Ministry of Education), School of Mathematics and Statistics, Hunan Normal University, Changsha, China 410081}
\email{kongmath@hunnu.edu.cn}
\thanks{$^2$Partially supported by China NSF grant (No. 12371027).}

 \author{Shaobin Tan$^3$}
  \address{School of Mathematical Sciences, Xiamen University,
 Xiamen, China 361005}
  \email{tans@xmu.edu.cn}
  \thanks{$^3$Partially supported by China NSF grant (No. 12131018).}

\begin{document}
\maketitle
\begin{abstract}In this paper, we associate the quantum toroidal algebra $\mathcal{E}_N$ of type $\mathfrak{gl}_N$ with quantum vertex algebra through equivariant $\phi$-coordinated quasi modules.
More precisely, for every $\ell\in \mathbb{C}$, by deforming the universal affine vertex algebra  of $\mathfrak{sl}_\infty$, we construct an $\hbar$-adic quantum $\Z$-vertex algebra $V_{\widehat{\mathfrak{sl}}_{\infty},\hbar}(\ell,0)$.  Then we prove that the category of restricted $\mathcal{E}_N$-modules of level $\ell$ is canonically isomorphic to that of
equivariant $\phi$-coordinated quasi $V_{\widehat{\mathfrak{sl}}_{\infty},\hbar}(\ell,0)$-modules.
\end{abstract}

\section{Introduction}

In this paper, we explore and establish natural connections between representations of the quantum toroidal algebra of type $\mathfrak{gl}_N$ ($N\ge 3$) and the quantum vertex algebra associated to affine Lie algebra of $\mathfrak{sl}_\infty$.

\subsection{Quantum toroidal algebras}
Let $\mathfrak{g}$ be a complex reductive Lie algebra. Quantum toroidal algebra of type $\mathfrak{g}$ introduced by Ginzburg-Kapranov-Vasserot (\cite{GKV}) is a  quantum deformation of  the universal enveloping algebra of central extension of the double loop algebra $\mathfrak{g}\otimes \C[x^{\pm 1}, y^{\pm 1}]$.
The case $\mathfrak{g}=\mathfrak{gl}_N$ is  particularly interesting:  one can construct a two-parameter quantum toroidal algebra  by adding an  additional quantum parameter $p$. The classical limit of this quantum toroidal $\mathfrak{gl}_N$ is the universal enveloping algebra of central extension of the matrix Lie algebra $\mathfrak{gl}_N\otimes \C_p[x^{\pm 1}, y^{\pm 1}]$ (\cite{VV2}), where $\C_p[x^{\pm 1}, y^{\pm 1}]$ denotes the quantum $2$-torus satisfying that $yx=pxy$.

Quantum toroidal $\mathfrak{gl}_N$  has been constructed by many different approaches in the literature:  the $(q,\gamma)$-analog of $W$-algebra $\mathcal{W}_{1+\infty}$ (\cite{M3}),
the spherical double affine Hecke algebra (\cite{SV}),
the elliptic Hall algebra (\cite{BS}),
and the quantum shuffle algebra (\cite{FT,N,T}).
These diverse realizations make the theory of quantum toroidal $\mathfrak{gl}_N$ quite interesting and impressive.
On the other hand,  quantum toroidal $\mathfrak{gl}_N$ serves as a unifying algebraic framework that bridges different areas in physics, including various solvable structures of 2d conformal field theory, supersymmetric gauge
theory, and string theory (see \cite{MNNZ} and the references therein).

It is notable that
the representation theory of quantum toroidal $\mathfrak{gl}_N$ is very rich and promising:
 the  Schur duality with double affine Hecke algebra (\cite{VV1}), the toroidal analogue of $(\mathfrak{gl}_n,\mathfrak{gl}_m)$ duality (\cite{FJM}),  the study of  Fock representations (\cite{FFJMM,FJMM1,FJMM2,FJMM3, M1,M2,TU,STU,VV2}),  the vertex representation constructions (see \cite{GJ, S}),  the connection with Macdonald polynomials (\cite{AKOS,SKAO}, the connection with Mckay correspondence (\cite{FJW}),  the branching rules (\cite{FJMM4}), and etc.
In this paper we study the representation theory of quantum toroidal  $\mathfrak{gl}_N$ from the point of view of the association with quantum vertex algebras.

\subsection{Quantum vertex algebras}
In the general field of vertex algebras, one conceptual problem is to develop suitable quantum vertex algebra theories and associate quantum vertex algebras to certain quantum affine algebras.
As one of the fundamental works, Etingof and Kazhdan (\cite{EK}) developed a theory of quantum vertex operator algebras
in the sense of formal deformations of vertex algebras, where the usual locality is replaced by the $S$-locality, governed by a rational quantum Yang-Baxter operator.
Partly motivated by the work of Etingof and Kazhdan, H. Li
conducted a series of studies.
While vertex algebras are analogues of commutative associative algebras, H. Li  (\cite{Li1, Li2}) introduced the notion of
nonlocal vertex algebras, which are analogues of noncommutative associative algebras. A nonlocal vertex algebra becomes a weak quantum vertex algebra  if it satisfies the $S$-locality, and a quantum vertex algebra  when the $S$-locality is controlled by a rational quantum Yang-Baxter operator.
In order to associate quantum vertex algebras to quantum affine algebras, a theory of $\phi$-coordinated quasi modules was developed in \cite{Li5, Li6}.
The $\hbar$-adic versions of these notions were introduced in \cite{Li4}.
In this framework, a quantum vertex operator algebra in sense of Etingof-Kazhdan is an $\hbar$-adic quantum vertex algebra whose classical limit is a vertex algebra.

In the pioneer work \cite{EK}, Etingof and Kazhdan constructed quantum vertex operator algebras
 as formal deformations of type $A$ universal affine vertex algebras
by using the $R$-matrix type relations from \cite{RS-RTT}.
Later, Butorac, Jing and Ko\v{z}i\'{c}  (see \cite{BJK-qva-BCD}) extended Etingof-Kazhdan's construction to
types $B$, $C$ and $D$ rational $R$-matrices.
The modules for these $\hbar$-adic quantum vertex algebras
are in one-to-one correspondence with restricted modules for the corresponding Yangian doubles (\cite{K-qva-phi-mod-BCD}).
Based on the $R$-matrix presentation of quantum affine algebras of classical types
(\cite{DF-qaff-RTT-Dr}, \cite{JLM-qaff-RTT-Dr-BD}, \cite{JLM-qaff-RTT-Dr-C}),
Ko\v{z}i\'{c} (\cite{Kozic-qva-tri-A, K-qva-phi-mod-BCD}) constructed certain quantum vertex  algebras
by using trigonometric $R$-matrices and established a one-to-one correspondence
between their $\phi$-coordinated modules
and restricted modules for the quantum affine algebras.

In \cite{JKLT-Defom-va}, the authors developed a method for constructing quantum vertex  algebras by using vertex bialgebras,
and constructed a family of quantum lattice vertex algebras, which are formal deformations of lattice vertex algebras.
Based on a reformulation of the Drinfeld type Serre relations, the third author  constructed  certain affine quantum vertex  algebras in \cite{K}, and proved that quantum affine algebras can be associated with these affine quantum vertex algebras in a way similar to that affine Lie algebras are associated with affine vertex algebras.
The main goal of this paper is to generalize such an association on quantum affine algebras to the quantum toroidal $\mathfrak{gl}_N$ setting.

\subsection{The outline of the paper and the main result}
From now on, let $N$ be a positive integer, $p$ a nonzero complex number, and $\ell$ a complex number.
In this paper, we will always assume that
\[
N\ge 3\ \text{and}\ p\ \text{is not a root of unity}.
\]
In Section 2, we introduce the notion of  quantum toroidal algebra $\mathcal{E}_N=\mathcal{E}_N(\hbar,p)$ of type $\mathfrak{gl}_N$ over the commutative ring $\C[[\hbar]]$, as well as certain related $\C[[\hbar]]$-algebras. We also introduce a notion of
restricted $\mathcal{E}_N$-module of level $\ell$.

Note that the algebra $\mathcal{E}_N$ cannot directly associated to $\phi$-coordinated modules for an $\hbar$-adic quantum vertex algebra, due to the existence of the additional parameter ``$p$".
 Then we come to a
theory of what were called $(G,\chi)$-equivariant $\phi$-coordinated quasi module for an $\hbar$-adic quantum $G$-vertex algebra (\cite{Li6,JKLT1}), where $G$ is a group and $\chi$ is a linear character of $G$.
By definition, an $\hbar$-adic quantum $G$-vertex algebra is an  $\hbar$-adic quantum vertex algebra $V$ together with a group homomorphism $G\rightarrow \mathrm{Aut}(V)$, while a $(G,\chi)$-equivariant $\phi$-coordinated quasi $V$-module is a $\phi$-coordinated quasi $V$-module together with a ``$(G,\chi)$-equivariant" action.

In Section 3, we first recall the notion of $\hbar$-adic quantum $G$-vertex algebra and then construct an $\hbar$-adic quantum $\Z$-vertex algebra $V_{\widehat{\mathfrak{sl}}_\infty,\hbar}(\ell,0)$. We also show that $V_{\widehat{\mathfrak{sl}}_\infty,\hbar}(\ell,0)$ is a quantum deformation of  $V_{\widehat{\mathfrak{sl}}_\infty}(\ell,0)$ by presenting a new realization of  $V_{\widehat{\mathfrak{sl}}_\infty}(\ell,0)$. Here, the notation $V_{\widehat{\mathfrak{sl}}_\infty}(\ell,0)$ stands for the universal affine vertex algebra  associated to $\mathfrak{sl}_\infty$ (\cite{LL}).

In Section 4, we start with the notion of $(G,\chi)$-equivariant $\phi$-coordinated quasi module for an $\hbar$-adic quantum $G$-vertex algebra. Then we prove the main result of this paper (see Theorem \ref{thm:main}):
the category of $(\Z,\chi_{p^N})$-equivariant $\phi$-coordinated quasi $V_{\widehat{\mathfrak{sl}}_\infty,\hbar}(\ell,0)$-modules is isomorphic to the category of restricted $\mathcal{E}_N$-modules of level $\ell$, where $\chi_{p^N}$ is the character of $\Z$ given by $\chi_{p^N}(n)=p^{nN}$ for $n\in \Z$.
Our main method is to construct a ``middle" category $\mathcal{R}_\ell^\phi$ which bridges the above two categories (see Propositions \ref{pr:vglRphi} and \ref{pr:rescate}).

In this paper we denote by $\mathbb{Z}$,  $\mathbb{N}$, $\mathbb{Z}_+$, $\mathbb{C}$
and $\mathbb{C}^{\times}$ respectively the sets  of integers, non-negative
integers, positive integers, complex numbers and nonzero complex numbers. For every $m\in\Z_+$, we denote by $S_m$ the permutation group on $m$ letters.
And, for a Lie algebra $\mathfrak{t}$, we denote by $\mathcal{U}(\mathfrak{t})$ the universal enveloping algebra of $\mathfrak{t}$.

Let $\hbar,z,w,z_{0},z_{1},z_{2},\dots$ be mutually
commuting independent formal variables, and let $\delta(z)=\sum_{n\in\Z}z^n$ be the usual delta function.
We use the standard  notations and conventions as in
\cite{LL}.
For example, for a vector space $U$, $U[[z_1,z_2,\dots,z_r]]$ is the space of formal (possibly doubly
infinite) power series in $z_1,z_2,\dots,z_r$ with coefficients in $U$, and
$U((z_1,z_2,\dots,z_r))$  is the space of lower truncated Laurent power series in
$z_1,z_2,\dots,z_r$ with coefficients in $U$.

\section{Quantum toroidal algebras}
In this section, we  first recall the notion of quantum toroidal algebra $\mathcal{E}_N$ of type $\mathfrak{gl}_N$ over the commutative ring $\C[[\hbar]]$, and then introduce a notion of restricted $\mathcal{E}_N$-module of level $\ell$.

\subsection{Quantum toroidal algebras}

 We start with some notations.
 Throughout this paper, set
 \[q=\textrm{exp }\hbar\in\C[[\hbar]].\] As in the Introduction, let $p$ be a generic complex number, and let $N\ge 3$ be a positive integer. Write
 \begin{align}
 B=(b_{i,j})_{0\leq i,j\leq N-1}=
 \begin{pmatrix}
 2&-1&\cdots&0&-1\\
 -1&2&\cdots&0&0\\
 \vdots&\vdots&\ddots&\vdots&\vdots\\
 0&0&\cdots&2&-1\\
-1&0& \cdots&-1&2
 \end{pmatrix}
\end{align} for the generalized Cartan matrix of type $A_{N-1}^{(1)}$, and write
\begin{align}\label{M}
M=(m_{i,j})_{0\leq i,j\leq N-1}=
 \begin{pmatrix}
 0&-1& \cdots&0&1\\
 1&0&\cdots&0&0\\
 \vdots&\vdots&\ddots&\vdots&\vdots\\
 0&0&\cdots&0&-1\\
 -1&0 & \cdots&1&0
 \end{pmatrix}.
\end{align}

In this paper, by a $\C[[\hbar]]$-algebra, we mean a topological algebra over $\C[[\hbar]]$, equipped with its canonical $\hbar$-adic topology.
For $0\le i,j\le N-1$, write
  \begin{align}\label{gij}
g_{ij}(z_1,z_2)=\frac{q^{b_{i,j}}z_1-z_2}{z_1-q^{b_{i,j}}z_2}\in \C_{\ast}[[z_1,z_2,\hbar]].
 \end{align}
Here $\C_{\ast}[[z_1,z_2,\hbar]]$ denote the algebra extension of $\C[[z_1,z_2,\hbar]]$ by inverting $z_1$, $z_2$, $z_1-cz_2$ with $c$ invertible in $\C[[\hbar]]$. Denote by
\[\iota_{z_1,z_2}:\C_{\ast}[[z_1,z_2,\hbar]]\to \C[[\hbar]]((z_1))((z_2)),\]
the canonical algebra embedding that preserves each elements of $\C[[z_1,z_2,\hbar]]$.
For every integer $n$ and every invertible element $a$ in $\C[[\hbar]]$, we set
\[[n]_{a}=\frac{a^n-a^{-n}}{a-a^{-1}}\in \C[[\hbar]].\]
In addition, for a positive integer $k$, we set
\begin{align*}
  \binom{n}{k}_a=\prod_{r=1}^{k}\frac{[n-r+1]_a}{[r]_a}.
\end{align*}

Now we  introduce the notion of a quantum toroidal algebra over the commutative ring $\C[[\hbar]]$ (see \cite{GKV,VV2}).

\begin{dfnt}
The quantum toroidal algebra of type $\mathfrak{gl}_N$  over $\C[[\hbar]]$ is the unital $\C[[\hbar]]$-algebra
\[\mathcal{E}_N=\mathcal{E}_N(\hbar,p)
\] topologically generated by the set
\begin{align}\label{gentor}
\{h_{i}(n),~x_{i}^{\pm}(n),~c\mid 0\leq i\leq N-1, n\in\Z\},
\end{align}
and subject to the relations in terms of generating functions
\begin{align*}
\phi_{i}^{\pm}(z)=q^{\pm h_{i}(0)}\mathrm{exp}\left(\pm(q-q^{-1})\sum_{\pm n>0}h_{i}(n)z^{-n}\right)\ \text{and}\
x_{i}^{\pm}(z)= \sum_{n\in\Z}x_{i}^{\pm}(n)z^{-n}.
\end{align*}
The relations are as follows $(0\leq i,j\leq N-1)$:
\begin{align*}
(\mathrm{Q1})\ &[c,\phi_{i}^{\pm}(z)]=0=[c,x_{i}^{\pm}(z)]=[\phi_{i}^{\pm}(z_1),\phi_{j}^{\pm}(z_2)],\\
  (\mathrm{Q2})\  &\phi_{i}^{+}(z_1)\phi_{j}^{-}(z_2)=\phi_{j}^{-}(z_2)\phi_{i}^{+}(z_1)\iota_{z_1,z_2}\left(g_{ij}(z_1,q^cp^{-m_{i,j}}z_2)^{-1}g_{ij}(z_1,q^{-c}p^{-m_{i,j}}z_2)\right),\\
   (\mathrm{Q3})\  &\phi_{i}^{+}(z_1)x_{j}^{\pm}(z_2)=x_{j}^{\pm}(z_2)\phi_{i}^{+}(z_1)\iota_{z_1,z_2}g_{ij}(z_1,q^{\mp\frac{1}{2}c}p^{-m_{i,j}}z_2)^{\pm 1},\\
    (\mathrm{Q4})\ &\phi_{i}^{-}(z_1)x_{j}^{\pm}(z_2)=x_{j}^{\pm}(z_2)\phi_{i}^{-}(z_1)\iota_{z_2,z_1}g_{ji}(z_2,q^{\mp\frac{1}{2}c}p^{m_{i,j}}z_1)^{\mp 1},\\
   (\mathrm{Q5})\ &[x_{i}^{+}(z_1),x_{j}^{-}(z_2)]=\frac{\delta_{i,j}}{q-q^{-1}}\left(\phi_{i}^{+}(q^{\frac{1}{2}c}z_2)\delta\left(\frac{q^cz_2}{z_1}\right)-\phi_{i}^{-}(q^{-\frac{1}{2}c}z_2)\delta\left(\frac{q^{-c}z_2}{z_1}\right)\right),\\
    (\mathrm{Q6})\ & (p^{m_{i,j}}z_1-q^{\pm b_{i,j}}z_2)x_{i}^{\pm}(z_1)x_{j}^{\pm}(z_2)=(q^{\pm b_{i,j}}p^{m_{i,j}}z_1-z_2)x_{j}^{\pm}(z_2)x_{i}^{\pm}(z_1),\quad \textrm{if }b_{i,j}\ne 0,\\
(\mathrm{Q7})\
     &\sum_{\sigma\in S_{1-b_{i,j}}}\sum_{r=0}^{1-b_{i,j}}(-1)^r\binom{1-b_{i,j}}{r}_q
     x_{i}^{\pm}(z_{\sigma(1)})\cdots x_{i}^{\pm}(z_{\sigma(r)})x_{j}^{\pm}(w)\\
     &\qquad\times x_{i}^{\pm}(z_{\sigma(r+1)})\cdots x_{i}^{\pm}(z_{\sigma(1-b_{i,j})})=0,\quad \textrm{if }b_{i,j}\le 0.
   \end{align*}
\end{dfnt}

We also denote by  $\mathcal{E}_N^l$
the unital $\C[[\hbar]]$-algebra topologically generated by the set \eqref{gentor}, and subject to relations (Q1)-(Q6). Then $\mathcal{E}_N$ is naturally  a quotient algebra of $\mathcal{E}_N^l$. Additionally, we denote by  $\mathcal{E}_N^f$
the unital $\C[[\hbar]]$-algebra topologically generated by the set \eqref{gentor}, and subject to relations (Q1)-(Q4), (Q6) and
\[(\mathrm{Q}5')\quad (z_1-q^cz_2)^{\delta_{i,j}}(z_1-q^{-c}z_2)^{\delta_{i,j}}[x_{i}^{+}(z_1),x_{j}^{-}(z_2)]=0\]
for $0\leq i,j\leq N-1$. Note that the relation (Q5$'$) can be deduced from (Q5), and so
 $\mathcal{E}_N$ and $\mathcal{E}_N^l$ are naturally  quotient algebras of $\mathcal{E}_N^f$.

For later use, we also recall the Drinfeld realization of the quantum affine algebra $\mathcal{U}_{q}(\widehat{\mathfrak{sl}}_3)$ (see \cite{D,B}).
\begin{dfnt}
The quantum affine algebra $\mathcal{U}_{q}(\widehat{\mathfrak{sl}}_3)$ over $\C[[\hbar]]$ is the unital $\C[[\hbar]]$-algebra topologically generated by the set
\begin{align}\label{eq:genaff}
\{k_i(n),~e_i^{\pm}(n),~c\mid i=1,2,~n\in\Z\},
\end{align}
and subject to the relations in terms of generating functions
\[\varphi_i^{\pm}(z)=q^{\pm k_i(0)}\mathrm{exp}\left(\pm(q-q^{-1})\sum_{\pm n>0}k_i(n)z^{-n}\right)\quad \textrm{and}\quad e_i^{\pm}(n)=\sum_{n\in\Z}e_i^{\pm}(n)z^{-n}.\]
The relations are as follows $(i,j=1,2):$
\begin{align*}
(\mathrm{L1})\  &[c,\varphi_i^{\pm}(z)]=0=[c,e_i^{\pm}(z)]=[\varphi_i^{\pm}(z_1),\varphi_j^{\pm}(z_2)],\\
(\mathrm{L2})\  &\varphi_i^{+}(z_1)\varphi_j^{-}(z_2)=\varphi_j^{-}(z_1)\varphi_i^{+}(z_2)\iota_{z_1,z_2}\left(g_{ij}(z_1,q^{c}z_2)^{-1}g_{ij}(z_1,q^{-c}z_2)\right),\\
(\mathrm{L3})\  &\varphi_i^{+}(z_1)e_j^{\pm}(z_2)=e_j^{\pm}(z_2)\varphi_i^{\pm}(z_1)\iota_{z_1,z_2}\left(g_{ij}(z_1,q^{\mp \frac{1}{2}c}z_2)\right)^{\pm 1},\\
(\mathrm{L4})\  &\varphi_i^{-}(z_1)e_j^{\pm}(z_2)=e_j^{\pm}(z_2)\varphi_i^{-}(z_1)\iota_{z_2,z_1}\left(g_{ji}(z_2,q^{\mp \frac{1}{2}c}z_1)\right)^{\mp 1},\\
(\mathrm{L5})\  &[e_i^{+}(z_1),e_j^{-}(z_2)]=\frac{\delta_{i,j}}{q-q^{-1}}\left(\varphi_i^{+}(z_1q^{-\frac{1}{2}c})\delta\left(\frac{z_2q^{c}}{z_1}\right)-\varphi_i^{-}(z_1q^{\frac{1}{2}c})\delta\left(\frac{z_2q^{-c}}{z_1}\right)\right),\\
(\mathrm{L6})\  &(z_1-q^{\pm b_{i,j}}z_2)e_i^{\pm}(z_1)e_j^{\pm}(z_2)=(q^{\pm b_{i,j}}z_1-z_2)e_j^{\pm}(z_2)e_i^{\pm}(z_1),\\
(\mathrm{L7})\  &\sum_{\sigma\in S_2}\left(e_{i}^{\pm}(z_{\sigma(1)})e_{i}^{\pm}(z_{\sigma(2)})e_{j}^{\pm}(w)
-(q+q^{-1})e_{i}^{\pm}(z_{\sigma(1)})e_{j}^{\pm}(w)e_{i}^{\pm}(z_{\sigma(2)})\right.\\
    &\qquad\qquad\quad\left.+e_{j}^{\pm}(w)e_{i}^{\pm}(z_{\sigma(1)})e_{i}^{\pm}(z_{\sigma(2)})\right)=0\quad \textrm{if }b_{i,j}=-1.
\end{align*}
\end{dfnt}

 Let $\mathcal{U}_q^{l}(\widehat{\mathfrak{sl}}_3)$ be the unital $\C[[\hbar]]$-algebra topologically generated by the set $\eqref{eq:genaff}$ and subject to relations (L1)-(L6). Then  $\mathcal{U}_q(\widehat{\mathfrak{sl}}_3)$ is naturally a quotient algebra of $\mathcal{U}_q^{l}(\widehat{\mathfrak{sl}}_3)$.

For $0\leq i\leq N-1$, denote by $\mathcal{U}_{i}$ and $\mathcal{U}_{i}^{l}$  the subalgebras of  $\mathcal{E}_N$ and $\mathcal{E}_N^l$ topologically generated by the set
\begin{align}\label{affgen}
\{p^{jn}h_{j}(n),~p^{jn}x_{j}^{\pm}(n),c\mid j=i,i+1,~n\in\Z\},
\end{align}
respectively.
Here, as a convention,  $i+1$ is understood as  $0$ when $i=N-1$.

Then we have the following  straightforward result.
\begin{lemt}\label{le:Uiaff}
As a $\C[[\hbar]]$-algebra, $\mathcal{U}_i$ is a quotient algebra of  $\mathcal{U}_q(\widehat{\mathfrak{sl}}_3)$ with the quotient map given by
\begin{align}\label{eq:quomap}
c\mapsto c,\quad k_{\epsilon}(n)\mapsto p^{(i+\epsilon-1)n}h_{i+\epsilon-1}(n), \quad \textrm{and}\quad  e_{\epsilon}^{\pm}(n)\mapsto p^{(i+\epsilon-1)n}x_{i+\epsilon-1}^{\pm}(n)
\end{align}
for $\epsilon=1,2$ and $n\in\Z$. Similarly, $\mathcal{U}_i^{l}$ is a quotient algebra of  $\mathcal{U}_q^{l}(\widehat{\mathfrak{sl}}_3)$ with the quotient map as in  \eqref{eq:quomap}.
\end{lemt}

\subsection{Restricted $\mathcal{E}_{N}$-module of level $\ell$}

Throughout this paper, let $\ell$ be a fixed complex number.
In this subsection, we introduce and study the notion of restricted $\mathcal{E}_{N}$-module of level $\ell$.

 Let $W$ be a $\C[[\hbar]]$-module.
 We say that
\begin{itemize}
\item $W$ is torsion-free if $\hbar v \neq  0$ for every $0\neq v\in W$,
\item $W$ is separated if $\cap_{n\geq 1}\hbar^{n}W=0$, and
\item $W$ is  $\hbar$-adically complete if every Cauchy sequence in $W$ with respect to the $\hbar$-adic topology has a limit in $W$.
\end{itemize}
Additionally, we say that
 $W$ is topologically free if $W = W_0[[\hbar]]$ for some vector space $W_0$ over $\C$. It is known that a $\C[[\hbar]]$-module is topologically free if and only if it is torsion-free, separated, and $\hbar$-adically complete (see \cite{Kas}).

Assume now that $W=W_{0}[[\hbar]]$ is topologically free. For $k\in\mathbb{Z}_{+}$, write
\begin{align}\label{ehk}
\mathcal{E}_{\hbar}^{(k)}(W;z_1,z_2,\dots,z_k)=\textrm{Hom}_{\C[[\hbar]]}(W,W_0((z_1,z_2,\dots,z_k))[[\hbar]]).
\end{align}
Note that the $\C[[\hbar]]$-module
\[\mathcal{E}_{\hbar}^{(k)}(W;z_1,z_2,\dots,z_k)=\left(\textrm{Hom}_{\C}\left(W_0,W_0((z_1,z_2,\dots,z_k)\right)\right)[[\hbar]]\] is topologically free as well. For convenience, we also write
\begin{align*}
\mathcal{E}_{\hbar}^{(k)}(W)=\mathcal{E}_{\hbar}^{(k)}(W;z_1,z_2,\dots,z_k)\quad\textrm{and}\quad \mathcal{E}_{\hbar}(W)=\mathcal{E}_{\hbar}^{(1)}(W).
\end{align*}

\begin{dfnt}
Let $W$ be an  $\mathcal{E}_N^f$-module. We say that $W$ is restricted if $W$ is  topologically free as a $\C[[\hbar]]$-module, and
\[\phi_{i}^{\pm}(z),~ x_{i}^{\pm}(z)\in\mathcal{E}_{\hbar}(W)\quad \textrm{for }\,0\leq i\leq N-1.\]
Additionally, we say that $W$ is of level $\ell$ if $c$ acts as the scalar $\ell$ on $W$.
\end{dfnt}

Note that for every $\mathcal{E}_N$-module,
by using  the quotient map
\[
\mathcal{E}_N^f\rightarrow \mathcal{E}_N,
\]
there is a natural $\mathcal{E}_N^f$-module structure on $W$.

\begin{dfnt}
An $\mathcal{E}_N$-module is said to be restricted of level $\ell$ if  as an $\mathcal{E}_N^f$-module, it is restricted of level $\ell$.
\end{dfnt}

For $m\in\mathbb{Z}_{+}$ and  $0\leq i_1,i_2,\dots,i_m\leq N-1$, write
\begin{align}
&x_{i_1,i_2,\dots,i_m}^{\pm}(z_1,z_2,\dots,z_m)\\
=\ &\left(\prod_{1\leq s<t\leq m}\iota_{z_s,z_t}f^{\pm }_{i_s,i_t}(p^{m_{i_s,i_t}}z_s,z_t)\right)x_{i_1}^{\pm}(z_1)x_{i_2}^{\pm}(z_2)\cdots x_{i_m}^{\pm}(z_m)\nonumber,
\end{align}
where
\begin{align}\label{eq:fpm}
f^{\pm }_{i_s,i_t}(z_s,z_t)=(z_s-q^{\pm b_{i_s,i_t}}z_t)(z_s-z_t)^{-\delta_{i_s,i_t}}.
\end{align}
And, for $0\leq i,j\leq N-1$, write
\begin{align}\label{cij}
C_{i,j}=-(-1)^{\delta_{i,j}}.
\end{align}

\begin{lemt}\label{le:xiij}
Let $W$ be a restricted $\mathcal{E}_N^f$-module and $m\in \Z_+$. Then for any $0\leq i_1,i_2,\dots,i_m\leq N-1$, one has that
\[x_{i_1,i_2,\dots,i_m}^{\pm}(z_1,z_2,\dots,z_m)\in\mathcal{E}_{\hbar}^{(m)}(W).\]
Moreover, for every $\sigma\in S_m$, one has that
\begin{align}\label{xsigma}
&x_{i_{\sigma(1)},i_{\sigma(2)},\dots,i_{\sigma (m)}}^{\pm}(z_{\sigma(1)},z_{\sigma(2)},\dots,z_{\sigma (m)})\\
=\ &\left(\prod_{\substack{1\leq s<t\leq m,\\ \sigma(s)>\sigma(t)}}C_{i_{\sigma(s)},i_{\sigma(t)}}p^{m_{i_{\sigma(s)},i_{\sigma(t)}}}\right)x_{i_1,i_2,\dots,i_m}^{\pm}(z_1,z_2,\dots,z_m),\nonumber
\end{align}
where $m_{i_{\sigma(s)},i_{\sigma(t)}}$ is as in \eqref{M}.
\end{lemt}
\begin{proof}
The assertion is clear for the case that $m=1$. When $m=2$, it follows from the relation (Q6) that
\begin{equation}\label{tildefxij}
\begin{split}
&\iota_{z_1,z_2}\left(f_{i,j}^{\pm }(p^{m_{i,j}}z_1,z_2)\right)x_{i}^{\pm }(z_1)x_{j}^{\pm }(z_2)\\
=\ &\iota_{z_1,z_2}\left((z_1-z_2)^{-\delta_{i,j}}\right)\left((p^{m_{i,j}}z_1-q^{\pm b_{i,j}}z_2)x_{i}^{\pm }(z_1)x_{j}^{\pm }(z_2)\right)\\
=\ &\iota_{z_2,z_1}\left((z_1-z_2)^{-\delta_{i,j}}\right)\left((p^{m_{i,j}}z_1-q^{\pm b_{i,j}}z_2)x_{i}^{\pm }(z_1)x_{j}^{\pm }(z_2)\right)\\
=\ &\iota_{z_2,z_1}\left((z_1-z_2)^{-\delta_{i,j}}\right)\left((q^{\pm b_{i,j}}p^{m_{i,j}}z_1-z_2)x_{j}^{\pm }(z_2)x_{i}^{\pm }(z_1)\right)\\
=\ &C_{i,j}p^{m_{i,j}}\iota_{z_2,z_1}\left(f_{j,i}^{\pm }(p^{m_{j,i}}z_2,z_1)\right)x_{j}^{\pm }(z_2)x_{i}^{\pm }(z_1).
\end{split}
\end{equation}
This implies that $x_{i,j}^{\pm}(z_1,z_2)\in\mathcal{E}_{\hbar}^{(2)}(W)$ and \eqref{xsigma} holds for the case that $m=2$.

In what follows we prove \eqref{xsigma} by an induction argument on $m$.
For every $m>2$, by using the induction hypothesis
and \eqref{tildefxij}, we have that
\begin{align*}
&x_{i_{\sigma(1)},i_{\sigma(2)},\dots,i_{\sigma (m+1)}}^{\pm}(z_{\sigma(1)},z_{\sigma(2)},\dots,z_{\sigma (m+1)})\\
=\ &\prod_{\substack{1\leq s<t\leq m\\ \sigma(s)>\sigma(t)}}C_{i_{\sigma(s)},i_{\sigma(t)}}p^{m_{i_{\sigma(s)},i_{\sigma(t)}}}\prod_{k=1}^{m}\iota_{z_{\sigma(k)},z_{\sigma(m+1)}}\left(f_{i_{\sigma(k)},i_{\sigma(m+1)}}^{\pm}(p^{m_{i_{\sigma(k)},i_{\sigma(m+1)}}}z_{\sigma(k)},z_{\sigma(m+1)})\right)\\
&\cdot x_{i_1,\dots,\hat{i}_{\sigma(m+1),\dots,i_{m+1}}}^{\pm}(z_1,\dots,\hat{z}_{\sigma(m+1)},\dots,z_{m+1})x_{i_{\sigma(m+1)}}^{\pm}(z_{\sigma(m+1)})\\
=\ & \prod_{\substack{1\leq s<t\leq m+1\\ \sigma(s)>\sigma(t)}}C_{i_{\sigma(s)},i_{\sigma(t)}}p^{m_{i_{\sigma(s)},i_{\sigma(t)}}}\prod_{k=1}^{m}x_{i_1,\dots,i_{m+1}}^{\pm}(z_1,\dots,z_{m+1}).
\end{align*}
This proves \eqref{xsigma}, which also implies $x_{i_1,i_2,\dots,i_m}^{\pm}(z_1,z_2,\dots,z_m)\in\mathcal{E}_{\hbar}^{(m)}(W)$. Thus we complete the proof.
\end{proof}

For every restricted $\mathcal{E}_N^f$-module $W$,  Lemma \ref{le:xiij} implies that
\[x_{i,j}^{\pm}(a_1z,a_2z),~x_{i,i,j}^{\pm}(a_1z,a_2z,a_3z)\in \mathcal{E}_{\hbar}(W),\]
where $0\leq i,j\leq N-1$ and $a_1,a_2,a_3\in \mathbb{C}[[\hbar]]$.
Furthermore, we have:

 \begin{prpt}\label{pr:rescon}
 Let $W$ be a restricted $\mathcal{E}_N^f$-module of level $\ell$. Then $W$ becomes a restricted $\mathcal{E}_N$-module of level $\ell$ if and only if the following relations hold on $W$:
 \begin{align}\label{xij+-}
 [x_{i}^{+}(z_1),x_{j}^{-}(z_2)]=\frac{\delta_{i,j}}{q-q^{-1}}\left(\phi_{i,q}^{+}(q^{\frac{1}{2}\ell}z_2)\delta\left(\frac{q^{\ell}z_2}{z_1}\right)-\phi_{i,q}^{-}(q^{-\frac{1}{2}\ell}z_2)\delta\left(\frac{q^{-\ell}z_2}{z_1}\right)\right)
 \end{align}
 for $0\leq i,j\leq N-1$,
 \begin{align}\label{serre}
 x_{i,i,j}^{\pm}(qp^{-m_{i,j}}z,q^{-1}p^{-m_{i,j}}z,z)=0
 \end{align}
for $0\leq i,j\leq N-1$ with $b_{i,j}=-1$, and
\begin{align}\label{serre-1}
  x_{i,j}^\pm(z,z)=0
\end{align}
for $0\leq i,j\leq N-1$ with $b_{i,j}=0$.
 \end{prpt}
 \begin{proof}
 Note that $W$ is a restricted $\mathcal{E}_{N}$-module of level $\ell$ if and only if
 \eqref{xij+-}
 and the relation (Q7) holds on $W$. Assume now that $W$ is a  restricted $\mathcal{E}_N^f$-module of level $\ell$ such that $\eqref{xij+-}$ holds on it.
 Thus, it suffices to prove that  (Q7) holds on $W$ if and only if \eqref{serre} and \eqref{serre-1} hold on $W$.
 For the case that $b_{i,j}=0$, it is clear that the relation (Q7) is equivalent to \eqref{serre-1}.
 So we only need to treat the case that $b_{i,j}=-1$.

In this case  $W$ becomes a $\mathcal{E}_N^l$-module, and hence a $\mathcal{U}_i^{l}$-module by taking restriction, where $0\leq i\leq N-1$. Through the quotient map \eqref{eq:quomap}, $W$ becomes a restricted $\mathcal{U}_q^{l}(\widehat{\mathfrak{sl}}_3)$-module of level $\ell$ (see \cite{K} for definitions).
 Then by Lemma \ref{le:Uiaff} and \cite[Theorem 5.17]{CJKT}, we obtain that (Q7) holds if and only if
\begin{align*}
\begin{small}
\left(\left(f_{i,i}^{\pm}(z_1,z_2)f_{i,j}^{\pm}(z_1,z_3)f_{i,j}^{\pm}(z_2,z_3)\right)x_{i,q}^{\pm}(p^{-i}z_1)x_{i,q}^{\pm}(p^{-i}z_2)x_{j,q}^{\pm}(p^{-j}z_3)\right)|_{z_2=q^{- 1}z_3,z_1=qz_3}=0
\end{small}
\end{align*}
for $0\leq i,j\leq N-1$ with $b_{i,j}=-1$.
Note that the left-hand side of the above equality equals to
\[x_{i,i,j}^{\pm}\left(qp^{-j-m_{i,j}}z_3,q^{-1}p^{-j-m_{i,j}}z_3,p^{-j}z_3\right).\]
It follows that (Q7) holds on $W$ if and only if  \eqref{serre} holds  for $0\leq i,j\leq N-1$ with $b_{i,j}=-1$ by substituting $p^{-j}z_3=z$.
 This completes the proof.
 \end{proof}

\section{Quantum vertex algebras associated to $\widehat{\mathfrak{sl}}_\infty$}

 In this section, we introduce an $\hbar$-adic quantum vertex algebra associated to the affine Lie algebra of $\mathfrak{sl}_\infty$.

\subsection{$\hbar$-adic quantum vertex algebra}
In this subsection, we recall some basic notions and results related to $\hbar$-adic quantum vertex algebras.

We first recall the notion of nonlocal vertex algebras (see \cite{Li2}).
\begin{dfnt}\label{de:nonlocalva}
A nonlocal vertex algebra over a commutative ring $R$ is a $R$-module $V_0$, equipped with a $R$-linear map
 \begin{align*}
 Y(\cdot, z):\ V_0&\to \ \mathrm{Hom}_{R}(V_0,V_0((z))),\\
 v&\mapsto Y(v,z)=\sum_{n\in\Z}v_nz^{-n-1},
 \end{align*}
 and equipped with a distinguished vector ${\bf 1}\in V_0$, satisfying the conditions that
 \begin{align*}
 Y({\bf 1},z)v=v, \quad Y(v,z){\bf 1}\in v+zV_0[[z]]\quad\textrm{for }u,v\in V_0,
 \end{align*}
 and that for $u,v,w\in V_0$, there exists $l\in\N$ such that
 \begin{align*}
(z_0+z_2)^{l}Y(u,z_0+z_2)Y(v,z_2)w=(z_0+z_2)^{l}Y(Y(u,z_0)v,z_2)w.
 \end{align*}
 \end{dfnt}

\begin{dfnt}
A vertex algebra is a nonlocal vertex algebra $V_0$ over $\C$ satisfying the following locality condition: for $u,v,w\in V_0$, there exists $k\in\mathbb{N}$ such that
\[(z_1-z_2)^kY(u,z_1)Y(v,z_2)w=(z_1-z_2)^kY(v,z_2)Y(u,z_1)w.\]
 \end{dfnt}

As an $\hbar$-adic analog of Definition \ref{de:nonlocalva}, we also have the following notion (see \cite{Li4}):
 \begin{dfnt}
 An $\hbar$-adic nonlocal vertex algebra is a topologically free $\C[[\hbar]]$-module $V$ equipped with a $\C[[\hbar]]$-module map
 \begin{align*}
  Y(\cdot, z):\ V&\to \mathcal{E}_{\hbar}(V),\\
 v&\mapsto Y(v,z)=\sum_{n\in\Z}v_nz^{-n-1},
 \end{align*}
 and equipped with a distinguished vector ${\bf 1}\in V$, satisfying the conditions that
 \begin{align}\label{vacu}
 Y({\bf 1},z)v=v, \quad Y(v,z){\bf 1}\in v+zV[[z]]\quad\textrm{for }v\in V,
 \end{align}
 and that for $u,v,w\in V$, $n\in\N$, there exists $l\in\N$ such that
 \begin{align}
(z_0+z_2)^{l}Y(u,z_0+z_2)Y(v,z_2)w-(z_0+z_2)^{l}Y(Y(u,z_0)v,z_2)w\in \hbar^nV[[z_0^{\pm 1},z_2^{\pm 1}]].
 \end{align}
\end{dfnt}

Let $V$ be a topologically free $\C[[\hbar]]$-module, which is equipped with a vector ${\bf 1}\in V$ and a $\C[[\hbar]]$-module map
\[ Y(\cdot, z):V\to \mathcal{E}_{\hbar}(V).\]
It is proved in \cite{Li3} that $(V,Y,{\bf 1})$ is an $\hbar$-adic nonlocal vertex algebra if and only if for every positive integer $n$, $(V/\hbar^nV, Y^{(n)},{\bf 1}+\hbar^nV)$ is a nonlocal vertex algebra over $\C[[\hbar]]/\hbar^n\C[[\hbar]]$, where \[Y^{(n)}(\cdot, z):V/\hbar^nV\to \textrm{Hom}_{\C[[\hbar]]/\hbar^n\C[[\hbar]]}(V/\hbar^nV,V/\hbar^nV((z)))\]
is the canonical $\C[[\hbar]]/\hbar^n\C[[\hbar]]$-map induced by $Y$.

\begin{dfnt}
A subset $S$ of an $\hbar$-adic nonlocal vertex algebra $V$ is called a generating subset if for every positive integer $n$,
\[\{a+\hbar^nV\mid a\in S\}\] generates $V/\hbar^nV$ as a nonlocal vertex algebra over $\C[[\hbar]]/\hbar^n\C[[\hbar]]$.
\end{dfnt}
For two $\C[[\hbar]]$-modules $V$ and $W$, denote by $V\widehat{\otimes}W$ the completed tensor product of $V$ and $W$. If $V=V_0[[\hbar]]$ and $W=W_0[[\hbar]]$ for some vector spaces $V_0$ and $W_0$ over $\C$, then we have that
\[V\widehat{\otimes}W=(V_0\otimes W_0)[[\hbar]].\]
 For $A(z_1,z_2),B(z_1,z_2)\in V[[z_1^{\pm1},z_2^{\pm 1}]]$, we write
$A(z_1,z_2)\sim B(z_1,z_2)$ if for every positive integer $n$, there exists $k\in\N$ such that
\[(z_1-z_2)^k\left(A(z_1,z_2)-B(z_1,z_2)\right)\in \hbar^nV[[z_1^{\pm 1},z_2^{\pm 1}]].\]

Now, let $(V,Y,{\bf 1})$ be an $\hbar$-adic nonlocal vertex algebra. Define a $\C[[\hbar]]$-module map
 \begin{align*}
 Y(z):\ V\widehat{\otimes}V\to V[[z,z^{-1}]]
 \end{align*}
by
\[Y(z)(u\otimes v)=Y(u,z)v\quad \textrm{for } u,v\in V.\]
Denote by $\partial$ the following canonical derivation of $V$:
\[\partial:V\rightarrow V,\quad u\mapsto u_{-2}{\bf 1}.\]

The following notion is introduced in \cite{Li4}:

 \begin{dfnt}
 An $\hbar$-adic quantum vertex algebra is an $\hbar$-adic nonlocal vertex algebra $V$ equipped with a $\C[[\hbar]]$-module map
 \begin{align*}
 S(z):\ V\widehat{\otimes}V\to V\widehat{\otimes}V\widehat{\otimes}\C((z))[[\hbar]],
 \end{align*}
 which satisfies the shift condition:
 \begin{align*}
 [\partial\otimes 1,S(z)]=-\frac{d}{dz}S(z),
 \end{align*}
 the quantum Yang-Baxter equation:
\begin{align*}
S^{12}(z_1)S^{13}(z_1+z_2)S^{23}(z_2)=S^{23}(z_2)S^{13}(z_1+z_2)S^{12}(z_1),
 \end{align*}
 and the unitary condition:
 \begin{align*}
 S^{21}(z)S(-z)=1,
 \end{align*}
and subject to the following axioms:
 \begin{enumerate}
     \item The vacuum property:
     \begin{align*}
     S(z)({\bf{1}}\otimes v)={\bf 1}\otimes v \quad\textrm{for }v\in V.
     \end{align*}
     \item The $S$-locality: for every $u,v\in V$, one has that
     \begin{align*}
     Y(u,z_1)Y(v,z_2)\sim Y(z_2)(1\otimes Y(z_1))S(z_2-z_1)(v\otimes u).
     \end{align*}
     \item The hexagon identity:
     \begin{align*}
     S(z_1)(Y(z_2)\otimes 1)=(Y(z_2)\otimes 1)S^{23}(z_1)S^{13}(z_1+z_2).
     \end{align*}
 \end{enumerate}
 \end{dfnt}

Let $G$ be an abstract group. By a  $G$-vertex algebra, we mean a pair $(V_0,\mathcal{R})$, where $V_0$ is a  vertex algebra and $\mathcal{R}$ is a group homomorphism from $G$ to $\mathrm{Aut}(V_0)$. Here, $\mathrm{Aut}(V_0)$ stands for the group of all vertex algebra automorphisms on $V_0$. Similarly, we have the obvious notions of nonlocal $G$-vertex algebras, $\hbar$-adic nonlocal $G$-vertex algebras, and  $\hbar$-adic quantum $G$-vertex algebras.
Note that for every $\hbar$-adic nonlocal $G$-vertex algebra $(V,\mathcal{R})$ and $n\in \Z_+$, $(V/\hbar^n V, \mathcal{R}^{(n)})$ is a nonlocal $G$-vertex algebra over $\C[[\hbar]]/\hbar^n\C[[\hbar]]$, where
\[
\mathcal{R}^{(n)}: G\rightarrow \mathrm{Aut}(V/\hbar^n V)
\] is the group homomorphism induced by $\mathcal{R}$.

 % For two $\hbar$-adic nonlocal vertex algebras $(V_1,Y_1,{\bf 1})$ and $(V_2,Y_2,{\bf 1})$, a homomorphism of $\hbar$-adic nonlocal vertex algebras from $V_1$ to $V_2$ is a $\C[[\hbar]]$-module map $f$ such that
 % \begin{align*}
 % f({\bf 1})={\bf 1},\quad f(Y(u,z)v)=Y(f(u),z)f(v)\quad\textrm{for }u,v\in V_1.
 % \end{align*}
 % For an $\hbar$-adic nonlocal vertex algebra $V$, we say that a homomorphism from  $V$ to $V$ is an automorphism if there exists a homomorphism $g:V\to V$ such that
 % \[f\circ g=g\circ f=\textrm{Id}_{V}.\]

 \subsection{New realization of the universal affine vertex algebra associated to  $\mathfrak{sl}_{\infty}$}

Denote by $\mathfrak{gl}_{\infty}$ the Lie algebra of all complex matrices $(d_{i,j})_{i,j\in\Z}$ such
that the number of nonzero $d_{i,j}$'s is finite.  For every $i,j\in\Z$, we write $E_{i,j}$ for the matrix whose only nonzero entry is the $(i,j)$-entry which is $1$. Equip $\mathfrak{gl}_{\infty}$ with the bilinear form $(\cdot,\cdot)$ defined by
\[(E_{i,j},E_{k,l})=\delta_{i,l}\delta_{j,k}\quad \textrm{for }i,j,k,l\in\Z,\]
which is nondegenerate, symmetric and invariant.

Write \[
\mathfrak{sl}_{\infty}=[\mathfrak{gl}_{\infty},\mathfrak{gl}_{\infty}]\] for the derived Lie subalgebra of $\mathfrak{gl}_{\infty}$. It is straightforward to see that $(\cdot,\cdot)$ is nondegenerate, symmetric and invariant  on $\mathfrak{sl}_{\infty}$ as well. Associated to the pair $(\mathfrak{sl}_\infty,(\cdot,\cdot))$, we have the  affine Lie algebra:
\[\widehat{\mathfrak{sl}}_{\infty}=(\mathfrak{sl}_{\infty}\otimes\C[t,t^{-1}])\oplus \C c,\]
where $c$ is a central element and the Lie bracket is given by
\begin{align*}
[a(m), b(n)]=[a,b](m+n)+m\delta_{m+n,0}(a,b)c
\end{align*}
for $a,b\in\mathfrak{sl}_{\infty}$ and $m,n\in\Z$. Here $a(m)$ stands for $a\otimes t^m$.

 Let $A=(a_{i,j})_{i,j\in\Z}$ be the
Cartan matrix associated to $\mathfrak{sl}_{\infty}$, where
 \[a_{i,i}=2,\quad a_{i,i\pm 1}=-1,\quad\textrm{and}\quad a_{i,j}=0\quad\textrm{ for }i,j\in\Z\textrm{ with }j\neq i,i\pm 1.\]
 For every $i\in\Z$, we set
\begin{align*}
H_i=E_{i,i}-E_{i+1,i+1}, \quad X_i^{+}=E_{i,i+1},\quad\textrm{and}\quad   X_i^{-}=E_{i+1,i}.
\end{align*}
Then we have the following straightforward  presentation of $\widehat{\mathfrak{sl}}_{\infty}$, whose proof is omitted (cf. \cite{MRY}).
 \begin{lemt}\label{le:slgen}
 The Lie algebra $\widehat{\mathfrak{sl}}_{\infty}$
 is generated by the set
 \[\{H_i(m),~X_i^{\pm}(m),~c \mid   i,m\in\Z\}\]
 and subject to the relations in terms of generating functions
 \[H_i(z)=\sum_{n\in\Z}H_i(n)z^{-n-1}\quad\textrm{and}\quad X_i^\pm(z)=\sum_{n\in\Z}X_i^\pm(n)z^{-n-1}.\]
The relations are as follows $(i,j\in \Z)$:
 \begin{align}
&\label{eq:HH}\left[H_i(z_1),H_j(z_2)\right]\ =\ a_{i,j}c\frac{\partial}{\partial z_2}z_1^{-1}\delta\left(\frac{z_2}{z_1}\right),\\
 &\left[H_i(z_1),X_j^{\pm}(z_2)\right]\ =\ a_{i,j}X_j^{\pm}(z_2)z_1^{-1}\delta\left(\frac{z_2}{z_1}\right),\\
 &\left[X_i^{+}(z_1),X_j^{-}(z_2)\right]\ =\ \delta_{i,j}\left(H_i(z_2)z_1^{-1}\delta\left(\frac{z_2}{z_1}\right)+c\frac{\partial}{\partial z_2}z_1^{-1}\delta\left(\frac{z_2}{z_1}\right)\right),\\
 &(z_1-z_2)^{1-\delta_{i,j}}\left[X_i^{\pm}(z_1),X_j^{\pm}(z_2)\right]\ =\ 0,\\
 &\label{eq:XXX}\left[X_i^{\pm}(z_1),\left[X_i^\pm(z_2),\dots,\left[X_i^{\pm}(z_{1-a_{i,j}}),X_{j}^{\pm}(w)\right]\cdots\right]\right]=0\quad\textrm{if } a_{i,j}\le 0.
 \end{align}
 \end{lemt}

Set $\widehat{\mathfrak{sl}}_{\infty}^{+}=(\mathfrak{sl}_{\infty}\otimes\C[t])\oplus \C c$, and form the induced $\widehat{\mathfrak{sl}}_{\infty}$-module
\[V_{\widehat{\mathfrak{sl}}_{\infty}}(\ell,0)=\mathcal{U}(\widehat{\mathfrak{sl}}_{\infty})\otimes_{\mathcal{U}(\widehat{\mathfrak{sl}}_{\infty}^{+})}\C_{\ell},\]
where $\C_{\ell}$ is the $\widehat{\mathfrak{sl}}_{\infty}^{+}$-module on which $\mathfrak{sl}_{\infty}\otimes\C[t]$ acts trivially and $c$ acts as the scalar $\ell$. It is known (see \cite{LL}) that there is a (unique) vertex algebra structure on $V_{\widehat{\mathfrak{sl}}_{\infty}}(\ell,0)$ such that ${\bf 1}=1\otimes 1$ is the vacuum vector and that
\[Y(a,z)=\sum_{n\in\Z}a(n)z^{-n-1}\quad \textrm{for }a\in \mathfrak{sl}_{\infty}.\]
Here we identify $a$ with $a(-1)\otimes 1$ in $V_{\widehat{\mathfrak{sl}}_{\infty}}(\ell,0)$ for $a\in \mathfrak{sl}_{\infty}$.
For every $n\in\Z$, there is a Lie algebra automorphism
\begin{align*}
\sigma_n:\ \mathfrak{sl}_{\infty}\to \mathfrak{sl}_{\infty}
\end{align*}
given by
\begin{align}\label{eq:sigman}
\sigma_n(E_{i,j})=E_{i+nN,j+nN}\quad \textrm{and}\quad  \sigma_n(E_{i,i}-E_{j,j})=E_{i+nN,i+nN}-E_{j+nN,j+nN}
\end{align}
for $i,j\in\Z$ with $i\neq j$. Note that, viewed as an automorphism of $\mathcal{U}(\widehat{\mathfrak{sl}}_{\infty})$,   $\sigma_{n}$ preserves the subalgebra $\mathcal{U}(\widehat{\mathfrak{sl}}_{\infty}^{+})$. Thus, for every $n\in\Z$, the automorphism $\sigma_n$ of $\mathfrak{sl}_{\infty}$ induces an $\widehat{\mathfrak{sl}}_{\infty}$-module automorphism of $V_{\widehat{\mathfrak{sl}}_{\infty}}(\ell,0)$, which we still denote by $\sigma_n$. One easily checks that $\sigma_n$ is a vertex algebra automorphism of $V_{\widehat{\mathfrak{sl}}_{\infty}}(\ell,0)$ as well.
Then we obtain in this way a group homomorphism
\[\sigma:\ \Z \to \mathrm{Aut}(V_{\widehat{\mathfrak{sl}}_{\infty}}(\ell,0)),\quad n\mapsto \sigma_n,\]
and hence obtain a $\Z$-vertex algebra
\[(V_{\widehat{\mathfrak{sl}}_{\infty}}(\ell,0),\sigma).\]

To deforming the $\Z$-vertex algebra $(V_{\widehat{\mathfrak{sl}}_{\infty}}(\ell,0),\sigma)$, one needs  another realization of  $(V_{\widehat{\mathfrak{sl}}_{\infty}}(\ell,0),\sigma)$ which we state below.

We start with the notion of Lie conformal algebra (see \cite{Kac,DLM}).
\begin{dfnt}
 A conformal Lie algebra $(C,Y^{+},T)$ is a vector space $C$ over $\C$, equipped with a linear operator $T$ on $C$ and a linear map
 \begin{align*}
 Y^{+}(\cdot,z):\ &C\to \mathrm{Hom}_{\C}(C,z^{-1}C[z^{-1}]),\\
 &u\ \mapsto Y^{+}(u,z)=\sum_{n\geq 0}u_nz^{-n-1}
 \end{align*}
 such that for every $u,v\in C$,
 \begin{align*}
 [T,Y^{+}(u,z)]\ &=\ Y^{+}(Tu,z)=\frac{d}{dz}Y^{+}(u,z),\\
 Y^{+}(u,z)v\ &=\ \mathrm{Sing}_{z}\left(e^{zT}Y^{+}(v,-z)u\right),\\
 [Y^{+}(u,z_1),Y^{+}(v,z_2)]\ &=\ \mathrm{Sing}_{z_1}\mathrm{Sing}_{z_2}\left(Y^{+}\left(Y^{+}(u,z_1-z_2)v,z_2\right)\right),
 \end{align*}
 where $\mathrm{Sing}$ stands for the singular part.
\end{dfnt}

 Let $C$ be an arbitrary conformal Lie algebra. Recall that the coefficient algebra of $C$ is a Lie algebra $\widehat{C}$, where
\[\widehat{C}=(C\otimes \C[t,t^{-1}])/\textrm{Im}(T\otimes 1+1\otimes d/dt)\]
as a vector space over $\C$ and the Lie bracket is given by
\[ [u(m),v(n)]\ =\ \sum_{k\geq 0}\binom{m}{k}(u_kv)(m+n-k)\quad\textrm{for }u,v\in C\textrm{ and }m,n\in\Z.\]
Here we write $u(m)$ for the image of $u\otimes t^m$ in $\widehat{C}$. It is known that $\widehat{C}$ admits a polar decomposition of Lie subalgebras (see \cite{DLM}):
\[\widehat{C}=\widehat{C}^{-}\oplus \widehat{C}^{+},\]
 where
\begin{align*}
\widehat{C}^{-}=\textrm{Span}_{\C}\{u(m)\mid u\in C,~m<0\}\quad \textrm{and}\quad \widehat{C}^{+}=\textrm{Span}_{\C}\{u(m)\mid u\in C,~m\geq 0\}.
\end{align*}
Form the induced $\widehat{C}$-module
\begin{align*}
V_C=\mathcal{U}(\widehat{C})\otimes_{\mathcal{U}(\widehat{C}^{+})}\C,
\end{align*}
where $\C$ is the trivial $\widehat{C}^{+}$-module.  It is proved in \cite[Theorem 4.8]{DLM} that there is a (unique) vertex algebra structure on $V_{C}$ such that ${\bf 1}=1\otimes 1\in V_{C}$ is the vacuum vector and the vertex operator map $Y$ is determined by
\begin{align}\label{eq:Vcverop}
Y(u,z)=\sum_{m\in\Z}u(m)z^{-m-1}\quad\textrm{for }u\in C.
\end{align}
Here, we identify $u$ with $u(-1)\otimes 1\in V_{C}$ for $u\in C$.
 % Denote by $\mathfrak{Conf}(\mathcal{B},N)$ the category whose objects are conformal Lie algebras $C$ which contain $\mathcal{B}$ as a generating subset and
 % \[a_nb=0\quad\textrm{for }a,b\in \mathcal{B},~n\geq N(a,b).\]

 Let $\mathcal{B}$ be a set and let
 \[N:\mathcal{B}\times \mathcal{B}\to \N\] be a symmetric function.
 It is proved in \cite[Proposition 3.1]{R} that there is a (unique) conformal Lie algebra $C(\mathcal{B},N)$ such that \begin{itemize}
     \item  $C(\mathcal{B},N)$ contains $\mathcal{B}$ as a generating set, and
     \item the
 coefficient algebra $\widehat{C}(\mathcal{B},N)$ of $C(\mathcal{B},N)$ is abstractly  generated by the set
 \[\{b(n)\mid b\in \mathcal{B},~n\in\Z\},\]
subject to the relations
\begin{align}\label{eq:local}
\sum_{i\geq 0}(-1)^i\binom{N(a,b)}{i}[a(n-i),b(m+i)]=0\quad\textrm{for }a,b\in\mathcal{B}\textrm{ and }m,n\in\Z.
\end{align}
 \end{itemize}

 The following result is given in \cite{R}.
 \begin{lemt}\label{le:vaho}
 Let $V$ be a vertex algebra, and let $f:\mathcal{B}\to V$ be a map such that
  \[(z_1-z_2)^{N(a,b)}[Y(f(a),z_1),Y(f(b),z_2)]=0\quad\textrm{for }a,b\in\mathcal{B}.\]
  Then there exists a unique vertex algebra homomorphism $\hat{f}:V_{C(\mathcal{B},N)}\to V$ such that $\hat{f}|_{\mathcal{B}}=f$.
 \end{lemt}

 Write $\mathcal{B}_{\widehat{\mathfrak{sl}}_{\infty}}=\{h_i,~x_i^{\pm}\mid i\in\Z\}$, and define a symmetric function \[N_{\widehat{\mathfrak{sl}}_{\infty}}:\mathcal{B}_{\widehat{\mathfrak{sl}}_{\infty}}\times \mathcal{B}_{\widehat{\mathfrak{sl}}_{\infty}}\to\N\] by setting
 \[N_{\widehat{\mathfrak{sl}}_{\infty}}(h_i,h_j)=2,\  N_{\widehat{\mathfrak{sl}}_{\infty}}(x_i^{\pm},x_j^{\pm})=1-\delta_{i,j},\  N_{\widehat{\mathfrak{sl}}_{\infty}}(h_i,x_j^{\pm })=1, \  N_{\widehat{\mathfrak{sl}}_{\infty}}(x_i^{+},x_j^{-})=2\delta_{i,j}\]
 for $i,j\in\Z$. Then we obtain a conformal Lie algebra $C(\mathcal{B}_{\widehat{\mathfrak{sl}}_{\infty}},N_{\widehat{\mathfrak{sl}}_{\infty}})$ as well as a vertex algebra $V_{C(\mathcal{B}_{\widehat{\mathfrak{sl}}_{\infty}},N_{\widehat{\mathfrak{sl}}_{\infty}})}$ associated to the pair $(\mathcal{B}_{\widehat{\mathfrak{sl}}_{\infty}},N_{\widehat{\mathfrak{sl}}_{\infty}})$.

Form the quotient vertex algebra
\[F(A,\ell)=V_{C(\mathcal{B}_{\widehat{\mathfrak{sl}}_{\infty}},N_{\widehat{\mathfrak{sl}}_{\infty}})}/I_1\]  of $V_{C(\mathcal{B}_{\widehat{\mathfrak{sl}}_{\infty}},N_{\widehat{\mathfrak{sl}}_{\infty}})}$, where
$I_1$ is the ideal of  $V_{C(\mathcal{B}_{\widehat{\mathfrak{sl}}_{\infty}},N_{\widehat{\mathfrak{sl}}_{\infty}})}$ generated by the elements
 \begin{align}\label{eq:genI1}
 (h_i)_0h_j,\quad (h_i)_1(h_j)-a_{i,j}\ell{\bf 1}, \quad \textrm{and}\quad(h_i)_0x_j^{\pm}\mp a_{i,j}x_j^{\pm}\quad\textrm{for }i,j\in \Z.
 \end{align}
For convenience, we still denote the image of a vector $v\in V_{C(\mathcal{B}_{\widehat{\mathfrak{sl}}_{\infty}},N_{\widehat{\mathfrak{sl}}_{\infty}})}$ in $F(A,\ell)$  by itself.

By Lemma \ref{le:vaho}, for every $n\in \Z$,
 there is a vertex algebra homomorphism
 \begin{align*}
 \rho_n^{(1)}:\ V_{C(\mathcal{B}_{\widehat{\mathfrak{sl}}_{\infty}},N_{\widehat{\mathfrak{sl}}_{\infty}})}\to V_{C(\mathcal{B}_{\widehat{\mathfrak{sl}}_{\infty}},N_{\widehat{\mathfrak{sl}}_{\infty}})}
 \end{align*}
 determined by
 \begin{align}\label{dfrho}
 \rho_n^{(1)}(h_i)=h_{i+nN}\quad \textrm{and}\quad \rho_n^{(1)}(x_i^{\pm})=x_{i+nN}^{\pm}\quad \textrm{for }i\in\Z.
 \end{align}
 It is obvious that $\rho_n^{(1)}$ is an automorphism.
Note that $\rho_n^{(1)}$ preserves the ideal $I_1$, and so it induces an automorphism of $F(A,\ell)$, which we still denote by $\rho_n^{(1)}$. Then we obtain  a  $\Z$-vertex algebra  $(F(A,\ell),\rho^{(1)})$, where
 \begin{eqnarray}\label{eq:defrho1}
 \rho^{(1)}:\ \Z \to \mathrm{Aut}(F(A,\ell)),\quad n\mapsto \rho_n^{(1)}.
 \end{eqnarray}

Consider also the quotient vertex algebra
\[F(A,\ell)/I_2,\]
 where  $I_2$ is the ideal of $F(A,\ell)$ generated by the elements
 \begin{equation}\label{eq:genI2}
 \begin{split}
 (x_i^{+})_{0}x_j^{-}-\delta_{i,j}h_i,\quad (x_i^{+})_{1}x_j^{-}-\delta_{i,j}\ell{\bf 1},\quad\left((x_i^{\pm })_{0}\right)^{1-a_{i,j}}x_j^{\pm}
\end{split}
 \end{equation}
for $i,j\in \Z$ with $a_{i,j}\le 0$.
 % For every $v\in F(A,\ell)$, we still use $v$ to denote its image in $F(A,\ell)/I_2$.
 For every $n\in\Z$, since the automorphism $\rho_n^{(1)}$ preserves the ideal $I_2$, it induces an automorphism, still called $\rho_n^{(1)}$, of $F(A,\ell)/I_2$.
Then we have a $\Z$-vertex algebra  $(F(A,\ell)/I_2,\rho^{(1)})$ as well, where the group homomorphism $\rho^{(1)}$  is as in \eqref{eq:defrho1}.
% Then $\rho^{(1)}$ induces a group homomorphism from $\Z$ to $\textrm{Aut}(F(A,\ell)/I_2)$, still call $\rho^{(1)}$,
% and so that is  .
% such that
%   \begin{align*}
% \bar{\rho}_n^{(1)}Y(v,z)(\bar{\rho}_n^{(1)})^{-1}=Y(\bar{\rho}_n^{(1)}(v),z)\quad\textrm{for }n\in\Z\textrm{ and }v\in \{h_i,x_i^{\pm}\mid i\in\Z\}.
%   \end{align*}
%  Following  \cite[Lemma 3.3]{JKLT1}, one obtains that $(F(A,\ell)/I_2,\bar{\rho}^{(1)})$ is a $\Z$-vertex algebra.

\begin{prpt}\label{pr:vasl}
The assignment
\[H_i\mapsto h_i\quad\textrm{and}\quad \ X_i^{\pm}\mapsto x_i^{\pm}\quad\textrm{for } i\in\Z\]
determines a $\Z$-vertex algebra  isomorphism from
$(V_{\widehat{\mathfrak{sl}}_{\infty}}(\ell,0),\sigma)$  to $(F(A,\ell)/I_2,\rho^{(1)})$.
\end{prpt}
  \begin{proof}
  By Lemma \ref{le:slgen}, \eqref{eq:local}, \eqref{eq:genI1} and \eqref{eq:genI2},  there is an $\widehat{\mathfrak{sl}}_{\infty}$-module structure on the quotient vertex algebra $F(A,\ell)/I_2$ such that
  \begin{align}\label{eq:slac}
H_i(z)\mapsto Y(h_i,z),\quad  X_i^{\pm}(z)\mapsto Y(x_i^{\pm},z),\quad  \textrm{and}\quad   c\mapsto\ell\quad \textrm{for }i\in\Z.
\end{align}
 Since $Y(a,z){\bf 1}\in \left(F(A,\ell)/I_2\right)[[z]]$ for $a\in F(A,\ell)/I_2$, there exists a surjective $\widehat{\mathfrak{sl}}_{\infty}$-module homomorphism
 \[\psi:\ V_{\widehat{\mathfrak{sl}}_{\infty}}(\ell,0)\to F(A,\ell)/I_2\]
such that
 \[\psi(H_i)=h_i\quad \textrm{and}\quad \psi(X_i^{\pm})=x_i^{\pm}\quad\textrm{for } i\in\Z\]
 by \eqref{eq:slac}.
 One easily checks that $\psi$ is a $\Z$-vertex algebra homomorphism as well.

 On the other hand, by Lemma \ref{le:vaho}, we get a vertex algebra
 homomorphism
 \[\psi':\ V_{C(\mathcal{B}_{\widehat{\mathfrak{sl}}_{\infty}},N_{\widehat{\mathfrak{sl}}_{\infty}})}\to V_{\widehat{\mathfrak{sl}}_{\infty}}(\ell,0) \]
 such that
 \[\psi'(h_i)=H_i\quad\textrm{and}\quad \psi'(x_i^{\pm})=X_i^{\pm}\quad \textrm{for }i\in\Z.\]
Then by \eqref{eq:genI1}, we have that $\psi'$ factors through $F(A,\ell)$, and we get a vertex algebra homomorphism $\psi'$ from $F(A,\ell)$ to $V_{\widehat{\mathfrak{sl}}_{\infty}}(\ell,0)$. Furthermore, by \eqref{eq:genI2}, we see that
$\psi'$ also factors through $F(A,\ell)/I_2$, and so we get a vertex algebra homomorphism  $\psi'$ from $F(A,\ell)/I_2$ to $V_{\widehat{\mathfrak{sl}}_{\infty}}(\ell,0)$.

  Since the elements $H_i,X_i^{\pm}$ ($i\in\Z$)  generate the vertex algebra $V_{\widehat{\mathfrak{sl}}_{\infty}}(\ell,0)$, % and $\{h_i+I_2,x_i+I_2\mid i\in\Z\}$ is a  generating subset of  $F(A,\ell)/I_2$,
 we have that \[\psi'\circ\psi=\textrm{Id}_{V_{\widehat{\mathfrak{sl}}_{\infty}}(\ell,0)},\]%\quad\textrm{and}\quad \psi\circ\psi'=\textrm{Id}_{F(A,\ell)/I_2}.\]
which implies that $\psi$ is injective. The assertion then follows.
  \end{proof}
 \subsection{$\hbar$-adic quantum vertex algebra associated to $\widehat{\mathfrak{sl}}_{\infty}$}
 In this subsection, we construct an $\hbar$-adic quantum $\Z$-vertex algebra associated to  $\widehat{\mathfrak{sl}}_{\infty}$ by deforming the $\Z$-vertex algebra $(V_{\widehat{\mathfrak{sl}}_{\infty}}(\ell,0),\sigma)$.

Let $n$ be an integer. We define the integers $\bar{n}$ and $\underline{n}$ by
\[n=\bar{n}N+\underline{n}\quad \textrm{with}\quad 0\leq \underline{n}\leq N-1.\]
Similar to the definition of quantum integers, we set
\[[n]_{q^{\frac{\partial}{\partial z}}}=\mathrm{Sgn}(n)\left( q^{(n-1)\frac{\partial}{\partial z}}+q^{(n-3)\frac{\partial}{\partial z}}+\cdots+q^{(-n+1)\frac{\partial}{\partial z}}\right),\]
where
\begin{equation*}
\mathrm{Sgn}(n)=    \begin{cases} 1\quad&\text{if}\ n>0;\\
0\quad&\text{if}\ n=0;\\
-1\quad&\text{if}\ n<0.
     \end{cases}
\end{equation*}

Write
\begin{align*}
\tau=\{\tau_{ij}(z),\tau_{ij}^{1,+}(z),\tau_{ij}^{1,-}(z),\tau_{ij}^{2,+}(z),\tau_{ij}^{2,-}(z),\tau_{ij}^{\epsilon_1,\epsilon_2}(z)\mid i,j\in\Z,\epsilon_1,\epsilon_2=\pm\},
\end{align*}
which is a subset of $\C((z))[[\hbar]]$, where
\begin{align}
\label{tauij}\tau_{ij}(z)\ &=\ [b_{\underline{i},\underline{j}}]_{q^{\frac{\partial}{\partial z}}}[\ell]_{q^{\frac{\partial}{\partial z}}}q^{\ell\frac{\partial}{\partial z}}\frac{p^{j-i+m_{\underline{i},\underline{j}}}e^z}{(1-p^{j-i+m_{\underline{i},\underline{j}}}e^z)^2}-a_{i,j}\ell z^{-2},\\
\label{eq:tauij12}\tau_{ij}^{1,\pm}(z)\ &=\ \tau_{ij}^{2,\pm}(z)\ =\ -[b_{\underline{i},\underline{j}}]_{q^{\frac{\partial}{\partial z}}}q^{\ell\frac{\partial}{\partial z}}\frac{1+p^{j-i+m_{\underline{i},\underline{j}}}e^z}{2-2p^{j-i+m_{\underline{i},\underline{j}}}e^z}-a_{i,j}\ell z^{-1},\\
\label{tij+-}\tau_{ij}^{+,-}(z)\ &=\ z^{-\delta_{i,j}}(z+2\ell\hbar)^{\delta_{i,j}},\\
\label{tij-+}\tau_{ij}^{-,+}(z)\ &=\ z^{-\delta_{i,j}}(z-2\ell\hbar)^{\delta_{i,j}}\frac{q^{b_{\underline{i},\underline{j}}}-p^{i-j+m_{\underline{j},\underline{i}}}e^{-z}}{1-q^{b_{\underline{i},\underline{j}}}p^{i-j+m_{\underline{j},\underline{i}}}e^{-z}},\\
\label{tijpm}\tau_{ij}^{\pm,\pm}(z)\ & =\ z^{\delta_{i,j}-1}\frac{(q^{b_{\underline{i},\underline{j}}}p^{i-j-m_{\underline{i},\underline{j}}}e^{-z})^{-\frac{1}{2}}-(q^{b_{\underline{i},\underline{j}}}p^{i-j-m_{\underline{i},\underline{j}}}e^{-z})^{\frac{1}{2}}}{\left((p^{i-j-m_{\underline{i},\underline{j}}}e^{-z})^{-\frac{1}{2}}-(p^{i-j-m_{\underline{i},\underline{j}}}e^{-z})^{\frac{1}{2}}\right)^{\delta_{i,j}}}.
\end{align}
We also set
\begin{align*}
F(z)\ =\ (q^z-q^{-z})/z=\sum_{n\geq 0}\frac{2\hbar^{2n+1}z^{2n}}{(2n+1)!}\in\hbar\C[z^2][[\hbar]].
\end{align*}

We need the following  technical result.

\begin{lemt}\label{le:tau}
For every $i,j\in\Z$, we have that
\begin{align}
\label{tij-tji}&\tau_{ij}(z)-\tau_{ji}(-z)\ =\ [b_{\underline{i},\underline{j}}]_{q^{\frac{\partial}{\partial z}}}[\ell]_{q^{\frac{\partial}{\partial z}}}\left(q^{\ell\frac{\partial}{\partial z}}-q^{-\ell\frac{\partial}{\partial z}}\right)\frac{p^{i-j+m_{\underline{j},\underline{i}}}e^{-z}}{(1-p^{i-j+m_{\underline{j},\underline{i}}}e^{-z})^2},\\
\label{tij+tji}&\tau_{ij}^{1,\pm}(z)+\tau_{ji}^{2,\pm}(-z)\ =\ [b_{\underline{i},\underline{j}}]_{q^{\frac{\partial}{\partial z}}}\left(q^{\ell\frac{\partial}{\partial z}}-q^{-\ell\frac{\partial}{\partial z}}\right)\frac{1+p^{i-j+m_{\underline{j},\underline{i}}}e^{-z}}{2-2p^{i-j+m_{\underline{j},\underline{i}}}e^{-z}},\\
\label{tijtji}&\tau_{ij}^{\pm,\pm}(z)\tau_{ji}^{\pm,\pm}(-z)^{-1}\ =\ \frac{q^{b_{\underline{i},\underline{j}}}-p^{j-i-m_{\underline{j},\underline{i}}}e^{z}}{1-q^{b_{\underline{i},\underline{j}}}p^{j-i-m_{\underline{j},\underline{i}}}e^{z}}=\tau_{ij}^{\pm,\mp}(z)^{-1}\tau_{ji}^{\mp,\pm}(-z),\\
\label{ftij-tji}&F\left(\frac{\partial}{\partial z}\right)\left(\tau_{ij}(z)-\tau_{ji}(-z)\right)\ =\ \left(q^{-\ell\frac{\partial}{\partial z}}-q^{\ell\frac{\partial}{\partial z}}\right)\left(\tau_{ij}^{1,+}(z)+\tau_{ji}^{2,+}(-z)\right),\\
\label{ftij+tji}&F\left(\frac{\partial}{\partial z}\right)\left(\tau_{ij}^{1,\pm}(z)+\tau_{ji}^{2,\pm}(-z)\right)\ =\ \mp\left(q^{\ell\frac{\partial}{\partial z}}-q^{-\ell\frac{\partial}{\partial z}}\right)\mathrm{log}\frac{\tau_{ij}^{+,\pm}(z)}{\tau_{ji}^{\pm,+}(-z)}.
\end{align}
\begin{proof}
The verifications of the first three equalities are straightforward. For the equality \eqref{ftij-tji}, by using  \eqref{tij-tji} and \eqref{tij+tji}, we have that
\begin{align*}
&\left(q^{-\ell\frac{\partial}{\partial z}}-q^{\ell\frac{\partial}{\partial z}}\right)\left(\tau_{ij}^{1,+}(z)+\tau_{ji}^{2,+}(-z)\right)\\
% =\ &\left(q^{-\frac{\partial}{\partial z}}-q^{\frac{\partial}{\partial z}}\right)[\ell]_{q^{\frac{\partial}{\partial z}}}\left(\tau_{ij}^{1,+}(z)+\tau_{ji}^{2,+}(-z)\right)\\
=\ &-F\left(\frac{\partial}{\partial z}\right)\frac{\partial}{\partial z}[\ell]_{q^{\frac{\partial}{\partial z}}}\left(\tau_{ij}^{1,+}(z)+\tau_{ji}^{2,+}(-z)\right)\\
=\ &F\left(\frac{\partial}{\partial z}\right)[b_{\underline{i},\underline{j}}]_{q^{\frac{\partial}{\partial z}}}[\ell]_{q^{\frac{\partial}{\partial z}}}\left(q^{\ell\frac{\partial}{\partial z}}-q^{-\ell\frac{\partial}{\partial z}}\right)\frac{\partial}{\partial z}[\ell]_{q^{\frac{\partial}{\partial z}}}\frac{1+p^{i-j+m_{\underline{j},\underline{i}}}e^{-z}}{2-2p^{i-j+m_{\underline{j},\underline{i}}}e^{-z}}\\
=\ &F\left(\frac{\partial}{\partial z}\right)[b_{\underline{i},\underline{j}}]_{q^{\frac{\partial}{\partial z}}}[\ell]_{q^{\frac{\partial}{\partial z}}}\left(q^{\ell\frac{\partial}{\partial z}}-q^{-\ell\frac{\partial}{\partial z}}\right)\frac{p^{i-j+m_{\underline{j},\underline{i}}}e^{-z}}{(1-p^{i-j+m_{\underline{j},\underline{i}}}e^{-z})^2}.
\end{align*}

For  the last equality, by using equalities \eqref{ftij-tji} and \eqref{tij-tji}, we have that
\begin{align*}
&\mp\left(q^{\ell\frac{\partial}{\partial z}}-q^{-\ell\frac{\partial}{\partial z}}\right)\textrm{log}\frac{\tau_{ij}^{+,\pm}(z)}{\tau_{ji}^{\pm,+}(-z)}\\
=\ &\mp F\left(\frac{\partial}{\partial z}\right)[\ell]_{q^{\frac{\partial}{\partial z}}}\frac{\partial}{\partial z}\textrm{log}\frac{\tau_{ij}^{+,\pm}(z)}{\tau_{ji}^{\pm,+}(-z)}\\
=\ &F\left(\frac{\partial}{\partial z}\right)[\ell]_{q^{\frac{\partial}{\partial z}}}\frac{p^{j-i-m_{\underline{j},\underline{i}}}e^{z}(q^{b_{\underline{i},\underline{j}}}-q^{-b_{\underline{i},\underline{j}}})}{\left(1-q^{-b_{\underline{i},\underline{j}}}p^{j-i-m_{\underline{j},\underline{i}}}e^{z}\right)\left(1-q^{b_{\underline{i},\underline{j}}}p^{j-i-m_{\underline{j},\underline{i}}}e^{z}\right)}\\
=\ &F\left(\frac{\partial}{\partial z}\right)[\ell]_{q^{\frac{\partial}{\partial z}}}\frac{1}{2}\left(-\frac{1+q^{-b_{\underline{i},\underline{j}}}p^{i-j-m_{\underline{i},\underline{j}}}e^{-z}}{1-q^{-b_{\underline{i},\underline{j}}}p^{i-j-m_{\underline{i},\underline{j}}}e^{-z}}+\frac{1+q^{b_{\underline{i},\underline{j}}}p^{i-j-m_{\underline{i},\underline{j}}}e^{-z}}{1-q^{b_{\underline{i},\underline{j}}}p^{i-j-m_{\underline{i},\underline{j}}}e^{-z}}\right)\\
=\ &F\left(\frac{\partial}{\partial z}\right)[\ell]_{q^{\frac{\partial}{\partial z}}}\left(q^{b_{\underline{i},\underline{j}}\frac{\partial}{\partial z}}-q^{-b_{\underline{i},\underline{j}}\frac{\partial}{\partial z}}\right)\frac{1+p^{i-j-m_{\underline{i},\underline{j}}}e^{-z}}{2-2p^{i-j-m_{\underline{i},\underline{j}}}e^{-z}}\\
=\ &F\left(\frac{\partial}{\partial z}\right)[b_{\underline{i},\underline{j}}]_{q^{\frac{\partial}{\partial z}}}\left(q^{\ell\frac{\partial}{\partial z}}-q^{-\ell\frac{\partial}{\partial z}}\right)\frac{1+p^{i-j+m_{\underline{i},\underline{j}}}e^{-z}}{2-2p^{i-j+m_{\underline{i},\underline{j}}}e^{-z}},
\end{align*}
 as desired.
\end{proof}
\end{lemt}

We introduce a category $\mathcal{M}_{\tau}$ as follows: an object in
$\mathcal{M}_{\tau}$ is a pair
\[(W,(h_{i,\hbar}(z),x_{i,\hbar}^{\pm}(z))_{i\in\Z}),\] where $W$ is a
topologically free $\C[[\hbar]]$-module and $h_{i,\hbar}(z),x_{i,\hbar}^{\pm}(z)\in\mathcal{E}_{\hbar}(W)$ for $i\in\Z$
such that the following conditions hold ($i,j\in\Z$):
\begin{align}
&\label{eq:hh}[h_{i,\hbar}(z_1), h_{j,\hbar}(z_2)]\\
=\ &a_{i,j}\ell\frac{\partial}{\partial z_2}\delta\left(\frac{z_2}{z_1}\right)+\iota_{z_1,z_2}\tau_{ij}(z_1-z_2)-\iota_{z_2,z_1}\tau_{ji}(z_2-z_1),\nonumber\\
%=\ &[b_{\underline{i},\underline{j}}]_{q^{\frac{\partial}{\partial z_2}}}[\ell]_{q^{\frac{\partial}{\partial z_2}}}\left(\iota_{z_1,z_2}q^{-\ell\frac{\partial}{\partial z_2}}-\iota_{z_2,z_1}q^{\ell\frac{\partial}{\partial z_2}}\right)\frac{p^{i-j+m_{\underline{j},\underline{i}}}e^{-z_1+z_2}}{\left(1-p^{i-j+m_{\underline{j},\underline{i}}}e^{-z_1+z_2}\right)^2},\nonumber\\
&\label{eq:hx}[h_{i,\hbar}(z_1),x_{j,\hbar}^{\pm}(z_2)]\\
=\ &\pm x_{j,\hbar}^{\pm}(z_2)\left(a_{i,j}z_{1}^{-1}\delta\left(\frac{z_2}{z_1}\right)+\iota_{z_1,z_2}\tau_{ij}^{1,\pm}(z_1-z_2)+\iota_{z_2,z_1}\tau_{ji}^{2,\pm}(z_2-z_1)\right),\nonumber\\
%=\ & \pm x_{j,\hbar}^{\pm}(z_2)[b_{\underline{i},\underline{j}}]_{q^{\frac{\partial}{\partial z_2}}} \left(\iota_{z_1,z_2}q^{-\ell\frac{\partial}{\partial z_2}}-\iota_{z_2,z_1}q^{\ell\frac{\partial}{\partial z_2}}\right)\frac{1+p^{i-j+m_{\underline{i},\underline{j}}}e^{-z_1+z_2}}{2-2p^{i-j+m_{\underline{i},\underline{j}}}e^{-z_1+z_2}},\nonumber
%&\label{eq:xxpm}\iota_{z_1,z_2}\frac{(q^{b_{\underline{i},\underline{j}}}p^{i-j-m_{\underline{i},\underline{j}}}e^{-z_1+z_2})^{-\frac{1}{2}}-(q^{b_{\underline{i},\underline{j}}}p^{i-j-m_{\underline{i},\underline{j}}}e^{-z_1+z_2})^{\frac{1}{2}}}{\left((p^{i-j-m_{\underline{i},\underline{j}}}e^{-z_1+z_2})^{-\frac{1}{2}}-(p^{i-j-m_{\underline{i},\underline{j}}}e^{-z_1+z_2})^{\frac{1}{2}}\right)^{\delta_{\underline{i},\underline{j}}}}x_{i,\hbar}^{\pm}(z_1)x_{j,\hbar}^{\pm}(z_2)\\
%=\ &\iota_{z_2,z_1}\frac{(q^{b_{\underline{i},\underline{j}}}p^{j-i-m_{\underline{j},\underline{i}}}e^{-z_2+z_1})^{-\frac{1}{2}}-(q^{b_{\underline{i},\underline{j}}}p^{j-i-m_{\underline{j},\underline{i}}}e^{-z_2+z_1})^{\frac{1}{2}}}{\left((p^{j-i-m_{\underline{j},\underline{i}}}e^{-z_2+z_1})^{-\frac{1}{2}}-(p^{j-i-m_{\underline{j},\underline{i}}}e^{-z_2+z_1})^{\frac{1}{2}}\right)^{\delta_{\underline{i},\underline{j}}}}x_{j,\hbar}^{\pm}(z_2)x_{i,\hbar}^{\pm}(z_1),\nonumber\\
% &\label{eq:xxmp}\iota_{z_1,z_2}\tau_{ij}^{\pm,\mp}(z_1-z_2)(z_1-z_2)^{2\delta_{i,j}}x_{i,\hbar}^{\pm}(z_1)x_{j,\hbar}^{\mp}(z_2)\\
% =\ &\iota_{z_2,z_1}\tau_{ji}^{\mp,\pm}(z_2-z_1)(z_1-z_2)^{2\delta_{i,j}}x_{j,\hbar}^{\mp}(z_2)x_{i,\hbar}^{\pm}(z_1).\nonumber
&\label{eq:xxpm}(\iota_{z_1,z_2}\tau_{ij}^{\pm,\pm}(z_1-z_2))(z_1-z_2)^{1-\delta_{i,j}}x_{i,\hbar}^{\pm}(z_1)x_{j,\hbar}^{\pm}(z_2)\\
=\ &(\iota_{z_2,z_1}\tau_{ji}^{\pm,\pm}(z_2-z_1))(z_1-z_2)^{1-\delta_{i,j}}x_{j,\hbar}^{\pm}(z_2)x_{i,\hbar}^{\pm}(z_1),\nonumber\\
%& (z_1-z_2)^{\delta_{i,j}}(z_1-z_2+2\ell\hbar)^{\delta_{i,j}}\\
%&\cdot\left(x_{i,\hbar}^{\pm}(z_1)x_{j,\hbar}^{\pm}(z_2)-\iota_{z_2,z_1}g_{ij,q}(p^{m_{\underline{i},\underline{j}}+j-i}e^{z_1-z_2}) x_{j,\hbar}^{\pm}(z_2)x_{i,\hbar}^{\pm}(z_1)\right)=0,\nonumber\\
&\label{eq:xxmp}(\iota_{z_1,z_2}\tau_{ij}^{\pm,\mp}(z_1-z_2))(z_1-z_2)^{2\delta_{i,j}}x_{i,\hbar}^{\pm}(z_1)x_{j,\hbar}^{\mp}(z_2)\\
=\ &(\iota_{z_2,z_1}\tau_{ji}^{\mp,\pm}(z_2-z_1))(z_1-z_2)^{2\delta_{i,j}}x_{j,\hbar}^{\mp}(z_2)x_{i,\hbar}^{\pm}(z_1).\nonumber
\end{align}
A morphism \[f\colon (W,(h_{i,\hbar}(z),x_{i,\hbar}^{\pm}(z))_{i\in\Z})\to (W',(h'_{i,\hbar}(z),(x_{i,\hbar}^{\pm})'(z))_{i\in\Z})\] of objects in $\mathcal{M}_\tau$ is a $\C[[\hbar]]$-module map
\[f:\ W\to W',\]
such  that
\[f\circ h_{i,\hbar}(z)=h'_{i,\hbar}(z)\circ f\quad\textrm{and}\quad f\circ x_{i,\hbar}^{\pm}(z)=(x_{i,\hbar}^{\pm})'(z)\circ f\quad\textrm{for }i\in\Z.\]

 Let $\mathcal{F}_{\tau}$ be the forgetful functor from $\mathcal{M}_{\tau}$ to the category of topologically free $\C[[\hbar]]$-modules. Define $\textrm{End}_{\C[[\hbar]]}(\mathcal{F}_{\tau})$ to be the algebra of endomorphisms of the functor $\mathcal{F}_{\tau}$. For each $W\in\mathcal{M}_{\tau}$, $\textrm{End}_{\C[[\hbar]]}(W)$ is a topological algebra over $\C[[\hbar]]$ such that
\[\{(K:\hbar^nW)\mid K\subset W,|K|<\infty,n\in\Z_{+}\}\]
forms a local basis at $0$, where
\[(K:\hbar^nW)=\{\varphi\in \textrm{End}_{\C[[\hbar]]}(W)\mid \varphi(K)\subset \hbar^nW\}.\]
Equip $\textrm{End}_{\C[[\hbar]]}(\mathcal{F}_{\tau})$ with the coarsest topology such that for every $W\in\mathcal{M}_{\tau}$, the canonical $\C[[\hbar]]$-algebra epimorphism from $\textrm{End}_{\C[[\hbar]]}(\mathcal{F}_{\tau})$ to $\textrm{End}_{\C[[\hbar]]}(W)$ is continuous. One can check that both $\textrm{End}_{\C[[\hbar]]}(\mathcal{F}_{\tau})$ and $\textrm{End}_{\C[[\hbar]]}(W)$ are topologically free as $\C[[\hbar]]$-modules (cf. \cite[Section 6.2]{CJKT}).

For each $i\in\Z$, we define endomorphisms $h_{i,\hbar}(n)$ and $x_{i,\hbar}^{\pm}(n)$ $(n\in\Z)$  of $\mathcal{F}_{\tau}$ as follows:
\[\sum_{n\in\Z}h_{i,\hbar}(n).vz^{-n-1}=h_{i,\hbar}(z)v\quad \textrm{and}\quad\sum_{n\in\Z}x_{i,\hbar}^{\pm}(n).vz^{-n-1}=x_{i,\hbar}^{\pm}(z)v,\]
where $v\in W$ and $(W,(h_{i,\hbar}(z),x_{i,\hbar}^{\pm}(z))_{i\in\Z})$ is an object of $\mathcal{M}_{\tau}$. Denote by $\mathcal{A}$ the closed subalgebra  of $\textrm{End}_{\C[[\hbar]]}(\mathcal{F}_{\tau})$ generated by
\[\{h_{i,\hbar}(n),x_{i,\hbar}^{\pm}(n)\mid i,n\in\Z \}.\]
Then, as a $\C[[\hbar]]$-module, $\mathcal{A}$ is topologically free (cf. \cite[Section 6.2]{CJKT}).

For a topologically free $\C[[\hbar]]$-module $V$ and a submodule $M$ of $V$, we recall the following notation given in \cite[Definition 3.4]{Li4}:
 \[[M]=\{v\in V\mid \hbar^nv\in M \textrm{ for some }n\in\N\}.\]Let $\mathcal{A}_{+}$ be the minimal closed left ideal of $\mathcal{A}$ such that  $h_{i,\hbar}(n)$ and $x_{i,\hbar}^{\pm}(n)$ ($i\in\Z$, $n\ge 0$) are contained in $\mathcal{A}_{+}$ and that $[\mathcal{A}_{+}]=\mathcal{A}_{+}$. Define
\begin{align}\label{ftau}
F_{\tau}(A,\ell)=\mathcal{A}/\mathcal{A}_{+}.
\end{align}
Set ${\bf 1}=1+\mathcal{A}_{+}\in F_{\tau}(A,\ell)$. For each $i\in\Z$, we identify $h_{i,\hbar}$ and $x_{i,\hbar}^{\pm}$ with $h_{i,\hbar}(-1){\bf 1}$ and $x_{i,\hbar}^{\pm}(-1){\bf 1}$ in $F_{\tau}(A,\ell)$, respectively.

 It is proved in \cite[Proposition 5.3]{K} that  there exists an $\hbar$-adic quantum vertex algebra structure on ${F}_{\tau}(A,\ell)$ such that   the vacuum vector is ${\bf 1}$, the vertex operator map $Y_{\tau}$ is determined by
\begin{align}\label{ytau}
Y_{\tau}(h_{i,\hbar},z)=h_{i,\hbar}(z)\quad \textrm{and}\quad Y_{\tau}(x_{i,\hbar}^{\pm},z)=x_{i,\hbar}^{\pm}(z)\quad\textrm{for }i\in\Z,
\end{align}
and the quantum Yang-Baxter operator $S_{\tau}(z)$ is determined by
\begin{align}
\label{eq:stauhh}S_{\tau}(z)(h_{j,\hbar}\otimes h_{i,\hbar})\ &=\ h_{j,\hbar}\otimes h_{i,\hbar}+{\bf 1}\otimes {\bf 1}\otimes\left(\tau_{ij}(-z)-\tau_{ji}(z)\right),\\
\label{eq:stauxh}S_{\tau}(z)(x_{j,\hbar}^{\pm}\otimes h_{i,\hbar})\ &=\ x_{j,\hbar}^{\pm}\otimes h_{i,\hbar}+x_{j,\hbar}^{\pm}\otimes {\bf 1}\otimes\left(\tau_{ij}^{1,\pm}(-z)+\tau_{ji}^{2,\pm}(z)\right),\\
\label{eq:stauhx}S_{\tau}(z)(h_{j,\hbar}\otimes x_{i,\hbar}^{\pm})\ &=\ h_{j,\hbar}\otimes x_{i,\hbar}^{\pm} +{\bf 1}\otimes x_{i,\hbar}^{\pm}\otimes \left(\tau_{ij}^{2,\pm}(-z)+\tau_{ji}^{1,\pm}(z)\right),\\
\label{eq:stauxx}S_{\tau}(z)(x_{j,\hbar}^{\epsilon_1}\otimes x_{i,\hbar}^{\epsilon_2})\ &=\ x_{j,\hbar}^{\epsilon_1}\otimes x_{i,\hbar}^{\epsilon_2}\otimes\tau_{ji}^{\epsilon_1,\epsilon_2}(z)\tau_{ij}^{\epsilon_2,\epsilon_1}(-z)^{-1}
\end{align}
for $i,j\in\Z$ and $\epsilon_1,\epsilon_2=\pm $.

We also need the following result, which  is proved in \cite[Proposition 5.4]{K}.
\begin{prpt}\label{le:uniofFtau}
Let $V$ be an $\hbar$-adic nonlocal vertex algebra, and let $Y$ be the vertex operator map of $V$. Suppose that there exist $\bar{h}_{i,\hbar}$, $\bar{x}_{i,\hbar}^{\pm}\in V$ $(i\in\Z)$ such that $(V,(Y(\bar{h}_{i,\hbar},z), Y(\bar{x}_{i,\hbar}^{\pm},z))_{i\in\Z})$ is an object of $\mathcal{M}_{\tau}$. Then there exists a unique $\hbar$-adic nonlocal vertex algebra homomorphism $\varphi:F_{\tau}(A,\ell)\to V$ such that
\[\varphi(h_{i,\hbar})=\bar{h}_{i,\hbar}\quad\textrm{and}\quad\varphi(x_{i,\hbar}^{\pm})=\bar{x}_{i,\hbar}^{\pm}\quad\textrm{for } i\in\Z.\]
\end{prpt}

Form the quotient $\hbar$-adic nonlocal vertex algebra
\begin{align}
V_{\widehat{\mathfrak{sl}}_{\infty},\hbar}(\ell,0)=F_{\tau}(A,\ell)/R_{\ell},
\end{align}
where $R_{\ell}$ is the minimal closed ideal of $F_{\tau}(A,\ell)$ which satisfies $[R_{\ell}]=R_{\ell}$ and   contains the following elements
\begin{align}
&\label{eq:Rell1}(x_{i,\hbar}^{+})_{0}x_{i,\hbar}^{-}-(q-q^{-1})^{-1}({\bf 1}-E(h_{i,\hbar}))\quad \textrm{for }i\in\Z,\\
&\label{eq:Rell2}(x_{i,\hbar}^{+})_{1}x_{i,\hbar}^{-}+2\ell\hbar(q-q^{-1})^{-1}E(h_{i,\hbar})\quad \textrm{for }i\in\Z,\\
&\label{eq:Rell3}(x_{i,\hbar}^{\pm})_{0}^{1-a_{i,j}}x_{j,\hbar}^{\pm} \quad\textrm{for }i,j\in\Z \textrm{ with }a_{i,j}\le 0.
\end{align}
Here,
\begin{equation}
\begin{split}
E(h_{i,\hbar})\ &=\ \left( \frac{F(1+\ell)}{F(1-\ell)}\right)^{\frac{1}{2}}\textrm{exp}\left(\left(-q^{-\ell\partial}F(\partial)h_{i,\hbar}\right)_{-1}\right){\bf 1}\in F_{\tau}(A,\ell).
\end{split}
\end{equation}

Write $U_{\tau}=\{h_{i,\hbar},x_{i,\hbar}^{\pm}\mid i\in\Z\}$, which is a generating subset of $F_{\tau}(A,\ell)$.
By a similar argument as that in the proofs of \cite[Proposition 6.16, Lemma 6.17, Lemma 6.18]{K}, we obtain from Lemma \ref{le:tau} that
\begin{align*}
S_{\tau}(z)(R_{\ell}\otimes U_{\tau})\quad\textrm{and}\quad S_{\tau}(z)(U_{\tau}\otimes R_{\ell})
\end{align*}
is a subset of
\[R_{\ell}\widehat{\otimes}F_{\tau}(A,\ell)\widehat{\otimes}\C((z))[[\hbar]]+F_{\tau}(A,\ell)\widehat{\otimes}R_{\ell}\widehat{\otimes}\C((z))[[\hbar]],\]
 Then the quantum Yang-Baxter operator $S_{\tau}(z)$ of  $F_{\tau}(A,\ell)$ induces a $\C[[\hbar]]$-module map as follows:
\begin{align}\label{eq:qyboV}
S_{\tau}(z):\ V_{\widehat{\mathfrak{sl}}_{\infty},\hbar}(\ell,0)\widehat{\otimes}V_{\widehat{\mathfrak{sl}}_{\infty},\hbar}(\ell,0)\to  V_{\widehat{\mathfrak{sl}}_{\infty},\hbar}(\ell,0)\widehat{\otimes}V_{\widehat{\mathfrak{sl}}_{\infty},\hbar}(\ell,0)\widehat{\otimes}\C((z))[[\hbar]].
\end{align}
It follows from \cite[Lemma 3.3]{K}  that the  $\hbar$-adic nonlocal vertex algebra $V_{\widehat{\mathfrak{sl}}_{\infty},\hbar}(\ell,0)$, together with the $\C[[\hbar]]$-module map $S_{\tau}(z)$, is an $\hbar$-adic quantum vertex algebra.

On the other hand,
 for every $n\in \Z$, note that \[\left(F_{\tau}(A,\ell), (Y_{\tau}(h_{i+nN,\hbar},z),Y_{\tau}(x_{i+nN,\hbar}^{\pm},z))_{i\in\Z}\right)\]
 is  an object of $\mathcal{M}_{\tau}$. By Proposition \ref{le:uniofFtau}, for every $n\in\Z$, there is an $\hbar$-adic nonlocal vertex algebra homomorphism  \[\rho_{n,\hbar}:\ F_{\tau}(A,\ell)\to F_{\tau}(A,\ell)\]  such that
 \begin{align}\label{eq:rhohn}
\rho_{n,\hbar}(h_{i,\hbar})=h_{i+nN,\hbar}\quad\textrm{and}\quad \rho_{n,\hbar}(x_{i,\hbar}^{\pm})=x_{i+nN,\hbar}^{\pm}\quad \textrm{for }i\in\Z.
\end{align}
 Since $\{h_{i,\hbar},x_{i,\hbar}^{\pm}\mid i\in\Z\}$ is a generating subset of $F_{\tau}(A,\ell)$, we have that $\rho_{n,\hbar}$ is an automorphism of the $\hbar$-adic quantum vertex algebra $F_{\tau}(A,\ell)$.
As $\rho_{n,\hbar}$ preserves the ideal $R_{\ell}$, it induces an automorphism, still call $\rho_{n,\hbar}$, of $V_{\widehat{\mathfrak{sl}}_{\infty},\hbar}(\ell,0)$.
Then we obtain a group homomorphism
\begin{align}\label{eq:grouphomoV}
\rho_{\hbar}:\ \Z\to \textrm{Aut}(V_{\widehat{\mathfrak{sl}}_{\infty},\hbar}(\ell,0)),\quad n\mapsto\rho_{n,\hbar}.
\end{align}

In summand, we obtain the following result.

% for every $v\in F_{\tau}(A,\ell)$, we still use $v$ to denote its image in $V_{\widehat{\mathfrak{sl}}_{\infty},\hbar}(\ell,0)$.
\begin{thmt}\label{vgl}
 The  $\hbar$-adic nonlocal vertex algebra $V_{\widehat{\mathfrak{sl}}_{\infty},\hbar}(\ell,0)$, together with the $\C[[\hbar]]$-module map $S_{\tau}(z)$ as in \eqref{eq:qyboV} and the group homomorphism $\rho_{\hbar}$ as in \eqref{eq:grouphomoV},
forms an $\hbar$-adic quantum $\Z$-vertex algebra.
\end{thmt}

In what follows, we show that the $\hbar$-adic quantum $\Z$-vertex algebra $(V_{\widehat{\mathfrak{sl}}_{\infty},\hbar}(\ell,0),\rho_{\hbar})$ is a deformation of the $\Z$-vertex algebra $(V_{\widehat{\mathfrak{sl}}_{\infty}}(\ell,0),\sigma)$.
 Note that the natural $\hbar$-adic vertex algebra homomorphism from $F_{\tau}(A,\ell)$ to $V_{\widehat{\mathfrak{sl}}_{\infty},\hbar}(\ell,0)$ induces the following  surjective  nonlocal vertex algebra homomorphism:
\begin{align}\label{eq:quomapF}
F_{\tau}(A,\ell)/\hbar F_{\tau}(A,\ell) \to V_{\widehat{\mathfrak{sl}}_{\infty},\hbar}(\ell,0)/\hbar V_{\widehat{\mathfrak{sl}}_{\infty},\hbar}(\ell,0).
\end{align}
On the other hand, it is proved in \cite[Proposition 5.12]{K} that $F_{\tau}(A,\ell)/\hbar F_{\tau}(A,\ell)$ is a vertex algebra, and so $V_{\widehat{\mathfrak{sl}}_{\infty},\hbar}(\ell,0)/\hbar V_{\widehat{\mathfrak{sl}}_{\infty},\hbar}(\ell,0)$ is a vertex algebra as well.
Furthermore, in view of Theorem \ref{vgl}, we have that \[(V_{\widehat{\mathfrak{sl}}_{\infty},\hbar}(\ell,0)/\hbar V_{\widehat{\mathfrak{sl}}_{\infty},\hbar}(\ell,0),\rho_{\hbar}^{(1)})\] is a $\Z$-vertex algebra, where
\begin{align*}
\rho_{\hbar}^{(1)}:\ &\Z\to \textrm{Aut}(V_{\widehat{\mathfrak{sl}}_{\infty},\hbar}(\ell,0)/\hbar V_{\widehat{\mathfrak{sl}}_{\infty},\hbar}(\ell,0)),\\
&n\mapsto \left(\rho_{n,\hbar}^{(1)}:\ v+\hbar V_{\widehat{\mathfrak{sl}}_{\infty},\hbar}(\ell,0)\mapsto \rho_{n,\hbar}(v)+\hbar V_{\widehat{\mathfrak{sl}}_{\infty},\hbar}(\ell,0)\right)
\end{align*}
is a group homomorphism induced from $\rho_{\hbar}$.

\begin{prpt}
The assignment ($i\in \Z$)
\[H_i\mapsto h_{i,\hbar}+\hbar V_{\widehat{\mathfrak{sl}}_{\infty},\hbar}(\ell,0)\quad \textrm{and}\quad X_i^{\pm}\mapsto x_{i,\hbar}^{\pm}+\hbar V_{\widehat{\mathfrak{sl}}_{\infty},\hbar}(\ell,0)\]
determines a surjective $\Z$-vertex algebra homomorphism from $$(V_{\widehat{\mathfrak{sl}}_{\infty}}(\ell,0),\sigma)\to(V_{\widehat{\mathfrak{sl}}_{\infty},\hbar}(\ell,0)/\hbar V_{\widehat{\mathfrak{sl}}_{\infty},\hbar}(\ell,0),\rho_{\hbar}^{(1)}).$$
\end{prpt}
\begin{proof}
It is proved in \cite[Proposition 5.12]{K} that
\[F(A,\ell)\cong F_{\tau}(A,\ell)/\hbar F_{\tau}(A,\ell)\] as vertex algebras, where   an isomorphism is determined by
\begin{align}\label{eq:isoFFtau}
h_i\mapsto h_{i,\hbar}+\hbar F_{\tau}(A,\ell)\quad\textrm{and}\quad x_i^{\pm}\mapsto x_{i,\hbar}^{\pm}+\hbar F_{\tau}(A,\ell)\quad\textrm{for }i\in\Z.
\end{align}
 This together with \eqref{eq:quomapF} induces a   surjective  vertex algebra homomorphism
\[\pi:\ F(A,\ell)\to V_{\widehat{\mathfrak{sl}}_{\infty},\hbar}(\ell,0)/\hbar V_{\widehat{\mathfrak{sl}}_{\infty},\hbar}(\ell,0). \]
Recall that  $I_2$ is an ideal of vertex algebra $F(A,\ell)$ generated by the elements in \eqref{eq:genI2}.
Then  $\pi$ factors through $F(A,\ell)/I_2$ by the definition of $V_{\widehat{\mathfrak{sl}}_{\infty},\hbar}(\ell,0)$. This together with Proposition \ref{pr:vasl} implies that there is a surjective vertex algebra homomorphism
\[\psi:\ V_{\widehat{\mathfrak{sl}}_{\infty}}(\ell,0)\to V_{\widehat{\mathfrak{sl}}_{\infty},\hbar}(\ell,0)/\hbar V_{\widehat{\mathfrak{sl}}_{\infty},\hbar}(\ell,0),\]
which satisfies the condition that
\begin{align*}
\psi\circ \sigma_n(v)=\rho_{n,\hbar}^{(1)}\circ \psi(v)
\end{align*}
for $n\in\Z$ and $v\in \{H_i,X_i^{\pm }\mid i\in\Z\}$ by \eqref{eq:sigman}, \eqref{eq:rhohn} and  \eqref{eq:isoFFtau}.
 Since
 \[\{H_i,X_i^{\pm }\mid i\in\Z\}\] is a generating subset of $V_{\widehat{\mathfrak{sl}}_{\infty}}(\ell,0)$, we get that $\psi$ is a surjective $\Z$-vertex algebra homomorphism, as desired.

\end{proof}

\section{Correspondences between representations of $\mathcal{E}_N$ and $V_{\widehat{\mathfrak{sl}}_{\infty},\hbar}(\ell,0)$}
 In this section, we establish an isomorphism between the category of restricted $\mathcal{E}_N$-modules of level $\ell$ and the category of equivariant $\phi$-coordinated quasi modules for the $\hbar$-adic quantum $\Z$-vertex algebra $(V_{\widehat{\mathfrak{sl}}_{\infty},\hbar}(\ell,0),\rho_{\hbar})$.

\subsection{Equivariant $\phi$-coordinated quasi module}
Set
\[\phi=\phi(z_1,z)=z_1e^z\in\C((z_1))[[z]],\]
and let $G$ be a group equipped with a linear character $\chi:G\to \C^{\times}$. Here we recall some  basic notions and
results related to $(G,\chi)$-equivariant $\phi$-coordinated quasi modules for $\hbar$-adic nonlocal $G$-vertex algebras.

 We begin with the notion of $\phi$-coordinated quasi module for a nonlocal vertex algebra (see \cite{Li5}).

 \begin{dfnt}
Let $V_0$ be a nonlocal vertex algebra over a commutative ring $R$. A $\phi$-coordinated quasi $V_0$-module is a $R$-module $W_0$   equipped with a $R$-linear map
 \[Y_{W_0}^{\phi}(\cdot,z):\ V_0\to \mathrm{Hom}_{R}(W_0,W_0((z))),\]
 satisfying the condition that
 $Y_{W_0}^{\phi}({\bf 1},z)=1_{W_0}$,
 and for every $u,v\in V_0$, there exists a nonzero polynomial $p(z_1,z_2)\in R[z_1,z_2]$ such that
\[p(z_1,z_2)Y_{W_0}^{\phi}(u,z_1)Y_{W_0}^{\phi}(v,z_2)\in \mathrm{Hom}_{R}(W_0,W_0((z_1,z_2))),\]
 \[\left(p(z_1,z_2)Y_{W_0}^{\phi}(u,z_1)Y_{W_0}^{\phi}(v,z_2)\right)|_{z_1=\phi(z_2,z_0)}=p(\phi(z_2,z_0),z_2)Y_{W_0}^{\phi}(Y(u,z_0)v,z_2).\]
 \end{dfnt}
 For a subgroup $\Gamma$ of $\C^{\times}$, set
 \begin{align}
 \C_{\Gamma}[z]=\langle z-\alpha \mid \alpha\in \Gamma\rangle,
 \end{align}
 the multiplicative monoid generated by $z-\alpha$ ($\alpha\in \Gamma$) in $\C[z]$. We also set
 \begin{align}
 \C_{\Gamma}[z_1,z_2]=\langle z_1-\alpha z_2 \mid \alpha\in \Gamma\rangle,
 \end{align}
 the multiplicative monoid generated by $z_1-\alpha z_2$ ($\alpha\in \Gamma$) in $\C[z_1,z_2]$.

 The following notion is introduced in \cite{Li6}.
 \begin{dfnt}\label{df:equimod}
 Let $(V_0,\mathcal{R})$ be a nonlocal $G$-vertex algebra over $\C[[\hbar]]/\hbar^n\C[[\hbar]]$ for some positive integer $n$.
 A $(G,\chi)$-equivariant $\phi$-coordinated quasi $V_0$-module is a $\phi$-coordinated quasi $V_0$-module $(W_0,Y_{W_0}^{\phi})$
 satisfying the condition that
 \begin{align}\label{eq:modequicon}
 Y_{W_0}^{\phi}(\mathcal{R}_gv,z)=Y_{W_0}^{\phi}(v,\chi(g)^{-1}z)\quad \textrm{for }g\in G,v\in V_0,
 \end{align}
 and that for $u,v\in V_0$, there exists $q(z_1,z_2)\in\C_{\chi(G)}[z_1,z_2]$ such that
 \[q(z_1,z_2)Y_{W_0}^{\phi}(u,z_1)Y_{W_0}^{\phi}(v,z_2)\in\mathrm{Hom}_{\C[[\hbar]]/\hbar^n\C[[\hbar]]}(W_0,W_0((z_1,z_2))).\]
 \end{dfnt}

As an $\hbar$-adic analog of Definition \ref{df:equimod}, we also have the following notion (see \cite{JKLT2}). Recall that for every $\hbar$-adic nonlocal $G$-vertex algebra $(V,\mathcal{R})$ and $n\in \Z_+$, $(V/\hbar^n V,\mathcal{R}^{(n)})$ is a  nonlocal $G$-vertex algebra over $\C[[\hbar]]/\hbar^n\C[[\hbar]]$.
%efinition \ref{df:equimod} (see \cite{JKLT2}).
\begin{dfnt}
Let $(V,\mathcal{R})$ be an $\hbar$-adic nonlocal (resp.\,quantum) $G$-vertex algebra. A $(G,\chi)$-equivariant $\phi$-coordinated quasi $V$-module is a topologically free $\C[[\hbar]]$-module $W$ equipped with a $\C[[\hbar]]$-module map
\begin{align}\label{eq:YWphi}
Y_W^{\phi}(\cdot,z):\ V\to\mathcal{E}_{\hbar}(W)
\end{align}
such that for every positive integer $n$, $W/\hbar^nW$ is a $(G,\chi)$-equivariant $\phi$-coordinated quasi $V/\hbar^nV$-module.
\end{dfnt}

For convenience,  in the rest of this subsection, we assume that the character $\chi$ is injective.
 The following result is proved in \cite[Proposition 3.35]{JKLT2}.

\begin{lemt}\label{le:equimod}
Let $(V,\mathcal{R})$ be an $\hbar$-adic nonlocal $G$-vertex algebra and let $(W,Y_W^{\phi})$ be a $(G,\chi)$-equivariant $\phi$-coordinated quasi $V$-module. Assume that
\[u,v\in V,~u^{(i)},v^{(i)}\in V,\; f_{i}(z)\in\C(z)[[\hbar]]\quad (1\leq i\leq r)\]
satisfying that
\[Y(u,z_1)Y(v,z_2)\sim \sum_{i=1}^r\iota_{z_2,z_1}\left( f_{i}(e^{z_2-z_1})\right)Y(v^{(i)},z_2)Y(u^{(i)},z_1).\]
Assume further that for each $n\in\Z_{+}$, $(\mathcal{R}_\sigma u)_jv\in\hbar^nV$ for all but finite many $(\sigma,j)\in G\times \N$. Then
\begin{eqnarray*}
&& Y_W^{\phi}(u,z_1)Y_W^{\phi}(v,z_2)-\sum_{i=1}^r\iota_{z_2,z_1}\left(f_i(z_2/z_1)\right)Y_W^{\phi}(v^{(i)},z_2)Y_W^{\phi}(u^{(i)},z_1)\\
&=&\sum_{\sigma\in G}\sum_{j\geq 0}Y_W^{\phi}((R_\sigma u)_jv,z_2)\frac{1}{j!}\left(z_2\frac{\partial}{\partial z_2}\right)^j\delta\left(\chi(\sigma)^{-1}\frac{z_2}{z_1}\right).
\end{eqnarray*}
\end{lemt}
 %assume that $G$ is a subgroup of $\C^{\times}$ and take $\chi$ to be the natural embedding of $G$ into $\C^{\times}$, i.e.,
% \[\chi(g)=g\quad \textrm{for }g\in G.\]

 In what follows, we  introduce a general construction of $\hbar$-adic nonlocal $G$-vertex algebras. Let $W$ be a fixed topologically free $\C[[\hbar]]$-module. For $n,k\in\Z_{+}$, the quotient map from $W$ to $W/ \hbar^n W$ induces a $\C[[\hbar]]$-module map
\begin{align*}
\pi_n^{(k)}:\ (\textrm{End}_{\C[[\hbar]]}W)[[z_1^{\pm 1},z_2^{\pm 1},\dots,z_k^{\pm 1}]]\to (\textrm{End}_{\C[[\hbar]]}W/ \hbar^{n}W)[[z_1^{\pm 1},z_2^{\pm 1},\dots,z_k^{\pm 1}]].
\end{align*}
For convenience, we write $\pi_n=\pi_n^{(1)}$.
\begin{dfnt}
 A subset $U$ of $\mathcal{E}_{\hbar}(W)$ is said to be $\hbar$-adically $(G,\chi)$-quasi compatible if for every $a_{1}(z),a_{2}(z),\dots,a_k(z)\in\mathcal{E}_{\hbar}(W)$ $(k\in\Z_{+})$ and every positive integer $n$, there exists $f(z_1,z_2)\in\C_{\chi(G)}[z_1,z_2]$ such that
\begin{align*}
\left(\prod_{1\leq i<j\leq k}f(z_i,z_j)\right)\pi_n(a_1(z_1))\pi_n(a_2(z_2))\cdots\pi_n(a_k(z_k))\in \mathcal{E}^{(k)}(W/\hbar^nW),
\end{align*}
where
\[\mathcal{E}^{(k)}(W/\hbar^nW)=\mathrm{Hom}_{\C}(W/\hbar^nW,W/\hbar^nW((z_1,z_2,\dots,z_k))).\]
In addition, we say that a subset $U$ of $\mathcal{E}_{\hbar}(W)$ is $\hbar$-adically quasi compatible if $U$ is $\hbar$-adically $(\C^{\times},1)$-quasi compatible.
\end{dfnt}

 Let $(a(z),b(z))$ be an $\hbar$-adically $(G,\chi)$-quasi compatible pair in $\mathcal{E}_{\hbar}(W)$. Then for every positive integer $n$, there exists $f_n(z_1,z_2)\in\C_{\chi(G)}[z_1,z_2]$ such that
\[f_n(z_1,z_2)\pi_n(a(z_1))\pi_n(b(z_2))\in \mathcal{E}^{(2)}(W/\hbar^nW).\]
Recall the following vertex operator introduced in \cite{JKLT2}:
\begin{align*}
Y_{\mathcal{E}}^{\phi}(a(z),z_0)b(z)=&\ \sum_{n\in\Z}a(z)_{n}^{\phi}b(z)z_0^{-n-1}\\
= &\  \varprojlim_{n>0}(f_n(\phi(z,z_0),z))^{-1}\left( f_n(z_1,z)\pi_n(a(z_1))\pi_n(b(z))\right)|_{z_1=\phi(z,z_0)}.\nonumber
\end{align*}

Suppose now that $U$ is an $\hbar$-adically quasi compatible subset of $\mathcal{E}_{\hbar}(W)$. In view of \cite[Theorem 3.31]{JKLT2}, there exists a minimal $\hbar$-adically quasi compatible $\C[[\hbar]]$-submodule $\langle U\rangle_{\phi}$, which satisfies the conditions that
 \begin{itemize}
\item $\langle U\rangle_{\phi}$ is topologically free, $[\langle U\rangle_{\phi}]=\langle U\rangle_{\phi}$, and
\item $\langle U\rangle_{\phi}$ is $Y_{\mathcal{E}}^{\phi}$-closed, i.e., for $a(z),b(z)\in\langle U\rangle_{\phi}$ and $n\in \Z$, $a(z)_{n}^{\phi}b(z)\in\langle U\rangle_{\phi}$.
\end{itemize}
Furthermore, $(\langle U\rangle_{\phi},Y_{\mathcal{E}}^{\phi},1_{W})$ carries an $\hbar$-adic nonlocal vertex algebra structure such that $W$ is a faithful $\phi$-coordinated quasi $\langle U\rangle_{\phi}$-module with
\[Y_W^{\phi}(a(z),z_0)=a(z_0)\] for $a(z)\in\langle U\rangle_{\phi}$.

Consider the group homomorphism
\begin{align}\label{gphomo}
\mathcal{R}:G\to\mathrm{GL}(\mathcal{E}_{\hbar}(W)),\quad g\mapsto \left(\mathcal{R}_{g}: a(z)\mapsto a(\chi(g)^{-1}z)\right).
\end{align}
We have the following result, which
 is an $\hbar$-adic analog of  \cite[Theorem 3.11]{JKLT1}.

\begin{prpt}\label{consthm}
 Let  $U$ be an $\hbar$-adically $(G,\chi)$-quasi compatible subset of $\mathcal{E}_{\hbar}(W)$ such that $U$ is $G$-stable.  Then the $\hbar$-adic nonlocal vertex algebra $\langle U\rangle_{\phi}$, together with the group homomorphism $\mathcal{R}$  defined in \eqref{gphomo}, is an  $\hbar$-adic  nonlocal $G$-vertex algebra.
 Furthermore, $W$ is a faithful $(G,\chi)$-equivariant $\phi$-coordinated quasi $\langle U\rangle_{\phi}$-module with \[Y_W^{\phi}(a(z),z_0)=a(z_0)\quad \textrm{for }a(z)\in\langle U\rangle_{\phi}.\]
\end{prpt}
\begin{proof}
One concludes from the construction of $\langle U\rangle_{\phi}$ (see \cite[Theorem 3.31]{JKLT2}) that in this case  $\langle U\rangle_{\phi}$ is $\hbar$-adically $(G,\chi)$-quasi compatible as well. For every positive integer $n$, there is a canonical group homomorphism
 \begin{align*}
 \mathcal{R}^{(n)}:\ G\to \mathrm{GL}(\mathcal{E}^{(1)}(W/\hbar^n W)),\quad g\mapsto \mathcal{R}_g^{(n)}
 \end{align*}
 induced by $\mathcal{R}$ (see \eqref{gphomo}).
 Following \cite[Theorem 3.11]{JKLT1}, it is straightforward to see that $\langle U\rangle_{\phi}/\hbar^n\langle U\rangle_{\phi}$ together with the group homomorphism $\mathcal{R}^{(n)}$ forms a nonlocal $G$-vertex algebra (over $\C[[\hbar]]/\hbar^n\C[[\hbar]]$) and that $W/\hbar^n W$ is a $(G,\chi)$-equivariant $\phi$-coordinated quasi $\langle U\rangle_{\phi}/\hbar^n\langle U\rangle_{\phi}$-module with \[Y_{W/\hbar^nW}^{\phi}(a(z),z_0)=a(z_0)\quad
\textrm{for } a(z)\in\langle U\rangle_{\phi}/\hbar^n\langle U\rangle_{\phi} .\]
 Then the assertion follows from \cite[Remark 3.22]{JKLT2}.
\end{proof}
We also need the following result, which is proved in \cite[Proposition 3.33]{JKLT2}.
\begin{lemt}\label{le:Yephi}
Let $V$ be an $\hbar$-adically quasi compatible subset of $\mathcal{E}_{\hbar}(W)$ such that $V=\langle V\rangle_{\phi}$. Let $r,s\in\Z_{+}$, and let
\[a_i(z),b_i(z),\alpha_j(z),\beta_j(z), \gamma_k(z)\in V,~f_i(z),g_j(z)\in\C(z)[[\hbar]]\]
for $1\leq i\leq r$, $1\leq j\leq s$, $k\in\N$ such that $\sum_{k=0}^{\infty}\gamma_k(z)\in V$ and that
\begin{align*}
&\sum_{i=1}^r\iota_{z_1,z_2}\left(f_i(z_1/z_2)\right)a_i(z_1)b_i(z_2)-\sum_{j=1}^{s}\iota_{z_2,z_1}\left(g_j(z_1/z_2)\right)\alpha_j(z_2)\beta_j(z_1)\\
=\ &\sum_{k=0}^{\infty}\gamma_{k}(z)\frac{1}{k!}\left(z_2\frac{\partial}{\partial z_2}\right)^k\delta\left(\frac{z_2}{z_1}\right).
\end{align*}
Then we have that
\begin{align*}
&\sum_{i=1}^r\iota_{z_1,z_2}\left(f_i(e^{z_1-z_2})\right)Y_{\mathcal{E}}^{\phi}(a_i(z),z_1)Y_{\mathcal{E}}^{\phi}(b_i(z),z_2)\\
\ &-\sum_{j=1}^{s}\iota_{z_2,z_1}\left(g_j(e^{z_1-z_2})\right)Y_{\mathcal{E}}^{\phi}(\alpha_j(z),z_2)Y_{\mathcal{E}}^{\phi}(\beta_j(z),z_1)\\
=\ &\sum_{k=0}^{\infty}Y_{\mathcal{E}}^{\phi}(\gamma_k(z),z_2)\frac{1}{k!}\left(\frac{\partial}{\partial z_2}\right)^k\delta\left(\frac{z_2}{z_1}\right).
\end{align*}
\end{lemt}
\subsection{$(\Z,\chi_{p^N})$-equivariant $\phi$-coordinated quasi $V_{\widehat{\mathfrak{sl}}_{\infty},\hbar}(\ell,0)$-modules}
In this subsection, we study the category of $(\Z,\chi_{p^N})$-equivariant $\phi$-coordinated  quasi $V_{\widehat{\mathfrak{sl}}_{\infty},\hbar}(\ell,0)$-modules.

 We first introduce a category $\mathcal{M}_{\ell}^{\phi}$ as follows: an object in $\mathcal{M}_{\ell}^{\phi}$ is a pair \[(W,(\psi_i(z),y_{i}^{\pm}(z))_{0\leq i\leq N-1}),\] where $W$ is a  topologically free $\C[[\hbar]]$-module and $\psi_{i}(z)$, $y_{i}^{\pm}(z)\in\mathcal{E}_{\hbar}(W)$ for $0\leq i\leq N-1$ such that  the following conditions hold ($0\leq i,j\leq N-1$):
\begin{eqnarray}
\notag &&[\psi_{i}(z_1),\psi_{j}(z_2)]\\
\label{eq:psi}&=&[b_{i,j}]_{q^{z_2\frac{\partial}{\partial z_2}}}[\ell]_{q^{z_2\frac{\partial}{\partial z_2}}}\left(\iota_{z_1,z_2}q^{-\ell z_2\frac{\partial}{\partial z_2}}-\iota_{z_2,z_1}q^{\ell z_2\frac{\partial}{\partial z_2}}\right)\frac{p^{-m_{i,j}}z_1z_2}{\left(z_1-p^{-m_{i,j}}z_2\right)^2},\\
\notag &&[\psi_{i}(z_1),y_{j}^{\pm}(z_2)]\\
\label{eq:psiy}&=&\pm y_{j}^{\pm}(z_2)[b_{i,j}]_{q^{z_2\frac{\partial}{\partial z_2}}}\left(\iota_{z_1,z_2}q^{-\ell z_2\frac{\partial}{\partial z_2}}-\iota_{z_2,z_1}q^{\ell z_2\frac{\partial}{\partial z_2}}\right)\frac{z_1+p^{-m_{i,j}}z_2}{2z_1-2p^{-m_{i,j}}z_2},\\
\notag&&(z_1-z_2)^{\delta_{i,j}}(z_1-q^{-2\ell}z_2)^{\delta_{i,j}}\big(y_{i}^{+}(z_1)y_{j}^{-}(z_2)\\
\label{eq:y+-} &-&\iota_{z_2,z_1}\left(g_{ij}(z_2,p^{m_{i,j}}z_1)\right)y_{j}^{-}(z_2)y_{i}^{+}(z_1))\big)=0,\\
% \label{eq:y++} &\iota_{z_1,z_2,\hbar}f_{ij}(p^{m_{i,j}}z_1,z_2)y_{i}^{\pm}(z_1)y_{j}^{\pm}(z_2)=C_{i,j}\iota_{z_2,z_1,\hbar}f_{ji}(p^{m_{j,i}}z_2,z_1)y_{j}^{\pm}(z_2)y_{i}^{\pm}(z_1),
\notag &&p^{-m_{i,j}}\iota_{z_1,z_2}\left(f_{ij}^{+}(p^{m_{i,j}}z_1,z_2)\right)y_{i}^{\pm}(z_1)y_{j}^{\pm}(z_2)\\
\label{eq:y++}&=&C_{i,j}\,\iota_{z_2,z_1}\left(f_{ji}^{+}(p^{m_{j,i}}z_2,z_1)\right)y_{j}^{\pm}(z_2)y_{i}^{\pm}(z_1),
\end{eqnarray}
where $f_{ij}^{+}(z_1,z_2)$ is as in \eqref{eq:fpm}, $g_{ij}(z_2,z_1)$ is as in \eqref{gij}, and $C_{i,j}$ is as in \eqref{cij}.
A morphism \[f\colon (W,(\psi_i(z),y_{i}^{\pm}(z))_{0\leq i\leq N-1})\to (W',(\psi_i'(z),(y_{i}^{\pm})'(z))_{0\leq i\leq N-1})\] of objects in $\mathcal{M}_\ell^\phi$ is a $\C[[\hbar]]$-module map
\[f:\ W\to W',\]
such that
\[f\circ \psi_i(z)=\psi_i'(z)\circ f\quad\textrm{and}\quad f\circ y_{i}^{\pm}(z)=(y_{i}^{\pm})'(z)\circ f\quad \textrm{for }0\leq i\leq N-1.\]

For every $a\in \C^\times$, we define a character on $\Z$ as follows:
\[
\chi_a: \Z\rightarrow \C^\times,\quad n\mapsto a^n.
\]
For every $(W,(\psi_{i}(z),y_{i}^{\pm}(z))_{0\leq i\leq N-1})\in \mathcal{M}_{\ell}^{\phi}$, form the following subset of $\mathcal{E}_{\hbar}(W)$: \[U_{W}=\{\psi_{i}(z),y_{i}^{\pm}(z)\mid 0\leq i\leq N-1\}.\]  From \eqref{eq:psi}-\eqref{eq:y++},  it follows that $U_W$ is an $\hbar$-adically $(\Z,\chi_p)$-quasi compatible subset of $\mathcal{E}_{\hbar}(W)$ (cf. \cite[Lemma 5.2]{Li5}).
By Proposition \ref{consthm},  we obtain an $\hbar$-adic nonlocal $\Z$-vertex algebra
\[(\langle U_W\rangle_{\phi},\mathcal{R})\] such that $W$ is a faithful $(\Z,\chi_p)$-equivariant $\phi$-coordinated quasi $\langle U_W\rangle_{\phi}$-module, where $\mathcal{R}$ is as in \eqref{gphomo}.

By taking the composition of  the homomorphism
\[\Z\rightarrow \Z,\quad n\mapsto nN
\] and the homomorphism $\mathcal{R}$, we get a group homomorphism
\begin{align*}
\mathcal{R}_N:\ \Z\to \mathrm{Aut}(\langle U_W\rangle_{\phi}),\quad n\mapsto \mathcal{R}_{nN}.
\end{align*}
Then
\[(\langle U_W\rangle_{\phi},\mathcal{R}_N)\] is also an $\hbar$-adic nonlocal $\Z$-vertex algebra  such that $W$ is a faithful $(\Z,\chi_{p^N})$-equivariant $\phi$-coordinated quasi $\langle U_W\rangle_{\phi}$-module.

Let $\mathcal{R}_{\ell}^{\phi}$ denote the full subcategory of $\mathcal{M}_{\ell}^{\phi}$ consisting of those objects \[(W,(\psi_{i}(z),y_{i}^{\pm}(z))_{0\leq i\leq N-1})\]  which satisfy the  conditions $(0\leq i,j\leq N-1)$:
\begin{align}
\label{eq:y+-'}&y_{i}^{+}(z_1)y_{j}^{-}(z_2)-\iota_{z_2,z_1}\left(g_{ij}(z_2,p^{m_{i,j}}z_1)\right)y_{j}^{-}(z_2)y_{i}^{+}(z_1)\\
=\ &\frac{\delta_{i,j}}{q-q^{-1}}\left(\delta\left(\frac{z_2}{z_1}\right)-E(\psi_{i}(z))\delta\left(\frac{q^{-2\ell}z_2}{z_1}\right)\right),\nonumber\\
\label{eq:serrey}&\left(\left(y_{i}^{\pm}(z)\right)_{0}^{\phi}\right)^{1-b_{i,j}} y_{j}^{\pm}(p^{m_{i,j}}z)=0\quad\textrm{if }b_{i,j}\le 0.
\end{align}
 Here, $g_{ij}(z_2,p^{m_{i,j}}z_1)$ is as in \eqref{gij} and
\begin{equation}\label{eq:defE}
\begin{split}
E(\psi_i(z))\ &=\ \left( \frac{F(1+\ell)}{F(1-\ell)}\right)^{\frac{1}{2}}\textrm{exp}\left(\left(-q^{-\ell\partial}F(\partial)\psi_i(z)\right)_{-1}\right){\bf 1}\in \langle U_W\rangle_{\phi}.
\end{split}
\end{equation}

For later use, we need the following technical result.

\begin{lemt}
For every $i,j\in\Z$, the following equalities hold in $V_{\widehat{\mathfrak{sl}}_{\infty},\hbar}(\ell,0)[[z,z^{-1}]]:$
\begin{align}
&\label{eq:singh}\mathrm{Sing}_{z}Y_{\tau}(h_{i,\hbar},z)h_{j,\hbar}=\mathrm{Sing}_{z_1}\tau_{ij}(z){\bf 1}+a_{ij}\ell{\bf 1}z^{-2},\\
&\label{eq:singx}\mathrm{Sing}_{z}Y_{\tau}(h_{i,\hbar},z)x_{j,\hbar}^{\pm}=\pm a_{i,j}x_{j,\hbar}^{\pm}z^{-1}\pm\mathrm{Sing}_{z}\tau_{ij}^{1,\pm}(z)x_{j,\hbar}^{\pm} ,\\
&\label{eq:singxpm1}\mathrm{Sing}_{z}Y_{\tau}(x_{i,\hbar}^{\pm},z)x_{j,\hbar}^{\pm}=0,\quad\textrm{if }a_{i,j}=0,\\
&\label{eq:singxpm2}\mathrm{Sing}_{z}Y_{\tau}(x_{i,\hbar}^{\pm},z)x_{j,\hbar}^{\pm}=\sum_{k\geq 0}(-\hbar)^k(x_{i,
\hbar}^{\pm})_{0}x_{j,\hbar}^{\pm}z^{-k-1}\quad\textrm{if }a_{i,j}<0,\\
&\label{eq:singxpm3}\mathrm{Sing}_{z}Y_{\tau}(x_{i,\hbar}^{\pm},z)x_{i,\hbar}^{\pm}=\sum_{k\geq 0}(2\hbar)^{k+1}(x_{i,
\hbar}^{\pm})_{-1}x_{i,\hbar}^{\pm}z^{-k-1},\\
&\label{eq:singxmp}\mathrm{Sing}_{z}Y_{\tau}(x_{i,\hbar}^{+},z)x_{i,\hbar}^{-}=(q-q^{-1})^{-1}\left({\bf 1}z^{-1}-E(h_{i,\hbar})(z+2\ell\hbar)^{-1}\right).
\end{align}
\end{lemt}
\begin{proof}
The  first three equalities are
implied by \eqref{eq:hh},
\eqref{eq:hx} and \eqref{eq:Rell3} respectively. The last equality follows from \cite[Proposition 6.10]{K}. For the  equalities \eqref{eq:singxpm2} and \eqref{eq:singxpm3}, it follows from \cite[Lemma 6.9]{K} that ($i,j\in \Z$)
\begin{align*}
&\mathrm{Sing}_{z}(z+\hbar)Y_{\tau}(x_{i,\hbar}^{\pm},z)x_{j,\hbar}^{\pm}=0\quad\textrm{if }a_{i,j}<0,\\
&\mathrm{Sing}_{z}z^{-1}(z-2\hbar)Y_{\tau}(x_{i,\hbar}^{\pm},z)x_{i,\hbar}^{\pm}=0.
\end{align*}
This implies that
\begin{align*}
(x_{i,\hbar}^{\pm})_{k+1}x_{j,\hbar}^{\pm}=-\hbar(x_{i,\hbar}^{\pm})_{k}x_{j,\hbar}^{\pm}\quad \textrm{and}\quad (x_{i,\hbar}^{\pm})_{k}x_{i,\hbar}^{\pm}=2\hbar(x_{i,\hbar}^{\pm})_{k-1}x_{i,\hbar}^{\pm}
\end{align*}
for $k\geq 0$, $i,j\in\Z$ with  $a_{i,j}<0$, as desired.
\end{proof}

% Denote by $\C(z)$ the field of rational functions in $z$. Note that there is a natural embedding as follows:
% \begin{align}\label{eq:embedding}
% \C(z_2/z_1)[[\hbar]]\to \C((z_2))((z_1))[[\hbar]]
% \end{align}
% In what follows, we view an element of $\C(z_2/z_1)[[\hbar]]$ as an element of $\C((z_2))((z_1))[[\hbar]]$ through the embedding \eqref{eq:embedding}.

\begin{prpt}\label{pr:modhh}
Let $(W,Y_W^{\phi})$ be a $(\Z,\chi_{p^N})$-equivariant $\phi$-coordinated  quasi $V_{\widehat{\mathfrak{sl}}_{\infty},\hbar}(\ell,0)$-module and let $i,j\in\Z$. Then we have that
\begin{align}
\notag&\left[Y_{W}^{\phi}(h_{i,\hbar},p^iz_1),Y_{W}^{\phi}(h_{j,\hbar},p^jz_2)\right]\\
\label{eq:Ywphihh}=\ &[b_{\underline{i},\underline{j}}]_{q^{z_2\frac{\partial}{\partial z_2}}}[\ell]_{q^{z_2\frac{\partial}{\partial z_2}}}\left(\iota_{z_1,z_2}q^{-\ell z_2\frac{\partial}{\partial z_2}}-\iota_{z_2,z_1}q^{\ell z_2\frac{\partial}{\partial z_2}}\right)\frac{p^{m_{\underline{j},\underline{i}}}z_1z_2}{\left(z_1-p^{m_{\underline{j},\underline{i}}}z_2\right)^2}.\\
&\left[Y_{W}^{\phi}(h_{i,\hbar},p^iz_1),Y_{W}^{\phi}(x_{j,\hbar}^{\pm},p^jz_2)\right]\nonumber\\
\label{eq:Ywphihx}=\ &Y_{W}^{\phi}(x_{j,\hbar}^{\pm},p^jz_2)[b_{\underline{i},\underline{j}}]_{q^{z_2\frac{\partial}{\partial z_2}}}\left(\iota_{z_1,z_2}q^{-\ell z_2\frac{\partial}{\partial z_2}}-\iota_{z_2,z_1}q^{\ell z_2\frac{\partial}{\partial z_2}}\right)\frac{z_1+p^{m_{\underline{j},\underline{i}}}z_2}{2z_1-2p^{m_{\underline{j},\underline{i}}}z_2}.
\end{align}
\end{prpt}
\begin{proof}
By \eqref{eq:hh}, we have that
\begin{align}\label{simhh}
Y_{\tau}(h_{i,\hbar},z_1)&Y_{\tau}(h_{j,\hbar},z_2) \sim Y_{\tau}(h_{j,\hbar},z_2)Y_{\tau}(h_{i,\hbar},z_1)+\iota_{z_2,z_1}[b_{\underline{i},\underline{j}}]_{q^{\frac{\partial}{\partial z_2}}}[\ell]_{q^{\frac{\partial}{\partial z_2}}}\\
&\cdot\left(q^{-\ell\frac{\partial}{\partial z_2}}-q^{\ell\frac{\partial}{\partial z_2}}\right)\frac{p^{i-j+m_{\underline{j},\underline{i}}}e^{-z_1+z_2}}{\left(1-p^{i-j+m_{\underline{j},\underline{i}}}e^{-z_1+z_2}\right)^2}Y_{\tau}({\bf 1},z_2)Y_{\tau}({\bf 1},z_1)\nonumber
\end{align}
for $i,j\in\Z$.
In view of  \eqref{tauij}, \eqref{eq:singh}, and \eqref{simhh}, it follows from
Lemma \ref{le:equimod} that
\begin{align*}
&\left[Y_{W}^{\phi}(h_{i,\hbar},p^iz_1),Y_{W}^{\phi}(h_{j,\hbar},p^jz_2)\right]\\
=\ &\iota_{z_2,z_1}\left([b_{\underline{i},\underline{j}}]_{q^{z_2\frac{\partial}{\partial z_2}}}[\ell]_{q^{z_2\frac{\partial}{\partial z_2}}}\left(q^{-z_2\ell\frac{\partial}{\partial z_2}}-q^{\ell z_2\frac{\partial}{\partial z_2}}\right)\frac{p^{m_{\underline{j},\underline{i}}}z_1z_2}{\left(z_1-p^{m_{\underline{j},\underline{i}}}z_2\right)^2}\right)\\
&+\sum_{n\in\Z}\sum_{k\geq 0}\mathrm{Res}_{z}z^{k}[b_{\underline{i},\underline{j}}]_{q^{\frac{\partial}{\partial z}}}[\ell]_{q^{\frac{\partial}{\partial z}}}q^{\ell\frac{\partial}{\partial z}}\frac{p^{i+nN-j+m_{\underline{j},\underline{i}}}e^z}{\left(1-p^{i+nN-j+m_{\underline{j},\underline{i}}}e^z\right)^2}\frac{1}{k!}\left(z_2\frac{\partial}{\partial z_2}\right)^k\delta\left(\frac{p^{-nN+j-i}z_2}{z_1}\right)\\
=\ &\iota_{z_2,z_1}[b_{\underline{i},\underline{j}}]_{q^{z_2\frac{\partial}{\partial z_2}}}[\ell]_{q^{z_2\frac{\partial}{\partial z_2}}}\left(q^{-z_2\ell\frac{\partial}{\partial z_2}}-q^{\ell z_2\frac{\partial}{\partial z_2}}\right)\frac{p^{m_{\underline{j},\underline{i}}}z_1z_2}{\left(z_1-p^{m_{\underline{j},\underline{i}}}z_2\right)^2}\\
&+\mathrm{Res}_{z}[b_{\underline{i},\underline{j}}]_{q^{\frac{\partial}{\partial z}}}[\ell]_{q^{\frac{\partial}{\partial z}}}q^{\ell\frac{\partial}{\partial z}}\frac{e^z}{\left(1-e^z\right)^2}e^{zz_2\frac{\partial}{\partial z_2}}\delta\left(\frac{p^{m_{\underline{j},\underline{i}}}z_2}{z_1}\right)\\
=\ &\iota_{z_2,z_1}\left([b_{\underline{i},\underline{j}}]_{q^{z_2\frac{\partial}{\partial z_2}}}[\ell]_{q^{z_2\frac{\partial}{\partial z_2}}}\left(q^{-z_2\ell\frac{\partial}{\partial z_2}}-q^{\ell z_2\frac{\partial}{\partial z_2}}\right)\frac{p^{m_{\underline{j},\underline{i}}}z_1z_2}{\left(z_1-p^{i-j+m_{\underline{j},\underline{i}}}z_2\right)^2}\right)\\
&+[b_{\underline{i},\underline{j}}]_{q^{z_2\frac{\partial}{\partial z_2}}}[\ell]_{q^{z_2\frac{\partial}{\partial z_2}}}q^{-\ell z_2\frac{\partial}{\partial z_2}}z_2\frac{\partial}{\partial z_2}\delta\left(\frac{p^{m_{\underline{j},\underline{i}}}z_2}{z_1}\right)\\
=\ &[b_{\underline{i},\underline{j}}]_{q^{z_2\frac{\partial}{\partial z_2}}}[\ell]_{q^{z_2\frac{\partial}{\partial z_2}}}\left(\iota_{z_1,z_2}q^{-\ell z_2\frac{\partial}{\partial z_2}}-\iota_{z_2,z_1}q^{\ell z_2\frac{\partial}{\partial z_2}}\right)\frac{p^{m_{\underline{j},\underline{i}}}z_1z_2}{\left(z_1-p^{m_{\underline{j},\underline{i}}}z_2\right)^2},
\end{align*}
where in the penultimate equality, we use the fact that
\begin{align}\label{eq:resfz}
\textrm{Res}_{z}\left(f\left(\frac{\partial}{\partial z}\right)z^{-r-1}\right)e^{zz_1}=\frac{1}{r!}z_1^rf(-z_1)
\end{align}
for every $f(z)\in\C[z][[\hbar]]$ and $r\in\N$. This proves the equality \eqref{eq:Ywphihh}. The proof of the equality \eqref{eq:Ywphihx} is similar, and we omit the details.
\end{proof}

\begin{prpt}\label{pr:modxxpm}
Let $(W,Y_W^{\phi})$ be a $(\Z,\chi_{p^N})$-equivariant $\phi$-coordinated  quasi $V_{\widehat{\mathfrak{sl}}_{\infty},\hbar}(\ell,0)$-module and let $i,j\in\Z$. Then we have that
\begin{eqnarray}\label{eq:Ywphixxpm}
&&\notag p^{-m_{\underline{i},\underline{j}}}\iota_{z_1,z_2}f_{\underline{i},\underline{j}}^{+}(p^{m_{\underline{i},\underline{j}}}z_1,z_2)Y_{W}^{\phi}(x_{i,\hbar}^{\pm},p^iz_1)Y_{W}^{\phi}(x_{j,\hbar}^{\pm},p^jz_2)\\
&=&C_{\underline{i},\underline{j}}\iota_{z_2,z_1}f_{\underline{j},\underline{i}}^{+}(p^{m_{\underline{j},\underline{i}}}z_2,z_1)Y_{W}^{\phi}(x_{j,\hbar}^{\pm},p^jz_2)Y_{W}^{\phi}(x_{i,\hbar}^{\pm},p^iz_1)
\end{eqnarray}
\end{prpt}
\begin{proof}
By  \eqref{eq:xxpm}, we get that
\begin{align}\label{simxx}
&Y_{\tau}(x_{i,\hbar}^{\pm},z_1)Y_{\tau}(x_{j,\hbar}^{\pm},z_2)\\
\sim & \iota_{z_2,z_1}\left(\tau_{ji}^{\pm,\pm}(z_2-z_1)\tau_{ij}^{\pm,\pm}(-z_2+z_1)^{-1}\right)Y_{\tau}(x_{j,\hbar}^{\pm},z_2)Y_{\tau}(x_{i,\hbar}^{\pm},z_1).\nonumber
\end{align}
Set
\begin{align*}
A(z_1,z_2)=&Y_{W}^{\phi}(x_{i,\hbar}^{\pm},z_1)Y_{W}^{\phi}(x_{j,\hbar}^{\pm},z_2)-\iota_{z_2,z_1}\left(\tau_{ji}^{\pm,\pm}(z_2-z_1)\right)\\
&\cdot\iota_{z_2,z_1}\left(\tau_{ij}^{\pm,\pm}(-z_2+z_1)^{-1}\right)Y_{W}^{\phi}(x_{j,\hbar}^{\pm},z_2)Y_{W}^{\phi}(x_{i,\hbar}^{\pm},z_1).
\end{align*}
Then by \eqref{tijpm}, we have that
\begin{align*}
A(p^iz_1,p^jz_2)=&Y_{W}^{\phi}(x_{i,\hbar}^{\pm},p^iz_1)Y_{W}^{\phi}(x_{j,\hbar}^{\pm},p^jz_2)-p^{m_{\underline{i},\underline{j}}}C_{\underline{i},\underline{j}}\iota_{z_2,z_1}\left(f_{\underline{j},\underline{i}}^{+}(p^{m_{\underline{j}\underline{i}}}z_2,z_1)\right)\\
&\cdot\iota_{z_2,z_1}\left(f_{\underline{i},\underline{j}}^{+}(p^{m_{\underline{i}\underline{j}}}z_1,z_2)^{-1}\right)Y_{W}^{\phi}(x_{j,\hbar}^{\pm},p^jz_2)Y_{W}^{\phi}(x_{i,\hbar}^{\pm},p^iz_1).
\end{align*}
Now we divide the proof into three cases.

\setlist[enumerate]{label=\arabic*., leftmargin=*}
\begin{enumerate}[wide]
 \item [$\textbf{Case 1.}$] $b_{\underline{i},\underline{j}}=0$. In this case, in view of  \eqref{eq:singxpm1} and \eqref{simxx}, it follows from  Lemma \ref{le:equimod} that
 \begin{align*}
&A(p^iz_1,p^jz_2)=\sum_{n\in\Z}\sum_{k\geq 0}Y_W^{\phi}\left((x_{i+nN,\hbar}^{\pm})_{k}x_{j,\hbar}^{\pm},p^jz_2\right)\frac{1}{k!}\left(z_2\frac{\partial }{\partial z_2}\right)^k\delta\left(\frac{p^{-nN+j-i}z_2}{z_1}\right)=0.
 \end{align*}
The equality \eqref{eq:Ywphixxpm} then follows by
multiplying  both sides of the above equality by  $z_1-z_2$.
\vspace{3mm}

 \item [$\textbf{Case 2.}$] $b_{\underline{i},\underline{j}}<0$. In this case, in view of  \eqref{eq:singxpm2}  and  \eqref{simxx}, it follows from Lemma \ref{le:equimod} that
 \begin{align*}
 A(p^iz_1,p^jz_2)=\ &\sum_{n\in\Z}\sum_{k\geq 0}Y_W^{\phi}\left((x_{i+nN,\hbar}^{\pm})_{k}x_{j,\hbar}^{\pm},p^jz_2\right)\frac{1}{k!}\left(z_2\frac{\partial }{\partial z_2}\right)^k\delta\left(\frac{p^{-nN+j-i}z_2}{z_1}\right)\\
 =\ &Y_W^{\phi}\left((x_{j+m_{\underline{i},\underline{j}},\hbar}^{\pm})_{0}x_{j,\hbar}^{\pm},p^jz_2\right)e^{-\hbar z_2\frac{\partial }{\partial z_2}}\delta\left(\frac{p^{-m_{\underline{i},\underline{j}}}z_2}{z_1}\right)\\
 =\ &Y_W^{\phi}\left((x_{j+m_{\underline{i},\underline{j}},\hbar}^{\pm})_{0}x_{j,\hbar}^{\pm},p^jz_2\right)\delta\left(\frac{p^{-m_{\underline{i},\underline{j}}}q^{-1}z_2}{z_1}\right).
 \end{align*}
 Multiplying both sides of this equality by $p^{m_{\underline{i},\underline{j}}}(z_1-q^{-1}p^{-m_{\underline{i},\underline{j}}}z_2)$, we obtain \eqref{eq:Ywphixxpm}.
 \vspace{3mm}

\item [$\textbf{Case 3.}$] $b_{\underline{i},\underline{j}}=2$. In this case, by   Lemma \ref{le:equimod}, \eqref{eq:singxpm3} and  \eqref{simxx}, we have  that
 \begin{align*}
 A(p^iz_1,p^jz_2)=\ &\sum_{n\in\Z}\sum_{k\geq 0}Y_W^{\phi}\left((x_{i+nN,\hbar}^{\pm})_{k}x_{j,\hbar}^{\pm},p^jz_2\right)\frac{1}{k!}\left(z_2\frac{\partial }{\partial z_2}\right)^k\delta\left(\frac{p^{-nN+j-i}z_2}{z_1}\right)\\
 =\ &2\hbar Y_W^{\phi}\left((x_{j,\hbar}^{\pm})_{-1}x_{j,\hbar}^{\pm},z_2\right)e^{2\hbar z_2\frac{\partial }{\partial z_2}}\delta\left(\frac{z_2}{z_1}\right)\\
 =\ &2\hbar Y_W^{\phi}\left((x_{j,\hbar}^{\pm})_{-1}x_{j,\hbar}^{\pm},z_2\right)\delta\left(\frac{q^{2}z_2}{z_1}\right).
 \end{align*}
Multiplying both sides by $q^{-1}(z_1-q^{2}z_2)$, we obtain that
 \begin{align*}
 (z_1-q^2z_2)Y_{W}^{\phi}(x_{i,\hbar}^{\pm},p^iz_1)Y_{W}^{\phi}(x_{i,\hbar}^{\pm},p^iz_2)=(q^2z_1-z_2)Y_{W}^{\phi}(x_{j,\hbar}^{\pm},p^iz_2)Y_{W}^{\phi}(x_{i,\hbar}^{\pm},p^iz_1).
 \end{align*}
This implies that
\begin{align}\label{eq:ywphipm}
\left((z_1-q^2z_2)Y_{W}^{\phi}(x_{i,\hbar}^{\pm},p^iz_1)Y_{W}^{\phi}(x_{i,\hbar}^{\pm},p^iz_2)\right)|_{z_1=z_2}=0.
\end{align}

It is straightforward to see  that
\begin{align*}
&p^{-m_{\underline{i},\underline{j}}}\iota_{z_1,z_2}f_{\underline{i},\underline{j}}^{+}(p^{m_{\underline{i},\underline{j}}}z_1,z_2)Y_{W}^{\phi}(x_{i,\hbar}^{\pm},p^iz_1)Y_{W}^{\phi}(x_{j,\hbar}^{\pm},p^jz_2)\\
&- C_{\underline{i},\underline{j}}\iota_{z_2,z_1}f_{\underline{j},\underline{i}}^{+}(p^{m_{\underline{j},\underline{i}}}z_2,z_1)Y_{W}^{\phi}(x_{j,\hbar}^{\pm},p^jz_2)Y_{W}^{\phi}(x_{i,\hbar}^{\pm},p^iz_1)\\
=\ &\delta_{i,j}z_1^{-1}\delta\left(\frac{z_1}{z_2}\right)\left((z_1-q^2z_2)Y_{W}^{\phi}(x_{i,\hbar}^{\pm},p^iz_1)Y_{W}^{\phi}(x_{i,\hbar}^{\pm},p^iz_2)\right).
\end{align*}
 This together with \eqref{eq:ywphipm} implies \eqref{eq:Ywphixxpm}, as desired.
 \end{enumerate}
\end{proof}

We also have the following result.
\begin{prpt}\label{pr:modxxmp}
Let $(W,Y_W^{\phi})$ be a $(\Z,\chi_{p^N})$-equivariant $\phi$-coordinated  quasi $V_{\widehat{\mathfrak{sl}}_{\infty},\hbar}(\ell,0)$-module and let $i,j\in\Z$. Then we have that
\begin{align*}
&Y_{W}^{\phi}(x_{i,\hbar}^{+},p^iz_1)Y_{W}^{\phi}(x_{j,\hbar}^{-},p^jz_2)-\iota_{z_2,z_1}\left(\frac{z_2-q^{b_{\underline{i},\underline{j}}}p^{m_{\underline{i},\underline{j}}}z_1}{q^{b_{\underline{i},\underline{j}}}z_2-p^{m_{\underline{i},\underline{j}}}z_1}\right)Y_{W}^{\phi}(x_{j,\hbar}^{-},p^jz_2)Y_{W}^{\phi}(x_{i,\hbar}^{+},p^iz_1)\\
=\ &{\delta_{\underline{i},\underline{j}}}(q-q^{-1})^{-1}\left(\delta\left(\frac{z_2}{z_1}\right)-E\left(Y_W^{\phi}(h_{i,\hbar},z_2)\right)\delta\left(\frac{q^{-2\ell}z_2}{z_1}\right)\right).
\end{align*}
\end{prpt}
\begin{proof}
By  \eqref{eq:xxmp}, we get that
\begin{align}\label{simxxmp}
&Y_{\tau}(x_{i,\hbar}^{+},z_1)Y_{\tau}(x_{j,\hbar}^{-},z_2)\\
\sim\ &\iota_{z_2,z_1}\left(\tau_{ji}^{-,+}(z_2-z_1)\tau_{ij}^{+,-}(-z_2+z_1)^{-1}\right)Y_{\tau}(x_{j,\hbar}^{-},z_2)Y_{\tau}(x_{i,\hbar}^{+},z_1).\nonumber
\end{align}
In view of \eqref{tij+-}, \eqref{tij-+}, \eqref{eq:singxmp}, and \eqref{simxxmp}, it follows from Lemma \ref{le:equimod} that
\begin{align*}
&Y_{W}^{\phi}(x_{i,\hbar}^{+},p^iz_1)Y_{W}^{\phi}(x_{j,\hbar}^{-},p^jz_2)-\iota_{z_2,z_1}\left(\frac{z_2-q^{b_{\underline{i},\underline{j}}}p^{m_{\underline{i},\underline{j}}}z_1}{q^{b_{\underline{i},\underline{j}}}z_2-p^{m_{\underline{i},\underline{j}}}z_1}\right)Y_{W}^{\phi}(x_{j,\hbar}^{-},p^jz_2)Y_{W}^{\phi}(x_{i,\hbar}^{+},p^iz_1)\\
=\ &\sum_{n\in\Z}\sum_{k\geq 0}\textrm{Res}_{z}z^k\delta_{i+nN,j}Y_W^{\phi}\left(Y(x_{i+nN,\hbar}^{+},z)x_{j,\hbar}^{-},p^jz_2\right)\frac{1}{k!}\left(z_2\frac{\partial }{\partial z_2}\right)^k\delta\left(\frac{z_2}{z_1}\right)\\
=\ &{\delta_{\underline{i},\underline{j}}}\textrm{Res}_{z}Y_W^{\phi}\left(Y(x_{i,\hbar}^{+},z)x_{i,\hbar}^{-},p^jz_2\right)e^{zz_2\frac{\partial }{\partial z_2}}\delta\left(\frac{z_2}{z_1}\right)\\
=\ &{\delta_{\underline{i},\underline{j}}}(q-q^{-1})^{-1}\left(\delta\left(\frac{z_2}{z_1}\right)-\textrm{Res}_{z}Y_W^{\phi}\left(E(h_{i,\hbar}),p^jz_2\right)(z+2\ell\hbar)^{-1}e^{zz_2\frac{\partial }{\partial z_2}}\delta\left(\frac{z_2}{z_1}\right)\right)\\
=\ &{\delta_{\underline{i},\underline{j}}}(q-q^{-1})^{-1}\left(\delta\left(\frac{z_2}{z_1}\right)-Y_W^{\phi}\left(E(h_{i,\hbar}),p^jz_2\right)\delta\left(\frac{q^{-2\ell}z_2}{z_1}\right)\right)\\
=\ &{\delta_{\underline{i},\underline{j}}}(q-q^{-1})^{-1}\left(\delta\left(\frac{z_2}{z_1}\right)-E\left(Y_W^{\phi}(h_{i,\hbar},p^jz_2)\right)\delta\left(\frac{q^{-2\ell}z_2}{z_1}\right)\right),
\end{align*}
where we use \eqref{eq:resfz} in the penultimate equality, and  use the facts that
\begin{align}
\label{eq:phitau}Y_W^{\phi}(Y_{\tau}(u,z)v,z_1)\ &=\ Y_{\mathcal{E}}^{\phi}(Y_W^{\phi}(u,z_1),z)Y_W^{\phi}(v,z_1),\\
Y_W^{\phi}(\partial(u),z)\ &=\ z\frac{\partial}{\partial z}Y_W^{\phi}(u,z)\nonumber
\end{align}
for $u,v\in V_{\widehat{\mathfrak{sl}}_{\infty},\hbar}(\ell,0)$ in the last equality. This completes the proof.

\end{proof}
% \begin{prpt}
% Let $(W,Y_W^{\phi})$ be a $(\Z,\chi_{\phi})$-equivariant $\phi$-coordinated  quasi $V_{\hat{\mathfrak{g}},\hbar}(\ell,0)$-module. Then $(W, Y_{W}^{\phi}(h_{i,\hbar},p^iz),Y_{W}^{\phi}(x_{i,\hbar}^{\pm},p^iz))$
%  is an object of $\mathcal{R}_{\ell}^{\phi}$.
% \end{prpt}

The following is the main result of this section, which gives an isomorphism between the category of $(\Z,\chi_{p^N})$-equivariant $\phi$-coordinated  quasi $V_{\widehat{\mathfrak{sl}}_{\infty},\hbar}(\ell,0)$-modules and the category $\mathcal{R}_{\ell}^{\phi}$.
 \begin{prpt}\label{pr:vglRphi}
 Let $\left(W,(\psi_{i}(z),y_{i}^{\pm}(z))_{0\leq i\leq N-1}\right)$ be an object of $\mathcal{R}_{\ell}^{\phi}$. Then $W$ is a $(\Z,\chi_{p^N})$-equivariant $\phi$-coordinated  quasi $V_{\widehat{\mathfrak{sl}}_{\infty},\hbar}(\ell,0)$-module with
 \[Y_W^{\phi}(h_{i,\hbar},z)=\psi_{\underline{i}}(p^{-i}z)\quad\textrm{and}\quad Y_W^{\phi}(x_{i,\hbar}^{\pm},z)=y_{\underline{i}}^{\pm}(p^{-i}z)\quad\textrm{for } i\in\Z.\]
 On the other hand, let $(W,Y_W^{\phi})$ be a $(\Z,\chi_{p^N})$-equivariant $\phi$-coordinated  quasi $V_{\widehat{\mathfrak{sl}}_{\infty},\hbar}(\ell,0)$-module. Then \[\left(W, (Y_{W}^{\phi}(h_{k,\hbar},p^kz),Y_{W}^{\phi}(x_{k,\hbar}^{\pm},p^kz))_{0\leq k\leq N-1}\right)\] is an object of $\mathcal{R}_{\ell}^{\phi}$.
 \end{prpt}
 \begin{proof}
Let $\left(W,(\psi_{i}(z),y_{i}^{\pm}(z))_{0\leq i\leq N-1}\right)\in \mathcal{R}_{\ell}^{\phi}$. Recall that $(\langle U_W\rangle_{\phi},\mathcal{R}_{N})$ is an $\hbar$-adic nonlocal $\Z$-vertex algebra  such that $W$ is a faithful $(\Z,\chi_{p^N})$-equivariant $\phi$-coordinated quasi $\langle U_W\rangle_{\phi}$-module.
 In view of \eqref{eq:psi}-\eqref{eq:y++}, it follows from  Lemma \ref{le:Yephi} that \[\left(\langle U_W\rangle_{\phi}, (Y_{\mathcal{E}}^{\phi}(\psi_{\underline{i}}(p^{-i}z_1),z),Y_{\mathcal{E}}^{\phi}(y_{\underline{i}}^{\pm}(p^{-i}z_1),z))_{i\in\Z}\right)\] is an object of $\mathcal{M}_{\tau}$.
Then by Lemma \ref{le:uniofFtau}, there is an $\hbar$-adic nonlocal vertex algebra homomorphism
 \[\varphi:\ F_{\tau}(A,\ell)\to \langle U_W\rangle_{\phi}\]
 such that
 \[\varphi(h_{i,\hbar})=\psi_{\underline{i}}(p^{-i}z)\quad \textrm{and}\quad\varphi(x_{i,\hbar}^{\pm})=y_{\underline{i}}^{\pm}(p^{-i}z)\quad \textrm{for }i\in\Z.\]

According to \eqref{eq:y+-'} and  Lemma \ref{le:Yephi}, we have that
 \begin{align}
 &\nonumber Y_{\mathcal{E}}^{\phi}(y_{\underline{i}}^{+}(p^{-i}z),z_1)Y_{\mathcal{E}}^{\phi}(y_{\underline{j}}^{-}(p^{-j}z),z_2)\\
 &\nonumber-\iota_{z_2,z_1}\left(\frac{q^{b_{\underline{i},\underline{j}}}-p^{m_{\underline{i},\underline{j}}}e^{z_1-z_2}}{1-q^{b_{\underline{i},\underline{j}}}p^{m_{\underline{i},\underline{j}}}e^{z_1-z_2}}\right)Y_{\mathcal{E}}^{\phi}(y_{\underline{j}}^{-}(p^{-j}z),z_2)Y_{\mathcal{E}}^{\phi}(y_{\underline{i}}^{+}(p^{-i}z),z_1)\\
\label{eq:Yphiyij} =\ &\frac{\delta_{\underline{i},\underline{j}}}{q-q^{-1}}\left(z_1^{-1}\delta\left(\frac{z_2}{z_1}\right)-Y_{\mathcal{E}}^{\phi}(E(\psi_{\underline{i}}(p^{-j}z),z_2))z_1^{-1}\delta\left(\frac{z_2-2\ell\hbar}{z_1}\right)\right).
 \end{align}
By applying  the both sides of $\eqref{eq:Yphiyij}$ to $1_{W}$, and taking $\textrm{Sing}_{z_1}\textrm{Res}_{z_2}z_2^{-1}$, we find that
 \begin{align*}
\textrm{Sing}_{z_1}Y_{\mathcal{E}}^{\phi}(y_{\underline{i}}^{+}(p^{-i}z),z_1)y_{\underline{j}}^{-}(p^{-j}z)=\frac{\delta_{\underline{i},\underline{j}}}{q-q^{-1}}\left(1_{W}z_1^{-1}-E(\psi_{\underline{i}}(p^{-j}z)(z_1+2\ell\hbar)^{-1}\right).
 \end{align*}
 It then follows that
 \begin{align*}
 y_{\underline{i}}^{+}(p^{-i}z)_{0}^{\phi}y_{\underline{i}}^{-}(p^{-i}z)\ &=\ (q-q^{-1})^{-1}\left(1_{W}-E(\psi_{\underline{i}}(p^{-i}z)\right),\\
  y_{\underline{i}}^{+}(p^{-i}z)_{1}^{\phi}y_{\underline{i}}^{-}(p^{-i}z)\ &=\ 2\ell\hbar(q-q^{-1})^{-1}E(\psi_{\underline{i}}(p^{-i}z)).
 \end{align*}
 This, together with \eqref{eq:serrey} and \eqref{eq:Rell1}-\eqref{eq:Rell3}, shows that $\varphi$ factors through $V_{\widehat{\mathfrak{sl}}_{\infty},\hbar}(\ell,0)$.

 Note that
 \begin{align*}
 &\varphi\circ \rho_{n,\hbar}(h_{i,\hbar})=\psi_{\underline{i}}(p^{-i-nN}z)=\mathcal{R}_{nN}\circ \varphi(h_{i,\hbar}),\\
 &\varphi\circ \rho_{n,\hbar}(x_{i,\hbar}^{\pm})=y_{\underline{i}}(p^{-i-nN}z)=\mathcal{R}_{nN}\circ \varphi(x_{i,\hbar}^{\pm})
 \end{align*}
 for $n,i\in\Z$.
 This gives that   $\varphi$ is a  nonlocal $\Z$-vertex algebra homomorphism as  $\{h_{i,\hbar},x_{i,\hbar}^{\pm}\mid i\in\Z\}$ is a generating subset of $V_{\widehat{\mathfrak{sl}}_{\infty},\hbar}(\ell,0)$. Therefore, $W$ is a $(\Z,\chi_{p^N})$-equivariant $\phi$-coordinated quasi $V_{\widehat{\mathfrak{sl}}_{\infty},\hbar}(\ell,0)$-module such that
 \begin{align*}
 &Y_W^{\phi}(h_{i,\hbar},z_0)=Y_W^{\phi}(\varphi(h_{i,\hbar}),z_0)=Y_{W}^{\phi}(\psi_{\underline{i}}(p^{-i}z),z_0)=\psi_{\underline{i}}(p^{-i}z_0),\\
  &Y_W^{\phi}(x_{i,\hbar}^{\pm},z_0)=Y_W^{\phi}(\varphi(x_{i,\hbar}^{\pm}),z_0)=Y_{W}^{\phi}(y_{\underline{i}}^{\pm}(p^{-i}z),z_0)=y_{\underline{i}}^{\pm}(p^{-i}z_0),
 \end{align*}
 for $i\in\Z$.

On the other hand, let  $(W,Y_W^{\phi})$ be a $(\Z,\chi_{p^N})$-equivariant $\phi$-coordinated quasi $V_{\widehat{\mathfrak{sl}}_{\infty},\hbar}(\ell,0)$-module. By Propositions \ref{pr:modhh}-\ref{pr:modxxmp}, it remains to prove that
\begin{align*}
\left(\left(Y_W^{\phi}(x_{k,\hbar}^{\pm},p^kz)_{0}^{\phi}\right)^{1-b_{k,t}}Y_W^{\phi}(x_{t,\hbar}^{\pm},p^{m_{k,t}}p^tz)\right)=0
\end{align*}
for $0\leq k,t\leq N-1$ with $b_{k,t}\le 0$. Indeed, if $k=0$ and $t=N-1$, by \eqref{eq:Rell3}, \eqref{eq:modequicon}, and \eqref{eq:phitau}, we have that
\begin{align*}
&\left(\left(Y_W^{\phi}(x_{0,\hbar}^{\pm},z)_{0}^{\phi}\right)^2Y_W^{\phi}(x_{N-1,\hbar}^{\pm},p^{N}z)\right)\\
=\ &\left(\left(Y_W^{\phi}(x_{0,\hbar}^{\pm},z)_{0}^{\phi}\right)^2Y_W^{\phi}(x_{-1,\hbar}^{\pm},z)\right)\\
=\ &Y_W^{\phi}\left(\left(x_{0,\hbar}^{\pm}\right)_0^2x_{-1,\hbar}^{\pm},z\right)=0.
\end{align*}
The verifications of other cases are similar and we omit details. Then we complete the proof.
 \end{proof}

%  \begin{lemt}
% Let $(W,\psi_{i,q}(z),y_{i,q}^{\pm}(z))$ be an object of $\mathcal{M}_{\ell}^{\phi}$. Then $U_W$ is an $\hbar$-adically $G$-quasi compatible subset of $\mathcal{E}_{\hbar}(W)$, where $G=\{p^n\mid n\in\Z\}$.
%  \end{lemt}

\subsection{The main result}

Let $W$ be a restricted $\mathcal{E}_N^f$-module of level $\ell$. For every  $0\leq i\leq N-1$, write
\begin{align}\label{hatx}
\hat{x}_{i}^{+}(z)=x_{i}^{+}(z),\quad \hat{x}_{i}^{-}(z)=x_{i}^{-}(q^{-\ell}z)\phi_{i}^{+}(q^{-\frac{1}{2}\ell}z)^{-1}
\end{align}
and
\begin{align}\label{hatphi}
\hat{\phi}_{i}(z)=F\left(z\frac{\partial}{\partial z}\right)^{-1}\mathrm{log}\frac{\phi_{i}^{+}(q^{\frac{1}{2}\ell}z)}{\phi_{i}^{-}(q^{-\frac{1}{2}\ell}z)}.
\end{align}
Then we have that:

\begin{lemt}\label{le:cateresmod}
The pair \[\left(W,(\hat{\phi}_{i}(z),\hat{x}_{i}^{\pm}(z))_{0\leq i\leq N-1}\right)\] is an object of $\mathcal{M}_{\ell}^{\phi}.$
\end{lemt}
\begin{proof}
Recall that  $\phi_{i}^{\pm}(z),x_{i}^{\pm}(z)\in\mathcal{E}_{\hbar}(W)$ for $0\leq i\leq N-1$.
Then we have  that  $\hat{\phi}_{i}(z),\hat{x}_{i}^{\pm}(z)\in\mathcal{E}_{\hbar}(W)$ by \eqref{hatx} and \eqref{hatphi}. In what follows we prove that the condition \eqref{eq:y++} holds on $W$.
From  (Q6), it follows that
\[p^{-m_{i,j}}\iota_{z_1,z_2}f_{i,j}^{+}(p^{m_{i,j}}z_1,z_2)\hat{x}_{i}^{+}(z_1)\hat{x}_{j}^{+}(z_2)-C_{i,j}\iota_{z_2,z_1}f_{j,i}^{+}(p^{m_{ji}}z_2,z_1)\hat{x}_{j}^{+}(z_2)\hat{x}_{i}^{+}(z_1)=0\]
and that
\begin{align}\label{eq:xi--}
\left((z_1-q^{-2}z_2)x_{i}^{-}(z_1)x_{i}^{-}(z_2)\right)|_{z_2=z_1}=0
\end{align}
for $0\leq i,j\leq N-1$.
Then by  (Q1), (Q3), (Q6), \eqref{hatphi} and \eqref{eq:xi--}, we have that
\begin{align*}
&p^{-m_{i,j}}\iota_{z_1,z_2}f_{i,j}^{+}(p^{m_{i,j}}z_1,z_2)\hat{x}_{i}^{-}(z_1)\hat{x}_{j}^{-}(z_2)-C_{i,j}\iota_{z_2,z_1}f_{j,i}^{+}(p^{m_{j,i}}z_2,z_1)\hat{x}_{j}^{-}(z_2)\hat{x}_{i}^{-}(z_1)\\
=\ &p^{-m_{i,j}}\iota_{z_1,z_2}f_{i,j}^{+}(p^{m_{i,j}}z_1,z_2)x_{i}^{-}(q^{-\ell}z_1)\phi_{i}^{+}(q^{-\frac{1}{2}\ell}z_1)^{-1}x_{j}^{-}(q^{-\ell}z_2)\phi_{j}^{+}(q^{-\frac{1}{2}\ell}z_2)^{-1}\\
&-C_{i,j}\iota_{z_2,z_1}f_{j,i}^{+}(p^{m_{j,i}}z_2,z_1)x_{j}^{-}(q^{-\ell}z_2)\phi_{j}^{+}(q^{-\frac{1}{2}\ell}z_2)^{-1}x_{i}^{-}(q^{-\ell}z_1)\phi_{i}^{+}(q^{-\frac{1}{2}\ell}z_1)^{-1}\\
=\ &q^{b_{i,j}}p^{m_{j,i}}\iota_{z_1,z_2}\left(f_{i,j}^{-}(p^{m_{i,j}}z_1,z_2)\right)x_{i}^{-}(q^{-\ell}z_1)x_{j}^{-}(q^{-\ell}z_2)\phi_{i}^{+}(q^{-\frac{1}{2}\ell}z_1)^{-1}\phi_{j}^{+}(q^{-\frac{1}{2}\ell}z_2)^{-1}\\
&-q^{b_{i,j}}C_{i,j}\iota_{z_2,z_1}\left(f_{i,j}^{-}(p^{m_{j,i}}z_2,z_1)\right)x_{j}^{-}(q^{-\ell}z_2)x_{i}^{-}(q^{-\ell}z_1)\phi_{j}^{+}(q^{-\frac{1}{2}\ell}z_2)^{-1}\phi_{i}^{+}(q^{-\frac{1}{2}\ell}z_1)^{-1}\\
=\ &\delta_{i,j}q^{b_{i,j}}z_2^{-1}\delta\left(\frac{z_2}{z_1}\right)\left((z_1-q^{-b_{i,j}}z_2)x_{i}^{-}(q^{-\ell}z_1)x_{j}^{-}(q^{-\ell}z_2)\right)\phi_{i}^{+}(q^{-\frac{1}{2}\ell}z_1)^{-1}\phi_{j}^{+}(q^{-\frac{1}{2}\ell}z_2)^{-1}\\
=\ &0,
\end{align*}
as desired. Similarly,  \eqref{eq:psi} follows from (Q2), \eqref{eq:psiy} follows from (Q2), (Q3), and (Q4), \eqref{eq:y+-} follows from (Q3) and (Q5$'$).  This completes the proof.
\end{proof}

Let $\left(W,(\psi_{i}(z),y_{i}^{\pm}(z))_{0\leq i\leq N-1}\right)$ be an object of $\mathcal{M}_{\ell}^{\phi}$. Define
\begin{align}\label{checkypm}
\check{y}_{i}^{+}(z)=y_{i}(z),\quad \check{y}_{i}^{-}(z)=y_{i}^{-}(q^{\ell}z)\check{\psi}_{i}^{+}(q^{\frac{1}{2}\ell}z)
\end{align}
and
\begin{align}
\label{checkpsi}\check{\psi}_{i}^{\pm}(z)=\mathrm{exp}\left( \pm F\left(z\frac{\partial}{\partial z}\right)\psi_{i}^{\pm}(q^{\mp\frac{1}{2}\ell}z)\right)
\end{align}
for $0\leq i\leq N-1$.  Here, for every $a(z)=\sum_{n\in\Z}a(n)z^{-n}\in \mathcal{E}_{\hbar}(W)$,
\[a^{\pm}(z)=\sum_{\pm n>0}a(n)z^{-n}+\frac{1}{2}a(0).\]
Then we have that:
\begin{lemt}\label{le:Mphi}
The $\C[[\hbar]]$-module $W$ is a restricted $\mathcal{E}_N^f$-module of level $\ell$ with
\begin{align}\label{modac}
\phi_{i}^{\pm}(z)=\check{\psi}_{i}^{\pm}(z)\quad\textrm{and}\quad x_{i}^{\pm}(z)=\check{y}_{i}^{\pm}(z)\quad \textrm{for } 0\leq i\leq N-1.
\end{align}
\end{lemt}
\begin{proof}
Note that $\psi_i(z),y_i^{\pm}(z)\in\mathcal{E}_{\hbar}(W)$ for $0\leq i\leq N-1$. This together with \eqref{checkypm} and \eqref{checkpsi} implies that $\check{\psi}_{i}^{\pm}(z),\check{y}_{i}^{\pm}(z)\in\mathcal{E}_{\hbar}(W)$.
Thus we only need to prove that the action \eqref{modac} preserves the relations (Q1)-(Q4), (Q5$'$), and (Q6).

We will only check the relation (Q2), while the verifications of other relations are omitted. By \eqref{checkpsi} and \eqref{eq:psi}, we have that
\begin{align*}
&\left[F\left(z_1\frac{\partial}{\partial z_1}\right)\psi_{i}^{+}(q^{-\frac{1}{2}\ell}z_1),-F\left(z_2\frac{\partial}{\partial z_2}\right)\psi_{j}^{-}(q^{\frac{1}{2}\ell}z_2)\right]\\
=\ &-F\left(z_1\frac{\partial}{\partial z_1}\right)F\left(z_2\frac{\partial}{\partial z_2}\right)[b_{i,j}]_{q^{z_2\frac{\partial}{\partial z_2}}}[\ell]_{q^{z_2\frac{\partial}{\partial z_2}}}\iota_{z_1,z_2}q^{-\ell z_2\frac{\partial}{\partial z_2}}\frac{p^{-m_{i,j}}q^{\ell}z_1z_2}{\left(z_1-p^{-m_{i,j}}q^{\ell}z_2\right)^2}\\
=\ & -F\left(z_1\frac{\partial}{\partial z_1}\right)F\left(z_2\frac{\partial}{\partial z_2}\right)\sum_{n\geq 0}n[b_{i,j}]_{q^n}[\ell]_{q^n}q^{-\ell n}(p^{-m_{i,j}}q^{\ell}z_2/z_1)^n\\
=\ &\sum_{n\geq 0}\frac{(q^{nb_{i,j}}-q^{-nb_{i,j}})(q^{-n\ell}-q^{n\ell})}{n}(p^{-m_{i,j}}z_2/z_1)^n\\
=\ &\iota_{z_1,z_2}\left(\textrm{log}\left(g_{ij}(z_1,q^{\ell}p^{-m_{i,j}}z_2)^{-1}g_{ij}(z_1,q^{-\ell}p^{-m_{i,j}}z_2)\right)\right).
\end{align*}
Then the relation (Q2) follows.
\end{proof}

Let $(W,(\psi_{i}(z),y_{i}^{\pm}(z))_{0\leq i\leq N-1})\in \mathcal{M}_{\ell}^{\phi}$. For $k\in\Z_{+}$ and $0\leq i_1,i_2,\dots,i_k\leq N-1$, set
\begin{align*}
y_{i_1,i_2,\dots,i_k}^{\pm}(z_1,z_2,\dots,z_k)=\left(\prod_{1\leq s<t\leq k}f_{i_s,i_t}^{+}(p^{m_{i_s,i_t}}z_s,z_t)\right)y_{i_1}^{\pm}(z_1)y_{i_2}^{\pm}(z_2)\cdots y_{i_k}^{\pm}(z_k).
\end{align*}
Note that (cf. Lemma \ref{le:xiij})
\begin{align*}
y_{i_1,i_2,\dots,i_k}^{\pm}(a_1z_1,a_2z_2,\dots,a_kz_k)\in\mathcal{E}_{\hbar}^{(k)}(W)\quad \text{for}\ a_1,a_2,\dots,a_k\in \C[[\hbar]].
\end{align*}

\begin{lemt}\label{le:Rphiserre}
\noindent (a) For $0\leq i\leq N-1$, we have that (see \eqref{eq:defE})
  \begin{align}\label{eq:Epsi}
  E(\psi_{i}(z))=\check{\psi}_{i}^{-}(q^{-\frac{3}{2}\ell} z_2)\check{\psi}_{i}^{+}(q^{-\frac{1}{2}\ell}z_2)^{-1}.
  \end{align}

\noindent (b) For $0\leq i,j\leq N-1$ with $b_{i,j}=-1$, we have that
  \begin{align}
  \left(\left(y_{i}^{\pm}(z)\right)_{0}^{\phi}\right)^2 y_{j}^{\pm}(p^{m_{i,j}}z)=0
  \end{align}
  if and only if
  \[y_{i,i,j}^{\pm}(qp^{-m_{i,j}}z,q^{-1}p^{-m_{i,j}}z,z)=0.\]

\noindent (c) For $0\leq i,j\leq N-1$ with $b_{i,j}=0$, we have that
\begin{align}
    \left(\left(y_{i}^{\pm}(z)\right)_{0}^{\phi}\right) y_{j}^{\pm}(p^{m_{i,j}}z)=0
\end{align}
if and only if
  \[y_{i,j}^{\pm}(z,z)=0.\]
 \end{lemt}
 \begin{proof}
 The equality \eqref{eq:Epsi} follows from \cite[Lemma 8.7]{K}. For the  assertion (b), recall that $(\langle U_W\rangle_{\phi}, \mathcal{R}_{N})$ is an $\hbar$-adic nonlocal $\Z$-vertex algebra and $W$ is a faithful $(\Z,\chi_{p^N})$-equivariant $\phi$-coordinated quasi $\langle U_W\rangle_{\phi}$-module.
 By using \eqref{eq:y++}, we have that
 \begin{align*}
 (z_1-q^{b_{i,j}}z_2)Y_W^{\phi}(y_{i}^{\pm}(z),z_1)Y_W^{\phi}(y_{j}^{\pm}(p^{m_{i,j}}z),z_2)\in\mathcal{E}_{\hbar}^{(2)}(W).
 \end{align*}
 It then follows that
 \begin{align*}
 &y_{i,j}^{\pm}(z_2e^{z_0},p^{m_{i,j}}z_2)\\
 =\ &\left((z_1-q^{-1}z_2)Y_W^{\phi}(y_{i}^{\pm}(z),z_1)Y_W^{\phi}(y_{j}^{\pm}(p^{m_{i,j}}z),z_2)\right)|_{z_1=z_2e^{z_0}}\\
=\ & (e^{z_0}z_2-q^{-1}z_2)Y_W^{\phi}(Y_{\mathcal{E}}^{\phi}(y_{i}^{\pm}(z),z_0)y_{j}^{\pm}(p^{m_{i,j}}z),z_2)\\
=\ &(e^{z_0}z_2-q^{-1}z_2)Y_{\mathcal{E}}^{\phi}(y_{i}^{\pm}(z_2),z_0)y_{j}^{\pm}(p^{m_{i,j}}z_2).
 \end{align*}

 By multiplying $z_2^{-1}(e^{z_0}-q^{-1})^{-1}$ and taking $\textrm{Res}_{z_0}$, we get that
 \begin{align*}
 &Y_W^{\phi}\left(y_{i}^{\pm}(z)_{0}^{\phi}y_{j}^{\pm}(p^{m_{i,j}}z),z_2\right)\\
  =\ &\textrm{Res}_{z_0}z_2^{-1}(e^{z_0}-q^{-1})^{-1}y_{i,j}^{\pm}(z_2e^{z_0},p^{m_{i,j}}z_2)\\
  =\ &z_2^{-1}y_{i,j}^{\pm}(q^{-1}z_2,p^{m_{i,j}}z_2).
 \end{align*}
Furthermore, we have that
 \begin{align*}
&f_{ii}^{+}(z_1,q^{-1}z_2)f_{ij}^{+}(p^{m_{i,j}}z_1,p^{m_{i,j}}z_2)Y_W^{\phi}\left(y_{i,q}^{\pm}(z),z_1\right)Y_W^{\phi}\left(y_{i}^{\pm}(z)_{0}^{\phi}y_{j}^{\pm}(p^{m_{i,j}}z),z_2\right)\\
 =\ &y_{i,i,j}^{\pm}(z_1,q^{-1}z_2,p^{m_{i,j}}z_2)\in\mathcal{E}_{\hbar}^{(2)}(W).
 \end{align*}
 This implies that
 \begin{align*}
  &y_{i,i,j}^{\pm}(z_2e^{z_0},q^{-1}z_2,p^{m_{i,j}}z_2)\\
 =\ &\left((z_1-qz_2)Y_W^{\phi}(y_{i}^{\pm}(z),z_1)Y_W^{\phi}\left(y_{i}^{\pm}(z)_{0}^{\phi}y_{j}^{\pm}(p^{m_{i,j}}z),z_2\right)\right)|_{z_1=z_2e^{z_0}}\\
=\ &z_2(e^{z_0}-q)Y_W^{\phi}\left(Y_{\mathcal{E}}^{\phi}\left(y_{i,q}^{\pm}(z),z_0\right)\left(y_{i}^{\pm}(z)_{0}^{\phi}y_{j}^{\pm}(p^{m_{i,j}}z)\right),z_2\right)\\
=\ & z_2(e^{z_0}-q)Y_{\mathcal{E}}^{\phi}\left(y_{i}^{\pm}(z_2),z_0)\right)\left(y_{i}^{\pm}(z_2)_{0}^{\phi}y_{j}^{\pm}(p^{m_{i,j}}z_2)\right).
 \end{align*}
Then the assertion follows from  the following equality:
\begin{align*}
&\left(\left(y_{i}^{\pm}(z_2)\right)_{0}^{\phi}\right)^2 y_{j}^{\pm}(p^{-m_{ij}}z_2)\\
=\ &Y_W^{\phi}\left(\left(\left(y_{i}^{\pm}(z)\right)_{0}^{\phi}\right)^2 y_{j}^{\pm}(p^{-m_{ij}}z),z_2\right)\\
=\ &\mathrm{Res}_{z_0}z_2^{-1}(e^{z_0}-q)^{-1}y_{i,i,j}^{\pm}(z_2e^{z_0},q^{-1}z_2,p^{m_{ij}}z_2)\\
=\ &z_2^{-1}y_{i,i,j}^{\pm}(z_2q,q^{-1}z_2,p^{m_{ij}}z_2).
\end{align*}
Finally, the assertion (c) is obvious.
 \end{proof}

The following result gives an isomorphism between the category of restricted $\mathcal{E}_N$-modules of level $\ell$ and the category $\mathcal{R}_{\ell}^{\phi}$.
\begin{prpt}\label{pr:rescate}
Let $W$ be a restricted $\mathcal{E}_N$-module of level $\ell$. Then
\[\left(W,(\hat{\phi}_{i}(z),\hat{x}_{i}^{\pm}(z))_{0\leq i\leq N-1}\right)\] is an object of $\mathcal{R}_{\ell}^{\phi}$. On the other hand, let $\left(W,(\psi_{i}(z),y_{i}^{\pm}(z))_{0\leq i\leq N-1}\right)$ be an object of $\mathcal{R}_{\ell}^{\phi}$. Then $W$ is a restricted $\mathcal{E}_N$-module of level $\ell$ with the module action determined by
\begin{align*}
\phi_{i}^{\pm}(z)=\check{\psi}_{i}^{\pm}(z)\quad\textrm{and}\quad x_{i}^{\pm}(z)=\check{y}_{i}^{\pm}(z)\quad \textrm{for } 0\leq i\leq N-1.
\end{align*}
\end{prpt}
\begin{proof}
Let  $W$ be a restricted $\mathcal{E}_{N}$-module of level $\ell$. It  follows from Lemma \ref{le:cateresmod} that $(W,(\hat{\phi}_{i}(z),\hat{x}_{i}^{\pm}(z))_{0\leq i\leq N-1})\in \mathcal{M}_{\ell}^{\phi}$. This, together with Proposition \ref{pr:rescon} and Lemma \ref{le:Rphiserre}, shows that $(W,(\hat{\phi}_{i}(z),\hat{x}_{i}^{\pm}(z))_{0\leq i\leq N-1})\in \mathcal{R}_{\ell}^{\phi}$.

Conversely, let $\left(W,(\psi_{i}(z),y_{i}^{\pm}(z))_{0\leq i\leq N-1}\right)\in \mathcal{R}_{\ell}^{\phi}$. By Lemma \ref{le:Mphi}, we see that
$W$ is naturally a restricted $\mathcal{E}_N$-module of level $\ell$. This, together with  Proposition \ref{pr:rescon} and Lemma \ref{le:Rphiserre},  shows that
$W$ is a restricted $\mathcal{U}_{q,p}(\doublehat{\mathfrak{sl}}_{N})$-module of level $\ell$.
\end{proof}

Now we obtain the main result of this paper:

\begin{thm}\label{thm:main}
Let $W$ be a restricted $\mathcal{E}_N$-module of level $\ell$. Then $W$ is a $(\Z,\chi_{p^N})$-equivariant $\phi$-coordinated quasi $V_{\widehat{\mathfrak{sl}}_{\infty},\hbar}(\ell,0)$-module such that
\[Y_W^{\phi}(h_{i,\hbar},z)=\hat{\phi}_{\underline{i}}(p^{-i}z)\quad\textrm{and}\quad  Y_W^{\phi}(x_{i,\hbar}^{\pm},z)=\hat{x}_{\underline{i}}^{\pm}(p^{-i}z)\quad\textrm{for }i\in\Z.\]
On the other hand, let $(W,Y_W^{\phi})$ be a $(\Z,\chi_{p^N})$-equivariant $\phi$-coordinated quasi $V_{\widehat{\mathfrak{sl}}_{\infty},\hbar}(\ell,0)$-module. Then $W$ is a restricted $\mathcal{E}_N$-module of level $\ell$ such that
\[\hat{\phi}_{j}(z)=Y_W^{\phi}(h_{j,\hbar},p^jz)\quad\textrm{and}\quad
\hat{x}_{j}^{\pm}(z)=Y_W^{\phi}(x_{j,\hbar}^{\pm},p^jz)\quad\textrm{for }0\leq j\leq N-1.\]
\end{thm}
\begin{proof}
  The assertion follows immediately from Propositions \ref{pr:vglRphi} and \ref{pr:rescate}.
\end{proof}
% we only need to check that $(W,\hat{\phi}_{i,q}(z),\hat{x}_{i,q}^{\pm}(z))$ satisfies the relation (Q3$'$) and (Q5$'$). For the relation (Q3$'$), using relations (Q3) and (Q5), we have that
% \begin{align*}
% &\hat{x}_{i,q}^{+}(z_1)\hat{x}_{j,q}^{-}(z_2)-\iota_{z_2,z_1}g_{ij,q}(p^{m_{ij}}z_1/z_2)\hat{x}_{j,q}^{-}(z_2)\hat{x}_{i,q}^{+}(z_1)\\
% =\ &x_{i,q}^{+}(z_1)x_{j,q}^{-}(q^{-\ell}z_2)\phi_{j,q}^{+}(q^{-\frac{1}{2}\ell}z_2)^{-1}\\
% &-\iota_{z_2,z_1}g_{ij,q}(p^{m_{ij}}z_1/z_2)x_{j,q}^{-}(q^{-\ell}z_2)\phi_{j,q}^{+}(q^{-\frac{1}{2}\ell}z_2)^{-1}x_{i,q}^{+}(z_1)\\
% =\ &[x_{i,q}^{+}(z_1),x_{j,q}^{-}(q^{-\ell}z_2)]\phi_{j,q}^{+}(q^{-\frac{1}{2}\ell}z_2)^{-1}\\
% =\ &\frac{\delta_{i,j}}{q-q^{-1}}\left(\delta\left( \frac{z_2}{z_1}\right)-\phi_{i,q}^{-}(q^{-\frac{3}{2}{\ell}}z_2)\phi_{j,q}^{+}(q^{-\frac{1}{2}\ell}z_2)^{-1}\delta\left(\frac{q^{-2\ell}z_2}{z_1}\right)\right).
% \end{align*}


\begin{thebibliography}{AAGBP}
\bibitem[AKOS]{AKOS}
H. Awata, H. Kubo, S. Odake, and J. Shiraishi, {\em Quantum $\mathcal{W}_N$ algebras and Macdonald polynomials}, Commun. Math. Phys., {\bf 179} (1996), 401–416.
\bibitem[B]{B}
J. Beck, {\em Braid group action and quantum affine algebras}, Commun. Math. Phys., {\bf 165} (1994), 555-568.

\bibitem[BJK]{BJK-qva-BCD}
M.~Butorac, N.~Jing, and S.~Ko\v{z}i\'{c}, {\em $\hbar$-adic quantum vertex algebras associated with rational
  ${R}$-matrix in types ${B}$, ${C}$ and ${D}$},
 Lett. Math. Phys., {\bf 109} (2019), 2439-2471.

\bibitem[BS]{BS} I. Burban and O. Schiffmann, {\em On the Hall algebra of an elliptic curve, I,} Duke Math. J. {\bf 161} (2012), 1171–1231.

\bibitem[CJKT]{CJKT}
F. Chen, N. Jing, F. Kong, and S. Tan,  {\em Twisted quantum affinization and quantization of extended
affine Lie algebras}, Trans. Amer. Math. Soc., {\bf 376} (2023), 969-1039.
\bibitem[D]{D}
V. Drinfeld, {\em  A new realization of Yangians and quantized affine algebras}, Soviet Math.
Dokl., {\bf 36} (1988), 212-216.

 \bibitem[DF]{DF-qaff-RTT-Dr} J.~Ding and I.~Frenkel. {\em Isomorphism of two realizations of quantum affine algebra $\mathcal{U}_q(\widehat{\mathfrak{gl}(n)})$},
 Commun. Math. Phys., {\bf 156} (1993), 277-300.

\bibitem[DLM]{DLM}
 C. Dong, H. Li, and G. Mason, {\em Vertex Lie algebras, vertex poisson algebras and vertex algebras}, Contemp. Math., {\bf 297} (2002), 69–96.

\bibitem[EK]{EK}
P.~Etingof and D.~Kazhdan, {\em Quantization of {L}ie bialgebras, {P}art {V}: {Q}uantum vertex
  operator algebras},
Selecta Math., {\bf 6} (2000), 105.
% \bibitem[FF]{FF}
% B. Feigin and E. Frenkel, {\em Quantum $W\mathcal{W}$-algebras and elliptic algebras}, Comm. Math. Phys., {\bf 178} (1996), 653–678.
\bibitem[FFJMM]{FFJMM}
B. Feigin, E. Feigin, M. Jimbo, T. Miwa, and E. Mukhin, {\em Quantum continuous $\mathfrak{gl}_{\infty}$: Tensor
products of Fock modules and $\mathcal{W}_n$-characters}, Kyoto J. Math., {\bf 51} (2011), 365–392.
\bibitem[FJM]{FJM}
B. Feigin, M. Jimbo, and E. Mukhin, {\em The $(\mathfrak{gl}_n,\mathfrak{gl}_m)$ duality in the quantum toroidal setting}, Commun. Math. Phys., {\bf 367} (2019), 455-481.
\bibitem[FJMM1]{FJMM1}
B. Feigin, M. Jimbo, T. Miwa, and E. Mukhin, {\em Quantum continuous $\mathfrak{gl}_{\infty}$: Semi-infinite construction of representations, Kyoto J. Math.}, {\bf 51} (2011), 337–364.
\bibitem[FJMM2]{FJMM2}
B. Feigin, M. Jimbo, T. Miwa, and E. Mukhin, {\em Quantum toroidal $\mathfrak{gl}_1$-algebra: Plane partitions}, Kyoto J. Math., {\bf 52} (2012), 621–659.
\bibitem[FJMM3]{FJMM3}
B. Feigin, M. Jimbo, T. Miwa, and E. Mukhin, {\em Representations of quantum toroidal $\mathfrak{gl}_n$}, J.
Algebra, {\bf 380} (2013), 78–108.

\bibitem[FJMM4]{FJMM4}
B. Feigin, M. Jimbo, T. Miwa, and E. Mukhin, {\em Branching rules for quantum toroidal $\mathfrak{gl}_n$},
Adv. Math., {\bf 300} (2016), 229–274.


\bibitem[FJW]{FJW}
 I. Frenkel, N. Jing, and W. Wang, {\em Quantum vertex representations via finite groups and the
McKay correspondence}, Commun. Math. Phys., {\bf 211} (2000), 365–393.

\bibitem[FT]{FT} B. Feigin and A. Tsymbaliuk, {\em Equivariant K-theory of Hilbert schemes via shuffle algebra,} Kyoto J. Math., {\bf 5} (2011), 831–854.

\bibitem[GJ]{GJ}
Y. Gao and N. Jing, {\em $U_q(\widehat{\mathfrak{gl}}_N)$ action on $\widehat{\mathfrak{gl}}_N$-module and quantum toroidal algebra}, J. Algebra, {\bf 273} (2004), 320-343.
\bibitem[GKV]{GKV}
V. Ginzburg, M. Kapranov, and E.Vasserot, {\em Langlands reciprocity for algebraic surfaces}, Math. Res. Lett., {\bf 2} (1995), 147-160.
% \bibitem[H1]{H1}
%  D. Hernandez, {\em Representations of quantum affinizations and fusion product}, Transform.
% Groups, {\bf 10} (2005), 163–200.
% \bibitem[H2]{H2}
%  D. Hernandez, {\em Drinfeld coproduct, quantum fusion tensor category and applications}, Proc.
% Lond. Math. Soc., {\bf 95} (2007), 567–608.
% \bibitem[HMNW]{HMNW}
% K. Harada, Y. Matsuo, G. Noshita, and A. Watanabe, {\em $q$-Deformation of corner vertex operator algebras by Miura transformation,} J. High Energy Phys., {\bf 04} (2021),  202.
\bibitem[JKLT1]{JKLT1}
 N. Jing, F. Kong, H. Li, and S. Tan, {\em $(G,\chi_{\phi})$-equivariant $\phi$-coordinated quasi modules for nonlocal vertex algbebras}, J. Algebra, {\bf 570} (2021), 24-74.

\bibitem[JKLT2]{JKLT-Defom-va}
N.~Jing, F.~Kong, H.~Li, and S.~Tan, {\em Deforming vertex algebras by vertex bialgebras},
 Comm. Cont. Math., {\bf 26} (2024), 2250067.

\bibitem[JKLT3]{JKLT2}
 N. Jing, F. Kong, H. Li, and S. Tan, {\em Twisted quantum affine algebras and equivariant $\phi$-coordinated
modules for quantum vertex algebras}. arXiv:2212.01895.
% \bibitem[JZ]{JZ}
%  N. Jing and  H. Zhang, {\em On Hopf algebraic structures of quantum toroidal algebras},  Comm. Algebra, {\bf 51} (2023), 1135–1157.

\bibitem[JLM1]{JLM-qaff-RTT-Dr-BD}
N.~Jing, M.~Liu, and A.~Molev,
{\em Isomorphism between the ${R}$-matrix and {D}rinfeld presentations of
  quantum affine algebra: {T}ype ${B}$ and ${D}$},
 SIGMA, {\bf 16} (2020), 043.

\bibitem[JLM2]{JLM-qaff-RTT-Dr-C}
N.~Jing, M.~Liu, and A.~Molev, {\em Isomorphism between the ${R}$-matrix and {D}rinfeld presentations of
  quantum affine algebra: {T}ype ${C}$},
J. Math. Phys., {\bf 61} (2020), 031701.

\bibitem[Kac]{Kac}
V. Kac, Vertex Algebras for Beginners,  American Mathematical Soc., Vol. 10, Providence, 1998.

\bibitem[Kas]{Kas}
C. Kassel,  Quantum groups,  Graduate Texts in Mathematics, Springer-Verlag, Vol. 155, New York, 1995.
 % \bibitem[Koj1]{Koj1}
 % T. Kojima, {\em Quadratic relations of the deformed $W$-superalgebra $\mathcal{W}_{q,t}(\mathfrak{sl}(2|1))$}, J. Math. Phys., {\bf 62} (2021), 051702.
% \bibitem[Koj2]{Koj2}
% T. Kojima, {\em Quadratic relations of the deformed $W$-superalgebra $\mathcal{W}_{q,t}(A(M,N))$}, J. Phys. A, {\bf 54} (2021), 335201.
 \bibitem[Kon]{K}
 F. Kong, {\em Quantum Affine Vertex Algebras Associated to Untwisted
 Quantum Affinization Algebras}, Commun. Math. Phys., {\bf 402} (2023), 2577-2625.


\bibitem[Ko\v{z}1]{Kozic-qva-tri-A}
S.~Ko\v{z}i\'{c}, {\em On the quantum affine vertex algebra associated with trigonometric
  ${R}$-matrix}, Selecta Math. (N. S.), (2021), 27-45.

\bibitem[Ko\v{z}2]{K-qva-phi-mod-BCD}
S.~Ko\v{z}i\'{c},
{\em $\hbar$-adic quantum vertex algebras in types ${B}$, ${C}$, ${D}$ and
  their $\phi$-coordinated modules}, J. Phys. A: Math. Theor., {\bf 54} (2021), 485202.
\bibitem[Li1]{Li1}
H. Li, {\em Axiomatic $G_1$-vertex algebras}, Comm. Cont. Math., {\bf 5} (2003), 1-47.
\bibitem[Li2]{Li2}
H. Li, {\em Nonlocal vertex algebras generated by formal vertex operators}, Sel. Math. New Ser., {\bf 11} (2005), 349-397.
\bibitem[Li3]{Li3}
H. Li, {\em A new construction of vertex algebras and qasi-modules for vertex algberas}, Adv. Math., {\bf 202} (2006), 232-286.
\bibitem[Li4]{Li4}
H. Li, {\em $\hbar$-adic quantum vertex algebras and their modules}, Commun. Math. Phys., {\bf 296} (2010), 475-523.
\bibitem[Li5]{Li5}
H. Li, {\em $\phi$-coordinated quasi modules for quantum vertex algebras}, Commun. Math. Phys., {\bf 308} (2011), 703-741.

\bibitem[Li6]{Li6}
H. Li, {\em $G$-equivariant $\phi$-coordinated quasi modules for quantum vertex algebras}, J. Math. Phys., {\bf 54} (2013), 051704.
\bibitem[LL]{LL}
 J. Lepowsky and H. Li, Introduction to Vertex Operator Algebras and Their Representations, volume 227, Birkhäuser Boston Incoporation, 2004.

\bibitem[M1]{M1}
K. Miki, {\em Toroidal braid group action and an automorphism of toroidal algebra
$U_q(\mathfrak{sl}_{n+1,tor})(n \geq 2)$},  Lett. Math. Phys., {\bf 47} (1999), 365-378.
\bibitem[M2]{M2}
 K. Miki, {\em Representations of quantum toroidal algebra $U_q(\mathfrak{sl}_{n+1,tor})(n\geq 2)$}, J. Math.
Phys., {\bf 41} (2000), 7079-7098.
% \bibitem[M3]{M3}
% K. Miki, {\em Some quotient algebras arising from the quantum toroidal algebra
% $U_q(\mathfrak{sl}_{n+1}(C_{\gamma}))(n \geq 2)$},  Osaka J. Math., {\bf 43} (2006), 895-922.
\bibitem[M3]{M3}
K. Miki, {\em A $(q,\gamma)$-analog of the $\mathcal{W}_{1+\infty}$ algebra}, J. Math. Phys., {\bf 48} (2007), 3520.
\bibitem[MNNZ]{MNNZ} Y. Matsuo, S. Nawata, G. Noshita, and R. Zhu, {\em Quantum toroidal algebras and solvable structures in gauge/string theory}, Physics Reports, {\bf 1055} (2024), 1-144.
\bibitem[MRY]{MRY}
 R. Moody,  S.E. Rao, and  T. Yokonuma, {\em Toroidal Lie algebras and vertex representations.} Geom.
 Dedicata, {\bf 35} (1990), 283–307.
\bibitem[N]{N} A. Neguţ, {\em Quantum toroidal and shuffle algebras}, Adv. Math., {\bf 372}
(2020), 107288.
\bibitem[R]{R} M. Roitman, {\em On free conformal and vertex algebras}, J. Algebra, {\bf 217} (1999), 496–527.


\bibitem[RSTS]{RS-RTT}
Y.~Reshetikhin and A.~Semenov-{T}ian {S}hansky,
{\em Central extensions of quantum current groups},
Lett. Math. Phys., {\bf 19} (1990), 133-142.

\bibitem[S]{S}
Y. Saito, {\em Quantum toroidal algebras and their vertex representations}, Publ. Res. Inst.
Math. Sci., {\bf 34} (1998), 155-177.
\bibitem[SKAO]{SKAO}
J. Shiraishi, H. Kubo, H. Awata, and S. Odake, {\em A quantum deformation of the Virasoro algebra and the Macdonald symmetric functions}, Lett. Math.
Phys., {\bf 38} (1996), 33–51.
\bibitem[STU]{STU}
Y. Saito, K. Takemura, and D. Uglov, {\em Toroidal actions on level 1 modules of $U_q(\widehat{sl}_n)$},
Transform. Groups, {\bf 3} (1998), 75–102.
\bibitem[SV]{SV} O. Schiffmann, E. Vasserot, {\em The elliptic Hall algebra, Cherednik Hecke algebras and Macdonald polynomials,} Compos. Math. {\bf 147}  (2011),
188–234.
\bibitem[T]{T} A. Tsymbaliuk, Shuffle approach towards quantum affine and toroidal algebras,
SpringerBriefs in Mathematical Physics, Vol 49, Springer, 2023.
% \bibitem[T]{T}
% A. Tsymbaliuk, {\em  Quantum affine Gelfand-Tsetlin bases and quantum toroidal algebra via $K$-theory of affine Laumon spaces}, Selecta Math. (N.S.), {\bf 16} (2010), 173–200.
\bibitem[TU]{TU}
K. Takemura and D. Uglov, {\em Level-0 action of $\mathcal{U}_q(\widehat{\mathfrak{sl}_n})$ on the $q$-deformed Fock spaces}, Comm.
Math. Phys., {\bf 190} (1998), 549–583.
\bibitem[VV1]{VV1}
M. Varagnolo and E. Vasserot, {\em Schur duality in the toroidal setting},  Comm. Math.
Phys., {\bf 182} (1996), 469–483.
\bibitem[VV2]{VV2}
M. Varagnolo and E. Vasserot, {\em Double-loop algebras and the Fock space}, Invent. Math., {\bf 133} (1998), 133-159.


\end{thebibliography}
\end{document}